\documentclass[a4paper,12pt]{article}
\usepackage{amsmath,amsthm,amsfonts,amssymb,bbm}
\usepackage{graphicx,psfrag,subfigure,color,cite}
\usepackage[UKenglish]{babel}
\usepackage{textcmds}
\usepackage{mathtools}
\usepackage{float}
\usepackage{changes}
\usepackage[colorlinks]{hyperref}

\numberwithin{equation}{section}

\usepackage[margin=1.2in]{geometry}

\newcommand{\e}{\varepsilon}
\newcommand{\Pb}{\mathbb{P}}
\newcommand{\E}{\mathbb{E}}

\newcommand{\R}{\mathbb{R}}
\newcommand{\N}{\mathbb{N}}
\newcommand{\Z}{\mathbb{Z}}

\newcommand{\eps}{\varepsilon}
\newcommand{\Or}{{\cal O}}

\newtheorem{prop}{Proposition}[section]
\newtheorem{assumption}[prop]{Assumption}
\newtheorem{thm}[prop]{Theorem}
\newtheorem{lem}[prop]{Lemma}

\newtheorem{cor}[prop]{Corollary}

\newtheorem{cla}[prop]{Claim}

\newtheorem{rem}[prop]{Remark}
\newenvironment{remark}{\begin{rem}\normalfont}{\end{rem}}

\title{Tagged particle fluctuations for TASEP with dynamics restricted by a moving wall}
\author{Patrik L.\ Ferrari\thanks{Institute for Applied Mathematics, Bonn University, Endenicher Allee 60, 53115 Bonn, Germany. E-mail: {\tt ferrari@uni-bonn.de}} \and
Sabrina Gernholt\thanks{Institute for Applied Mathematics, Bonn University, Endenicher Allee 60, 53115 Bonn, Germany. E-mail: {\tt sgernhol@uni-bonn.de}}
}

\date{30th August 2025}

\begin{document}
\maketitle
\sloppy

\begin{abstract}
We consider the totally asymmetric simple exclusion process on $\Z$ with step initial condition and with the presence of a rightward-moving wall that prevents the particles from jumping. This model was first studied in~\cite{BBF21}. We extend their work by determining the limiting distribution of a tagged particle in the case where the wall has influence on its fluctuations in neighbourhoods of multiple macroscopic times.
\end{abstract}

\section{Introduction}
The totally asymmetric simple exclusion process (TASEP) is one of the most studied interacting particle systems in one spatial dimension. It was introduced into mathematics by Spitzer in~\cite{Spi70}. The process consists of particles on $\Z$, with each particle moving to the right by one step after a waiting time described by an exponential clock with mean one. This clock starts counting from the moment the right-neighbouring site becomes empty, and all clocks are independent random variables. The model can also be viewed as an evolution of an interface and, as such, belongs to the Kardar-Parisi-Zhang universality class~\cite{KPZ86}.

Although being a simple model, due to the interaction between particles, the large time fluctuations of particle positions or integrated current are very non-trivial. In this paper, we consider the so-called step initial condition, that is, initially sites $\ldots,-3,-2,-1,0$ are occupied by particles while $1,2,\ldots$ are empty. Moreover, the jumps of the rightmost particle are suppressed whenever its position is to the right of a given increasing function $t\mapsto f(t)$. We can think of it as a barrier or a wall, which moves to the right and forces the system of particles to stay behind it. This setting was first studied by Borodin-Bufetov-Ferrari in~\cite{BBF21}.

If $f(t)\gg t$, the effect of the wall is negligible, since the rightmost particle without the wall moves with unit speed. In this case, a tagged particle in the bulk of the system exhibits asymptotic fluctuations identical to those in the system without the wall, namely fluctuations distributed according to the GUE Tracy-Widom distribution (Johansson in~\cite{Jo00b}). On the other hand, if $f(t)\ll t$, for instance if $f(t)=v t$ with $v\in (0,1)$, the influence of the wall is relevant: in the framework of \cite{BBF21}, the particle density has a rarefaction fan which goes over into a constant density profile until reaching the position of the wall. From the KPZ theory, one expects to still see the GUE Tracy-Widom law on the rarefaction fan, while in the constant density region one anticipates the GOE Tracy-Widom distribution as obtained for TASEP with constant non-random density as initial condition (see Sasamoto~\cite{Sas05}, Borodin-Ferrari-Pr\"ahofer~\cite{BFP06} and Ferrari-Occelli~\cite{FO17} for generic density). In the transition region, the particle fluctuations should be those of the Airy$_{2\to 1}$ process (see Borodin-Ferrari-Sasamoto~\cite{BFS07}). These predictions actually come true as proven in Proposition~4.8 of~\cite{BBF21}.

More generally, in~\cite{BBF21}, a non-decreasing, not necessarily linear, function $f(t)$ is considered, such that the influence of the wall on the fluctuations of a tagged particle is restricted to one macroscopic time region. As a result, the limiting distribution of the particle is determined in the form of a variational formula (see Theorem~4.4 of~\cite{BBF21}),
that also gives the one-point distribution of the KPZ fixed point for a generic initial condition. This was achieved by the solution of TASEP in the KPZ fixed point paper by Matetski-Quastel-Remenik~\cite{MQR17} (see also~\cite{Jo03,QR13,QR16,BL13,CLW16,CFS16,FO18} for other instances of such variational formulas).

In this paper, we consider the case where the influence of the wall is not restricted to a single macroscopic time region but occurs over multiple times. In our main result, Theorem~\ref{thm_main_result}, we prove that when the influence of the wall occurs in two distinct macroscopic time regions, then the asymptotic fluctuations of the tagged particle are distributed according to the product of two distribution functions, each given by a variational formula. The extension to multiple times is straightforward as mentioned in Remark~\ref{remark_several_time_regions_or_removed_wall}.

Heuristically, the product structure of the limit distribution can be understood as follows. For a tagged particle to lie to the right of a given position $s$ at time $t$, all particles to its right must reach positions beyond $s$. The presence of a wall influence means that, within a time interval, these particles experience significant slowdowns, which lowers the probability of reaching the desired positions. Therefore, in the scaling limit, the distribution of the tagged particle is shaped by the dynamics of the process in the regions of wall influences. For different macroscopic times, these dynamics are asymptotically independent, leading to the observed decoupling.
Thus the product form of the limiting distribution is due to the fact that the fluctuation of the (rescaled) tagged particle is asymptotically the maximum of two independent random variables. In the context of TASEP, this form of distribution occurred previously in the presence of shocks in a series of papers by Ferrari and Nejjar~\cite{FN13, FN16, N17, Fer18, FN19,FN24,Nej21}, see also Quastel-Remenik~\cite{QR18} and Bufetov-Ferrari~\cite{BF22}.

In our situation, the product form stems from the circumstance that the values of the auxiliary tagged particle process without wall in the starting formula (see Proposition~\ref{prop_3.1_of_BBF22}) are asymptotically independent for times at macroscopic distance. This follows from the fact that the fluctuations coming from mesoscopic times close to time zero are asymptotically irrelevant. The study of shocks relies on related physical ideas. In that case, the fluctuations acquired in any mesoscopic time close to the end time are asymptotically irrelevant.

In this paper, we use several methods. The starting formula (Proposition~\ref{prop_3.1_of_BBF22}) expressing the distribution of a tagged particle with wall constraint as a function of a tagged particle process without wall was proven in Proposition~3.1 of~\cite{BBF21} using colour-position symmetry from Borodin-Bufetov~\cite{BB19} (see also~\cite{AAV11,AHR09,Buf20,Gal20} for related works). A second key ingredient is functional slow decorrelation (see~\cite{CLW16} for the last passage percolation (LPP) setting), for which we present a proof based on a criterion by~\cite{Bil68}. As input we need tightness of the particle process, which is derived using comparison inequalities (in the spirit of the Cator-Pimentel approach for LPP~\cite{CP15b}). One problem is that these inequalities are not satisfied under the basic coupling. For that reason, we introduce a new coupling called ``clock coupling'' under which they hold, see Section~\ref{sectCoupling}. Finally, we need to prove some localization results for the backwards paths introduced in~\cite{Fer18}. They are utilized to obtain independence of tagged particle positions from the fact that they effectively depend on the randomness in distinct neighbourhoods of backwards paths, see Corollary~\ref{CorIndependence} for an explicit statement. Our approach for the localization of backwards paths substantially differs from other approaches in the particle representation \cite{Fer18,FN19}. We employ the strategy going back to Basu-Sidoravicius-Sly~\cite{BSS14} (formulated for LPP) of using mid-time estimates and iterating (see \cite{BF22} in the context of TASEP height functions).
However, applying this strategy using the particle representation of TASEP requires extra arguments (Lemma~\ref{lemma_prob_1_ordering_backwards_paths} and Lemma~\ref{lemma_alwaysstepparticle}), and the approach only provides control of fluctuations to the right. Instead of controlling the fluctuations to the left directly, we employ a particle-hole duality, see Section~\ref{sect4.3}. Such localization issues were not present for the LPP and in the TASEP height function representation~\cite{BSS14,BF22,BF20}, where the treatment of right and left fluctuations was the same. Still, with the starting formula being given in terms of tagged particles, we are able to stay in this setting by the mentioned additional arguments, without employing extra mappings to other representations of the model. Finally, we wish to emphasize that our localization strategy relies on both basic coupling results and comparison inequalities under clock coupling.

\paragraph*{Outline} In Section~\ref{sectResults} we introduce the model and state the main results, with two explicit examples. In Section~\ref{sectCoupling} we discuss properties of the backwards paths, the new coupling needed for comparison inequalities and the main asymptotic results obtained from them. Section~\ref{sectLocal} contains the localization results and is followed by the functional slow decorrelation result in Section~\ref{SectSlowDec}. Finally, in Section~\ref{section_proof_of_the_main_results} we prove the main results. We collect some well-known results on TASEP asymptotics and bounds in the Appendix.

\paragraph{Acknowledgements:} The work was partly funded by the Deutsche Forschungsgemeinschaft (DFG, German Research Foundation) under Germany’s Excellence Strategy - GZ 2047/1, projekt-id 390685813 and by the Deutsche Forschungsgemeinschaft (DFG, German Research Foundation) - Projektnummer 211504053 - SFB 1060.

\newpage
\section{Model and main results}\label{sectResults}
We consider the totally asymmetric simple exclusion process (TASEP) on $\Z$. It is an interacting particle system that maintains the order of particles. We denote the position of the particle with label $n$ at time $t$ by $x_n(t)\in\Z$ and use the right-to-left order, namely $x_{n+1}(t)<x_n(t)$ for all $n$ and $t$. Particles try to jump to their right-neighbouring site at rate $1$, provided that the arrival site is empty.

In this paper, we consider TASEP with step initial condition, that is, $x_n(0)=-n+1$ for $n=1,2,\ldots$. Furthermore, we impose the constraint that all jump trials to the right of a deterministic barrier $t\mapsto f(t)$ are suppressed. We assume that $f$ is a non-decreasing function with $f(0)\geq 0$. In other words, the barrier acts as a moving wall blocking the particles to its left. We denote by $x^f(t)$ the TASEP with this barrier condition and, by construction, we have $x_1^f(t)\leq f(t)$ for all times $t\geq 0$. By $x(t)$ we denote another TASEP with step initial condition but without wall constraint.

The starting point is a finite time result relating the distribution of $x_n^f(T)$, for $T \geq 0$, to the process $(x_n(t),t \in [0,T])$ without wall constraint. It is derived by coupling both processes to multi-species TASEPs and applying the colour-position symmetry from~\cite{BB19} (see also~\cite{AAV11,AHR09,Buf20,Gal20} for related works).
\begin{prop}[Proposition 3.1 of \cite{BBF21}] \label{prop_3.1_of_BBF22}
	Let $f$ be a non-decreasing function on $\R_{\geq 0}$ with $f(0) \geq 0$.
	 Then, for any $n \in \N$ and $S \in \R$, it holds
	\begin{equation}\label{eq1.2}
	\Pb( x_n^f(T) > S) = \Pb( x_n(t) > S - f(T-t) \textrm{ for all } t \in [0,T], x_n(T) > S).
	\end{equation}
\end{prop}

In~\cite{BBF21}, the case where $x_n(t) > S - f(T-t)$ is non-trivial only for times around one given macroscopic time was analysed. In this paper, we treat the case where there are two time windows in $[0,T]$ where the inequality is non-trivial (the generalization for finitely many times is straightforward). More precisely, we consider the tagged particle with label $n=\alpha T$ at time $T\gg 1$, where $\alpha \in (0,1)$ is fixed. Further, we assume that there are two neighbourhoods of times $\alpha_0 T $ and $\alpha_1 T$ where the influence of the wall is relevant, with\footnote{We can also include $\alpha_1=1$ with the only difference that the supremum in the second term in Theorem~\ref{thm_main_result} is restricted for example to $\R_+$, compare for instance to Theorem~4.7 of~\cite{BBF21}.} $\alpha < \alpha_0 < \alpha_1 < 1$. The lower bound $\alpha T$ is due to the fact that until time $\alpha T$, the $\alpha T$-th particle did not move yet (macroscopically).

This motivates considering two rescalings of the tagged particle process $x_{\alpha T}(t)$ around the times $\alpha_0 T$ and $\alpha_1 T$. For $i \in \{0,1\}$, let
\begin{equation}
\tilde{X}_T^{i}(\tau) := \frac{ x_{\alpha T}(\alpha_i T - \tilde{c}_2^{i} \tau T^{2/3}) - \tilde{\mu}^{i}(\tau,T)}{- \tilde{c}_1^{i} T^{1/3}}, \label{eq_definition_process_tildeX_T}
\end{equation}
where
\begin{equation}
\tilde{c}_1^{i} := \frac{(\sqrt{\alpha_i} - \sqrt{\alpha})^{2/3} \alpha_i^{1/6}}{\alpha^{1/6}}, \ \tilde{c}_2^{i} := \frac{2(\sqrt{\alpha_i} - \sqrt{\alpha})^{1/3} \alpha_i^{5/6}}{\alpha^{1/3}}
\end{equation}
and
\begin{equation}
\tilde{\mu}^{i}(\tau,T) := \sqrt{\alpha_i} ( \sqrt{\alpha_i} - 2 \sqrt{\alpha})T - 2 \tau \frac{(\sqrt{\alpha_i} - \sqrt{\alpha})^{4/3} \alpha_i^{1/3}}{\alpha^{1/3}} T^{2/3}.
\end{equation}
In~\cite{BBF21}, weak convergence of $(\tilde{X}_T^{i}(\tau))$ is stated as follows\footnote{This result can also be obtained by taking the one for last passage percolation, for which tightness follows from the comparison inequalities of~\cite{CP15b}, and then using functional slow decorrelation (see~\cite{CFS16} for the exponential LPP setting, previously~\cite{CLW16} for geometric LPP).}.
\begin{lem}[Corollary 4.1 of \cite{BBF21}] \label{corollary_4.1_of_BBF22_tightness_weak_conv_Xtilde}
	For $T \to \infty$, the process $(\tilde{X}_T^{i}(\tau))$ converges weakly in the space of c\`adl\`ag functions on compact intervals to $(\mathcal{A}_2(\tau) - \tau^2)$, where $\mathcal{A}_2$ denotes an Airy$_2$ process.
\end{lem}
This weak convergence and the resulting tightness of $(\tilde{X}_T^{i}(\tau))$ play a key role in the proof of the main Theorem~\ref{thm_main_result} (see \cite{Bil68} for definitions and properties). We equip the space of c\`adl\`ag functions on compact intervals with the Skorokhod ($J_1$) topology. Considering a fixed interval $[a,b] \subseteq \R$, we denote this space by $\mathbb{D}([a,b])$.

The proof of Lemma~\ref{corollary_4.1_of_BBF22_tightness_weak_conv_Xtilde} used a comparison inequality, Proposition~2.2 of~\cite{BBF21}, which unfortunately contains an inaccuracy (the statement is true only in law, due to the use of the basic coupling). This can be fixed with help of the new coupling we introduce in Section~\ref{sectCoupling}; it allows us to prove the comparison inequality in Lemma~\ref{lemma_comparison_particle_increments}. Since the latter implies the needed asymptotic comparison to stationary TASEP in a uniform sense, as stated in Proposition \ref{prop_comparison_increments_of_particles}, it can be used to amend the proof of the weak convergence, Lemma~\ref{corollary_4.1_of_BBF22_tightness_weak_conv_Xtilde}, directly in the particle positions representation, without passing to LPP models.

In order to observe a non-trivial wall influence during the time interval $[0,T]$ and still see some movement of the particle, we need a scaling
\begin{equation}
x_{\alpha T}^f (T) \simeq \xi T \textrm{ for some } \xi \in (-\alpha, 1-2\sqrt{\alpha}).
\end{equation}
Later on, we suppose $\xi$ to be arbitrary but fixed. We impose the following conditions on the barrier $f$.
\begin{assumption} \label{assumption_for_main_result}
	Let $f$ be a non-decreasing c\`adl\`ag function on $\R_{ \geq 0}$ with $f(0) \geq 0$. We require:
	\begin{itemize}
		\item[(a)] For some fixed $\eps\in (0, \alpha_0 - \alpha)$
		and for $t \in [0,T]$ satisfying $| t- \alpha_0 T | > \varepsilon T$ and $| t - \alpha_1 T | > \varepsilon T$, it holds
\begin{equation}
f(T-t) \geq T f_0 ((T-t)/T) + K(\varepsilon)T,
\end{equation}
where $K(\e ) $ is a positive constant and the function $f_0 : [0,1] \to \R $ is defined by
\begin{equation}
f_0(\beta ) := \begin{cases} \xi - \sqrt{1-\beta}(\sqrt{1-\beta} - 2 \sqrt{\alpha}), \ &\beta \in [0,1-\alpha), \\ \xi + \alpha, \ &\beta \in [1-\alpha,1]. \end{cases} \end{equation}
		\item[(b)] For $i \in \{0,1\}$, parametrize $T-t = (1-\alpha_i) T + \tilde{c}_2^{i} \tau T^{2/3}$ and let
\begin{equation}
f(T-t) = \xi T - \tilde{\mu}^{i} (\tau,T) - \tilde{c}_1^{i} (\tau^2 - g_T^{i}(\tau))T^{1/3}, \ \tau \in \R.
\end{equation}
		The sequences $(g_T^0), (g_T^1)$ converge uniformly on compact sets to piecewise continuous and c\`adl\`ag functions $g_0$ and $g_1$ respectively. Further, there exists a constant $M \in \R $ such that for all $T$ large enough and $i \in \{0,1\}$, it holds
\begin{equation}
g_T^{i}(\tau) \geq - M + \tfrac12 \tau^2 \textrm{ for } | \tau | \leq \e (\tilde{c}_2^{i})^{-1} T^{1/3}.
\end{equation}
\end{itemize}
\end{assumption}
\begin{remark} \label{remark_weaken_assumption}
	Our arguments are still valid if the following reductions are made from our assumptions: \begin{itemize}
		\item Instead of $\alpha < \alpha_0 < \alpha_1$ fixed (independent of $T$), we can assume $\alpha < \alpha_0$ fixed and for some $\sigma>0$, \mbox{$\alpha_1 T - \alpha_0 T \geq T^{2/3+\sigma}$} for all $T$ large enough.
		\item In Assumption~\ref{assumption_for_main_result} (a), it suffices to assume $f(T-t) \geq T f_0((T-t)/T) + K(\eps)T$ for $t \in [\alpha T, (\alpha+\delta)T]$ for some small $\delta > 0$, and to replace $K(\eps)T$ by some term growing faster than $ T^{1/3}$ in the other regions.
	\end{itemize}
\end{remark}
By $g_i$ being piecewise continuous, we mean that each bounded interval in $\R$ can be decomposed into finitely many subintervals on which $g_i$ is continuous. Since $g_i$ is c\`adl\`ag, these intervals are closed at their left edge point.

Consider \eqref{eq1.2} under Assumption~\ref{assumption_for_main_result}, then the scaling $x^f_{\alpha T}(T) \simeq \xi T$ is consistent with the law of large numbers for $x_{\alpha T}(t)$, see Lemma~\ref{Lemma_A.2}. Assumption~\ref{assumption_for_main_result} (a) ensures that a macroscopically visible wall influence on the tagged particle can only happen in neighbourhoods of the times $\alpha_0 T$ and $\alpha_1 T$: away from them, the right hand side in the second probability in \eqref{eq1.2} is macroscopically smaller than the law of large numbers of $x_{\alpha T}(t)$. The sequences $(g_T^i), i \in \{0,1\}$, can be interpreted as $\Or(T^{1/3})-$fluctuations of the rescaled wall position around the law of large numbers near to the times $\alpha_i T, i \in \{0,1\}$.

Our main result is the following theorem.
\begin{thm} \label{thm_main_result}
	For $f$ satisfying Assumption~\ref{assumption_for_main_result} and for each $S \in \R$, it holds
	\begin{equation}
\begin{aligned}
	&\lim_{T \to \infty} \Pb(x_{\alpha T}^f(T) \geq \xi T - S T^{1/3}) \\
	& =  \Pb\Bigl (\ \sup_{ \tau \in \R} \{ \mathcal{A}_2^0 ( \tau) - g_0(\tau) \} \leq S (\tilde{c}_1^0)^{-1} \Bigr) \Pb\Bigl(\ \sup_{ \tau \in \R} \{ \mathcal{A}_2^1 ( \tau) - g_1(\tau) \} \leq S(\tilde{c}_1^1)^{-1} \Bigr), \label{limit_distr_main_result}
\end{aligned}
	\end{equation}
	where $\mathcal{A}_2^0$ and $\mathcal{A}_2^1$ are two independent Airy$_2$ processes.
\end{thm}
\begin{remark} \label{remark_several_time_regions_or_removed_wall}
The proof of Theorem~\ref{thm_main_result} can directly be extended to the case of $n+1$ time windows of possible wall influences, for any $n \in \N$. In this case, one fixes $0 < \alpha < \alpha_0 < \dots < \alpha_n < 1$ and modifies Assumption~\ref{assumption_for_main_result} as follows: the inequality in (a) needs to hold for $t \in [0,T]$ with $|t - \alpha_i T | > \eps T$ for $i = 0, \dots, n $. In part (b), one considers $n+1$ parametrizations of $f(T-t)$ and obtains sequences $(g_T^{i}), i \in \{0, \dots, n \}$ with the same properties as before. The limit distribution becomes
\begin{equation}
\lim_{T \to \infty} \Pb( x_{\alpha T}^f(T) \geq \xi T - S T^{1/3}) = \prod_{i=0}^n \Pb\Bigl( \sup_{ \tau \in \R} \{ \mathcal{A}_2^{i}(\tau) - g_i(\tau)\} \leq S (\tilde{c}_1^{i})^{-1}\Bigr)
\end{equation}
for Airy$_2$ processes $\mathcal{A}_2^{0}, \dots, \mathcal{A}_2^{n}$ and $S \in \R$.

Secondly, it is possible to remove the wall constraint in some regions by allowing the limit functions $g_i$ to equal infinity in some time interval. For $I_i \subseteq \R$ denoting the region where $g_i \equiv \infty$, the limit distribution is then given by
\begin{equation}
\Pb\Bigl( \sup_{ \tau \in \R \setminus I_0 } \{ \mathcal{A}_2^0(\tau) - g_0(\tau)\} \leq S (\tilde{c}_1^{0})^{-1}\Bigr) \Pb\Bigl( \sup_{ \tau \in \R \setminus I_1 } \{ \mathcal{A}_2^1(\tau) - g_1(\tau)\} \leq S (\tilde{c}_1^{1})^{-1}\Bigr).
\end{equation}
\end{remark}

\paragraph*{Two examples} Finally, we examine two concrete examples.
\begin{itemize}
\item[(1)] Consider the piecewise linear function $f$, see Figure~\ref{figure_example_wall_fct}, given by
	\begin{equation}\label{eq1.13}
	f(t) = \begin{cases} c_1 T + v_1 t & \textrm{ if } t \in [0, \tfrac{2-\alpha_0-\alpha_1}{2} T), \\
	c_0 T + v_0 t & \textrm{ if } t \in [ \tfrac{2-\alpha_0-\alpha_1}{2} T, T], \end{cases}
	\end{equation}
	where we set $c_i = \xi - \Bigl ( 1 - \sqrt{ \tfrac{\alpha}{\alpha_i}} - \sqrt{\alpha \alpha_i} \Bigr )$ and $v_i =  1 - \sqrt{ \tfrac{\alpha}{\alpha_i}}  \in (0,1)$, for $i \in \{0,1\}$. We fix $\xi \in (-\alpha, 1-2\sqrt{\alpha})$ such that $c_i \geq 0, i \in \{0,1\}$. The function $f$ jumps to a higher value at the point $\tfrac{2-\alpha_0-\alpha_1}{2} T$. Consequently, $f$ takes non-negative values, is c\`adl\`ag and non-decreasing.

\begin{figure}[t!]
	\centering
	\includegraphics[scale=0.8]{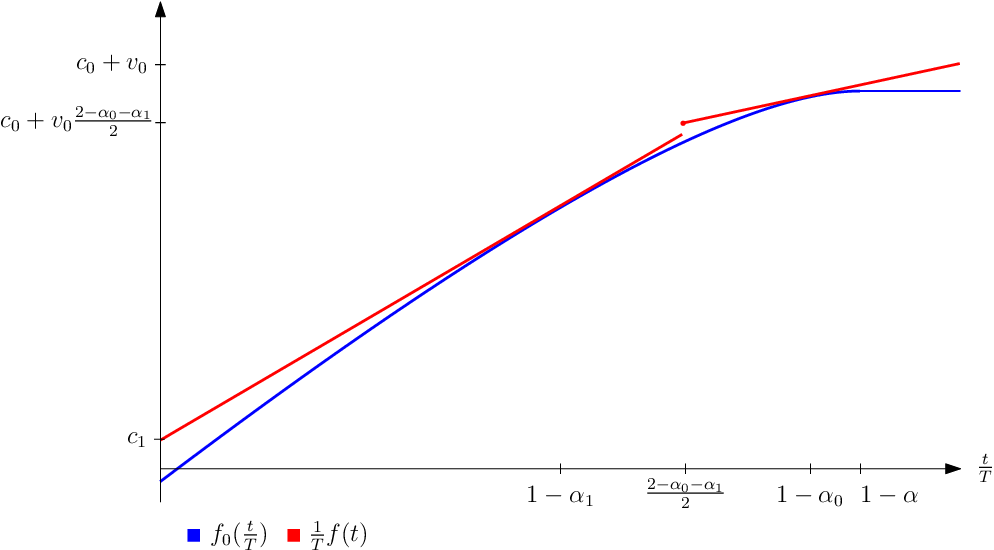}
	\caption{The function $f$ in the first example, as defined in~\eqref{eq1.13}.}
	\label{figure_example_wall_fct}
\end{figure}
A computation shows that $f$ fulfils Assumption~\ref{assumption_for_main_result} with $g_i(\tau) = \tau^2, i \in \{0,1\}$.
	Applying Theorem~\ref{thm_main_result} and recalling the variational formula by Johansson, Corollary 1.3 of~\cite{Jo03b}, we establish
	\begin{equation}
\begin{aligned}
	& \lim_{ T \to \infty} \Pb( x^f_{\alpha T}(T) \geq \xi T - S T^{1/3}) \\
	&=  \Pb \Bigl ( \sup_{\tau \in \R} \{\mathcal{A}_2^{0} (\tau) - \tau^2\} \leq S (\tilde{c}_1^{0})^{-1} \Bigr) \Pb \Bigl( \sup_{\tau\in\R} \{\mathcal{A}_2^{1} (\tau) - \tau^2\} \leq S (\tilde{c}_1^{1})^{-1} \Bigr) \\
	&=  F_{1}( 2^{2/3} S (\tilde{c}_1^{0})^{-1}) F_{1}( 2^{2/3} S (\tilde{c}_1^{1})^{-1}).
\end{aligned}
	\end{equation}
Here, $F_{1}$ denotes the GOE Tracy-Widom distribution function.

\item[(2)] Consider the function
 \begin{equation}
	f(t) = \begin{cases} c_1 T + v_1 t & \textrm{ if } t \in [0, \tfrac{2-\alpha_0-\alpha_1}{2} T), \\
	c_0 T + v_0 t & \textrm{ if } t \in [ \tfrac{2-\alpha_0-\alpha_1}{2} T, (1-\alpha_0)T], \\ \infty & \textrm{ if } t \in ((1-\alpha_0)T, T]. \end{cases}
	\end{equation}
	Then, we have\footnote{For $t > (1-\alpha_0)T$, the wall constraint is removed because $f(t) = \infty$. Therefore, the convergence in Assumption~\ref{assumption_for_main_result} (b) for $i=0$ is restricted to $\tau \in \R_-$.}  $
	g_1(\tau) = \tau^2$ and $
	g_0(\tau) = \begin{cases} \tau^2, \ &\tau \in \R_-, \\ \infty, \ &\tau \in \R_+. \end{cases}
	$
	
	As pointed out in Remark~\ref{remark_several_time_regions_or_removed_wall}, the limit distribution becomes
\begin{equation}
\begin{aligned}
	 & \Pb\Bigl( \sup_{ \tau \in \R_-} \{\mathcal{A}_2^0(\tau)-\tau^2\} \leq S (\tilde{c}_1^{0})^{-1}\Bigr) \Pb\Bigl( \sup_{ \tau \in \R} \{\mathcal{A}_2^1(\tau) - \tau^2\} \leq S (\tilde{c}_1^{1})^{-1}\Bigr) \\
	& = F_{2 \to 1;0}( S (\tilde{c}_1^{0})^{-1} ) F_{1}(2^{2/3} S (\tilde{c}_1^{1})^{-1})
\end{aligned}
	\end{equation}
by Corollary 1.3 of~\cite{Jo03b} and Theorem 1 of~\cite{QR13b}. By $F_{2 \to 1;0}$, we denote the distribution function of the Airy$_{2 \to 1}$ process in the point $0$.
\end{itemize}

\section{Backwards paths and couplings}\label{sectCoupling}

In order to control the space-time region which influences the position of a tagged particle, we introduce the backwards paths. In the framework of particle positions, their notion goes back to~\cite{Fer18}, see also~\cite{FN19,BBF21}.

Unless mentioned otherwise, $x(t)$ denotes a TASEP with arbitrary initial condition throughout this section. For any label $N \in \Z$, we define a process of backwards indices $N(t \downarrow \cdot)$ starting from time $t > 0$ backwards to time $0$ in the following way: we set $N(t \downarrow t) = N$. For each $s \in [0,t]$ such that\footnote{By $s^+$, we mean a time infinitesimally larger than $s$.} $N(t \downarrow s^+) = n \in \Z$ and there is a suppressed jump attempt of the particle with label $n$ at time $s$, we update $N(t \downarrow s) = n-1$.

The \emph{backwards path associated to the label $N$ at time $t$} is defined as
\begin{equation}
\pi_{N,t} = \{ x_{N(t \downarrow s)}(s), s \in [0,t]\}.
\end{equation}
 Since for TASEP, the probability of simultaneous jump attempts of several particles equals $0$, the backwards path $\pi_{N,t}$ almost surely has steps of size $1$.

\subsection{Finite time results}
In this section, we collect some properties of backwards paths on finite time intervals. First, we discuss how the theory of backwards paths can be applied for the localization of correlation of particles in TASEP with respect to space and time.

\paragraph*{Independence away from the backwards path region} A tagged particle position $x_N(t)$ at time $t$ does not depend on the behaviour of particles outside of the region of its backwards path. More precisely, if one can find a deterministic region that contains the backwards path with probability converging to $1$ as time tends to infinity, then the tagged particle is asymptotically independent of the randomness outside of this region. Though used before~\cite{Fer18,FN19}, a detailed statement (outside the proofs) about this property of backwards paths has not been written down yet. We do so in the following Lemma~\ref{lemma_tagged_particle_position_only_depends_on_backwards_path} and Corollary~\ref{CorIndependence}.
\begin{lem} \label{lemma_tagged_particle_position_only_depends_on_backwards_path}
For a given pair $N,t$, define the event $E_{\cal C}$ such that the backwards path $\pi_{N,t}$ starting at $x_N(t)$ is contained in a deterministic region $\mathcal{C} \subseteq \Z \times [0,t]$. Denote by $\tilde{x}(t)$ the TASEP with $\tilde{x}(0) = x(0)$ that shares its jump attempts with $x(t)$ in $\mathcal{C}$, but has a \emph{deterministic} density $1$ (resp.\ $0$) to the left (resp.\ right) of $\cal C$. Then
\begin{equation}
\Pb(x_N(t)= \tilde x_N(t)|E_{\cal C})=1.
\end{equation}
In other words, since the configuration of $\tilde x$ is deterministic outside of $\cal C$, on the event $E_{\cal C}$, the random variable $x_N(t)$ is independent of the randomness outside of $\cal C$.
\end{lem}
\begin{remark} \label{remark_on_lemma_3.1}
In fact, since in the underlying graphical construction of $x(t)$, there are almost surely no simultaneous Poisson events at any time, we actually have $E_{\cal C} \subseteq \{ x_N(t) = \tilde{x}_N(t) \}$ almost surely.
\end{remark}
\begin{proof}
Assume that the event $E_{\cal C}$ occurs. We have to show that $x_N(t) = \tilde{x}_N(t)$. This is obtained similarly as Proposition 3.4 of~\cite{Fer18}. By definition, we have \mbox{$x_n(t) \leq \tilde{x}_n(t)$} for all $n,t$. It holds $\tilde{x}_{N(t \downarrow 0)}(0) = x_{N(t \downarrow 0)}(0)$, with $N(t \downarrow \cdot)$ constructed via the evolution of $x(t)$. If we still have $N(t \downarrow \tau) = N(t \downarrow 0)$ for $\tau \in [0,t]$ and $\tilde{x}_{N(t \downarrow 0)}(\tau)$ jumps to the right, then the same holds for $x_{N(t \downarrow \tau)}(\tau)$ since otherwise, the backwards index would be updated to $N(t \downarrow 0)-1$. Thus, for $\tau_0$ being the time when \mbox{$N(t \downarrow \tau_0) = N(t \downarrow 0)+1$} is updated to $N(t \downarrow 0)$, we have\footnote{By $\tau_0^-$, we mean a time infinitesimally smaller than $\tau_0$.} $\tilde{x}_{N(t \downarrow \tau_0 ^- )} (\tau_0 ^- ) = x_{N(t \downarrow \tau_0 ^-)} (\tau_0 ^-)$. This yields
\begin{equation}
x_{N(t\downarrow 0)}(\tau_0) - 1 = x_{N(t \downarrow \tau_0)}(\tau_0) \leq \tilde{x}_{N(t \downarrow \tau_0)}(\tau_0) < \tilde{x}_{N(t \downarrow 0)}(\tau_0) = x_{N(t \downarrow 0)}(\tau_0). \label{eq_pf_dep_only_on_backwards_path}
\end{equation}
The last equality holds true since we almost surely do not have several jump attempts at the same time. From \eqref{eq_pf_dep_only_on_backwards_path}, we deduce
$x_{N(t \downarrow \tau_0)}(\tau_0) = \tilde{x}_{N(t \downarrow \tau_0)}(\tau_0)$. Iterating this procedure, we obtain $x_N(t) = \tilde{x}_N(t)$.
\end{proof}
As a direct corollary, we get the following result.
\begin{cor}\label{CorIndependence}
For two pairs $N_1,t_1$ and $N_2,t_2$ and two deterministic \emph{disjoint} regions ${\cal C}_1,{\cal C}_2$, define the event $E=\{\pi_{N_1,t_1}\subseteq{\cal C}_1,\pi_{N_2,t_2}\subseteq{\cal C}_2\}$. Then, for any $x_1,x_2\in\Z$,
\begin{equation}
\Pb(\{x_{N_1}(t_1)\leq x_1\}\cap\{x_{N_2}(t_2)\leq x_2\})=\Pb(x_{N_1}(t_1)\leq x_1)\Pb(x_{N_2}(t_2)\leq x_2)+R
\end{equation}
with $|R|\leq 6\Pb(E^c).$
\end{cor}
\begin{proof}
For $k=1,2$, let $x^{(k)}$ denote the process with deterministic configurations outside of ${\cal C}_k$ (as $\tilde x$ defined in Lemma~\ref{lemma_tagged_particle_position_only_depends_on_backwards_path}). Then, for any $x_1,x_2\in\Z$,
\begin{equation}
\begin{aligned}
&\left|\Pb(\{x_{N_1}(t_1)\leq x_1\}\cap\{x_{N_2}(t_2)\leq x_2\})-\Pb(\{x_{N_1}(t_1)\leq x_1\}\cap\{x_{N_2}(t_2)\leq x_2\}\cap E)\right|\\
&\leq \Pb(E^c).
\end{aligned}
\end{equation}
For the second term, by Lemma~\ref{lemma_tagged_particle_position_only_depends_on_backwards_path} we have
\begin{equation}
\Pb(\{x_{N_1}(t_1)\leq x_1\}\cap\{x_{N_2}(t_2)\leq x_2\}\cap E)=\Pb(\{x^{(1)}_{N_1}(t_1)\leq x_1\}\cap\{x^{(2)}_{N_2}(t_2)\leq x_2\}\cap E).
\end{equation}
Furthermore,
\begin{equation}
\begin{aligned}
&|\Pb(\{x^{(1)}_{N_1}(t_1)\leq x_1\}\cap\{x^{(2)}_{N_2}(t_2)\leq x_2\}\cap E)-\Pb(x^{(1)}_{N_1}(t_1)\leq x_1)\Pb(x^{(2)}_{N_2}(t_2)\leq x_2)|\\
&\leq \Pb(E^c),
\end{aligned}
\end{equation}
where we used that the processes $x^{(1)}_{N_1}(t_1)$ and $x^{(2)}_{N_2}(t_2)$ are independent random variables since ${\cal C}_1 \cap {\cal C}_2 = \varnothing$.
Therefore,
\begin{equation} \label{eq3.8}
\begin{aligned}
&|\Pb(\{x_{N_1}(t_1)\leq x_1\}\cap\{x_{N_2}(t_2)\leq x_2\})-\Pb(x^{(1)}_{N_1}(t_1)\leq x_1)\Pb(x^{(2)}_{N_2}(t_2)\leq x_2)|\\
&\leq 2\Pb(E^c).
\end{aligned}
\end{equation}
Thus,
\begin{equation}
\begin{aligned}
&\left|\Pb(\{x_{N_1}(t_1)\leq x_1\}\cap\{x_{N_2}(t_2)\leq x_2\})-\Pb(x_{N_1}(t_1)\leq x_1)\Pb(x_{N_2}(t_2)\leq x_2)\right|\\
&\leq 2\Pb(E^c) + |\Pb(x^{(1)}_{N_1}(t_1)\leq x_1)\Pb(x^{(2)}_{N_2}(t_2)\leq x_2)-\Pb(x_{N_1}(t_1)\leq x_1)\Pb(x_{N_2}(t_2)\leq x_2)|.
\end{aligned}
\end{equation}
The second summand is bounded by
\begin{equation}
\begin{aligned}
&|\Pb(x^{(1)}_{N_1}(t_1)\leq x_1)-\Pb(x_{N_1}(t_1)\leq x_1)| \Pb(x^{(2)}_{N_2}(t_2)\leq x_2)\\
&+\Pb(x_{N_1}(t_1)\leq x_1)|\Pb(x^{(2)}_{N_2}(t_2)\leq x_2)-\Pb(x_{N_2}(t_2)\leq x_2)|\leq 4 \Pb(\E^c).
\end{aligned}
\end{equation}
In the last step, we applied the same arguments as for \eqref{eq3.8} for both summands.
\end{proof}

\begin{remark} \label{remark_on_cor_3.3}
Corollary~\ref{CorIndependence} rigorously outlines how to achieve independence of particle positions when their backwards paths are restricted to disjoint regions. This application of Lemma~\ref{lemma_tagged_particle_position_only_depends_on_backwards_path} can be extended to tagged particle positions in different TASEPs sharing their jump attempts at each site, and further from fixed times to time intervals. Our use of Lemma~\ref{lemma_tagged_particle_position_only_depends_on_backwards_path} in Section~\ref{sectLocal} and Section~\ref{section_proof_of_the_main_results} is more implicit since we only control the fluctuations of backwards paths in one direction and utilize a particle-hole duality afterwards.
\end{remark}

\paragraph*{Two couplings}
Further properties of backwards paths as well as implications of their theory are obtained by comparing them for different, suitably coupled TASEPs. For this purpose, we consider two kinds of coupling. TASEPs are coupled by the well-known \emph{basic coupling} if one uses the same family of Poisson processes in the graphical construction by Harris~\cite{Har72,Har78}, meaning that the TASEPs share their jump attempts at each site. Still, the underlying Poisson processes can also be attached to the particles instead of to the sites, see~\cite{GKM21}. Based on this, we say that several TASEPs are coupled by \emph{clock coupling} if the jump attempts of their respective particles with the same label are described by the same Poisson process.

For both couplings, it is possible to obtain a concatenation property of the backwards paths. First considering the case of basic coupling, we introduce the following notation: let $x^{\textrm{step}, Z}(\tau,t )$ describe a TASEP starting at time $\tau \geq 0$ from step initial condition with rightmost particle at position $Z \in \mathbb{Z}$. That is, \begin{equation}x_n^{\textrm{step}, Z}(\tau, \tau) = - n + Z + 1 \textrm{ for } n \in \N.\end{equation} We omit $Z$ in the notation if $Z=0$, and if we also have $\tau = 0$, then we write $x^{\textrm{step}}(t)$. Additionally, we denote $y_n^Z(\tau,t) = x_n^{\textrm{step}, Z}(\tau,t) - Z$.
\begin{prop}[Proposition 3.4 of~\cite{Fer18}] \label{Prop_3.4_[Fer18]_basic_coupling} Let $\tau \in [0,t]$ and assume that all occurring processes are coupled by basic coupling. Then, it holds
	\begin{equation} x_N(t) = x_{N(t \downarrow \tau)}(\tau) + y_{N-N(t \downarrow \tau)+1}^{x_{N(t \downarrow \tau)}(\tau)} (\tau,t) = x_{N-N(t \downarrow \tau)+1}^{{\rm step},x_{N(t \downarrow \tau)}(\tau)} (\tau,t)\end{equation}
	and
	\begin{equation} y_{N-N(t \downarrow \tau)+1}^{x_{N(t \downarrow \tau)}(\tau)} (\tau,t) \overset{(d)}{=} x_{N-N(t \downarrow \tau) +1}^{{\rm step}} (t-\tau).\end{equation}
\end{prop}
Moreover, Lemma 2.1 of~\cite{Sep98c} yields
\begin{equation}
x_N(t) = \min_{n \leq N} \Bigl \{ x_n(\tau) + y_{N-n+1}^{x_n(\tau)}(\tau,t) \Bigr \} \label{estimate_Sep98_backwards_paths}.
\end{equation}
These identities hold true given that at no time several Poisson events occur simultaneously in the graphical construction of the processes. This is almost surely the case. Thus, we keep in mind that the statements in this section actually hold with probability $1$, without mentioning it explicitly each time.

By inspecting the proofs of Proposition~\ref{Prop_3.4_[Fer18]_basic_coupling} and \eqref{estimate_Sep98_backwards_paths}, we observe that the same arguments hold true for the case of clock coupling of the processes.
\begin{cor} \label{Prop_3.4_[Fer18]_clock_coupling} We denote by $y^{{\rm step}, m}(\tau;t)$ a TASEP with step initial condition starting at time $\tau \in [0,t]$, in which all particles with labels smaller than $m \in \mathbb{Z}$ are removed and whose rightmost particle starts at the position $x_m(\tau)$. That is, $y_n^{{\rm step}, m}(\tau;\tau) = x_m(\tau) - n + m$ for $n \geq m$. We couple the process with $x(t)$ by clock coupling. Then, it holds
\begin{equation}
x_N(t) = y_N^{{\rm step}, N(t \downarrow \tau)}(\tau;t) \label{eq_conca_clock_1}
\end{equation}
and
\begin{equation}
x_N(t) = \min_{m \leq N} \{ y_N^{{\rm step}, m}(\tau;t) \}. \label{eq_conca_clock_2}
\end{equation}
\end{cor}

\paragraph*{Comparison of backwards paths to TASEP with step initial condition}
A crucial observation for the bound on right fluctuations of a backwards path in Proposition~\ref{prop_right_fluctuations_backwards_path} is that under basic coupling, we can control them on a finite time interval by considering another backwards path in a TASEP with step initial condition instead.
\begin{lem} \label{lemma_prob_1_ordering_backwards_paths} Let $0 \leq t_1 < t_2 \leq T$. Assume $x_{N(T \downarrow t_1)}(t_1) \leq x_1$, $x_{N(T \downarrow t_2)}(t_2) \leq x^{{\rm step},x_1}_M(t_1,t_2)$ for some label $M \in \Z$, and the processes $x(t)$ and $ x^{{\rm step},x_1}(t_1, t)$ are coupled by basic coupling. Then for any $\tau\in [t_1,t_2]$,
	\begin{equation}
x_{N(T \downarrow \tau)}(\tau) \leq x^{{\rm step}, x_1}_{M(t_2 \downarrow \tau)}(t_1,\tau).
\end{equation}
\end{lem}
This lemma holds true independently of the initial condition chosen for the original TASEP. For its proof, we first prove Lemma~\ref{lemma_alwaysstepparticle}, which states that if the backwards path $x_{N(T \downarrow \tau)}(\tau)$ is to the left of a site $x_1 \in \Z$ at time $t_1 \in [0,T]$, then (almost surely) at times $\tau \in [t_1,T]$ there is always a particle of the TASEP $x^{\textrm{step},x_1}(t_1, \tau)$ at the same position as the backwards path. Both relations are illustrated in Figure~\ref{figure_relation_backwards_paths_step_IC}.

\begin{lem} \label{lemma_alwaysstepparticle}
	Assume $x_{N(T \downarrow t_1)}(t_1) \leq x_1$ and the processes $x(t)$ and $ x^{{\rm step},x_1}(t_1, t)$ are coupled by basic coupling. Then,
	for each $\tau \in [t_1,T]$ there exists a label $m_\tau \in \N$ such that \begin{equation} x_{N(T \downarrow \tau)}(\tau) = x^{{\rm step}, x_1}_{m_\tau}(t_1, \tau). \end{equation}
\end{lem}
\begin{figure}[t!]
	\centering
	\includegraphics[scale=1]{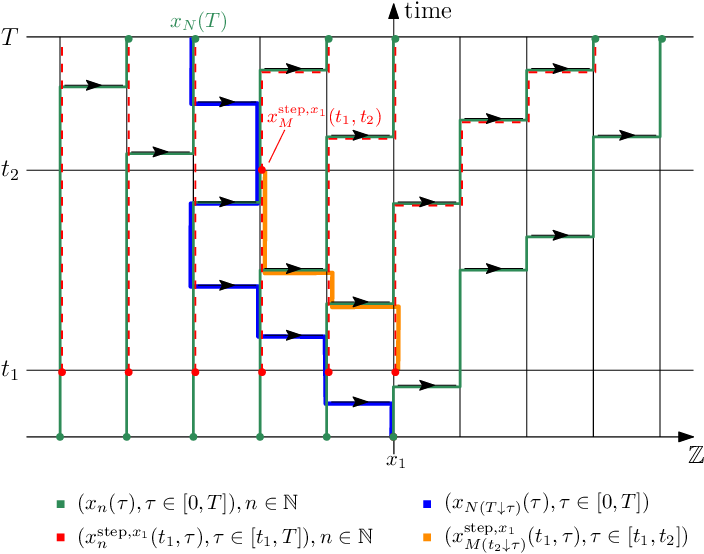}
	\caption{ Illustration of the relations displayed in Lemma~\ref{lemma_prob_1_ordering_backwards_paths} and Lemma~\ref{lemma_alwaysstepparticle}. }
	\label{figure_relation_backwards_paths_step_IC}
\end{figure}
\begin{proof}
	Since we have $x_{N(T \downarrow t_1)}(t_1) \leq x_1$, for all $n \in \{ N(T \downarrow t_1), \dots, N \}$ there is some $m = m(n) \in \N$ such that $x_n(t_1) = x^{\textrm{step}, x_1}_m(t_1, t_1)$. Suppose there exists a first time $\tau_1 > t_1$ such that for some $n \in \{ N(T \downarrow t_1), \dots, N \}$, there is no $m \in \N$ such that $x_n(\tau_1) = x^{\textrm{step}, x_1}_m(t_1,\tau_1)$.
	Then, two cases emerge, assuming $x_n(\tau_1^-) = x_m^{\textrm{step},x_1}(t_1,\tau_1^-)$.

	\textbf{Case 1:} $x_n(\tau_1)$ jumps and $x_m^{\textrm{step}, x_1}(t_1,\tau_1)$ is prohibited to jump. But then, there needs to be a particle of $x^{\textrm{step}, x_1} (t_1, \cdot)$ at the new position $x_n(\tau_1)$, thus this cannot happen in the discussed scenario.

	\textbf{Case 2:} The jump trial of $x_n(\tau_1)$ is suppressed while $x_m^{\textrm{step}, x_1}(t_1,\tau_1)$ jumps. Then, there is a particle $x_k(\tau_1)$ with label $k \in \N$ at the new position of $x_m^{\textrm{step}, x_1}(t_1,\tau_1)$ whereas there is no particle of $x^{\textrm{step}, x_1}(t_1,\cdot)$. As there are no simultaneous Poisson events in the graphical construction, there cannot be a jump attempt at site $x_k(\tau_1)$ at time $\tau_1$ as well. Since we set $\tau_1$ to be the \emph{first} time such that for \mbox{$n \in \{N(T \downarrow t_1), \dots, N\}$} there is no $m \in \N$ such that $x_n(\tau_1) = x^{\textrm{step}, x_1}_m(t_1,\tau_1)$, we obtain $k \not \in \{N(T \downarrow t_1), \dots,N\}$, thus $n = N(T \downarrow t_1), k = N(T \downarrow t_1) -1$. But in this case, it cannot hold $N(T \downarrow \tau_1) = n = N(T \downarrow t_1)$ since we would need to update the index process to $N(T \downarrow t_1)-1$ here. Thus, we already have $N(T \downarrow \tau_1) \geq N(T \downarrow t_1) +1 $.

	Next, we argue similarly for $n \in \{ N(T \downarrow t_1) +1, \dots, N \}$ and start at time $\tau_1$, as again for each $n \in \{ N(T \downarrow t_1) + 1 , \dots, N \}$ there exists a label $m \in \N$ such that $x_n(\tau_1) = x^{\textrm{step}, x_1}_m(t_1, \tau_1)$.
	Since the set $ \{ N(T \downarrow t_1), \dots, N \}$ is finite, the iteration ends after a finite number of steps.
\end{proof}

We are now ready to prove Lemma~\ref{lemma_prob_1_ordering_backwards_paths}.
\begin{proof}[Proof of Lemma~\ref{lemma_prob_1_ordering_backwards_paths}]
Suppose, going backwards in time, there is a time $\tau \in (t_1,t_2]$ such that $x_{N(T \downarrow \tau)}(\tau) = x^{\textrm{step},x_1}_{M(t_2 \downarrow \tau)}(t_1, \tau)$. We need to treat two scenarios:

	\textbf{Claim 1:} If it holds $x^{\textrm{step},x_1}_{M(t_2 \downarrow \tau^-)}(t_1, \tau^-) = x^{\textrm{step},x_1}_{M(t_2 \downarrow \tau)}(t_1, \tau) -1 $, then $x_{N(T \downarrow \tau^-)}(\tau^-) = x_{N(T \downarrow \tau)}(\tau) - 1$ is valid as well.
	
 First, suppose $x_{N(T \downarrow \tau^-)}(\tau^-) = x_{N(T \downarrow \tau)}(\tau)$. We know by Lemma~\ref{lemma_alwaysstepparticle} that for each $t_1 \leq s < \tau$, there is some label $m_s \in \N$ such that it holds $x_{N(T \downarrow s)}(s) = x^{\textrm{step},x_1}_{m_s}(t_1, s)$. But then, since we premised $x^{\textrm{step},x_1}_{M(t_2 \downarrow \tau^-)}(t_1, \tau^-) = x^{\textrm{step},x_1}_{M(t_2 \downarrow \tau)}(t_1, \tau) -1 $, this jump and the one of the particle of $x^{\textrm{step},x_1}(t_1, \cdot)$ previously at the position $x_{N(T \downarrow \tau)}(\tau)$ need to happen simultaneously at time $\tau$. As there are no simultaneous Poisson events in the construction of the processes, this cannot be true. 	Next, suppose $x_{N(T \downarrow \tau^-)}(\tau^-) = x_{N(T \downarrow \tau)}(\tau) +1$. Then, at time $\tau$ there must be a suppressed jump attempt of the $N(T \downarrow \tau)$-th particle and, since we assumed $x^{\textrm{step},x_1}_{M(t_2 \downarrow \tau^-)}(t_1, \tau^-) = x^{\textrm{step},x_1}_{M(t_2 \downarrow \tau)}(t_1, \tau) -1 $, also a jump attempt at the position next to the left. Again, this cannot occur. Thus, we have verified Claim 1.

	\textbf{Claim 2:} If it holds $x_{N(T \downarrow \tau^-)}(\tau^-) = x_{N(T \downarrow \tau)}(\tau) +1$, then $x^{\textrm{step}, x_1}_{M(t_2 \downarrow \tau^-)}(t_1,\tau^-) = x^{\textrm{step}, x_1}_{M(t_2 \downarrow \tau)}(t_1,\tau) + 1$ is valid as well.
	
	 We notice that given a jump attempt at time $\tau$, the equality $x^{\textrm{step}, x_1}_{M(t_2 \downarrow \tau^-)}(t_1,\tau^-) = x^{\textrm{step}, x_1}_{M(t_2 \downarrow \tau)}(t_1,\tau)$ is impossible: either the jump trial is suppressed and the backwards paths moves to the right, or the $M(t_2 \downarrow \tau^-)$-th particle jumps at time $\tau$.

	 Further, if it holds $x^{\textrm{step}, x_1}_{M(t_2 \downarrow \tau^-)}(t_1,\tau^-) = x^{\textrm{step}, x_1}_{M(t_2 \downarrow \tau)}(t_1,\tau) - 1,$ then there needs to be a jump attempt at $x^{\textrm{step}, x_1}_{M(t_2 \downarrow \tau)}(t_1,\tau) - 1$ at time $\tau$. Because of the assumption in Claim 2, there is also a jump attempt at $x^{\textrm{step}, x_1}_{M(t_2 \downarrow \tau)}(t_1,\tau) = x_{N(T \downarrow \tau)}(\tau)$. But several synchronous jump attempts cannot occur. Thus, Claim 2 is confirmed.

	Combining Claim 1 and Claim 2, we deduce that the order of the backwards paths is maintained at all times $\tau \in [t_1,t_2]$.
\end{proof}

\paragraph*{Comparison of increments with clock coupling} Apart from describing space-time-correlations, key applications of the theory of backwards paths are the comparison of increments of tagged particle positions over time intervals as well as the comparison of particle distances at a fixed time. In both cases, we need clock coupling of the processes.

For the comparison of increments of tagged particle positions, we first observe some elementary properties of clock coupling. We consider two TASEPs $x(t)$ and $\tilde{x}(t)$ under this coupling. To begin with, clock coupling preserves the partial order defined by
\begin{equation}
x(t) \leq \tilde{x}(t) \Leftrightarrow x_n(t) \leq \tilde{x}_n(t) \textrm{ for all } n \in \Z.
\end{equation}
\begin{lem} \label{lemma_clock_coupling_preserves_order}
	Let $x(0) \leq \tilde{x}(0)$. Then, it holds $x(t) \leq \tilde{x}(t)$ for any time $t \geq 0$.
\end{lem}
\begin{proof}
	A similar argument as in the graphical construction by Harris~\cite{Har72,Har78} (with the space variable replaced by the labels of particles) implies that for any given time $t$, the construction of the particle evolution can be divided into almost surely finite (random) blocks of particles in which only finitely many Poisson events take place before time $t$. Thus, it is enough to verify the preservation of the partial order at each Poisson event, which we again suppose to occur at distinct times. Assume that at time $t_0$, there is a jump attempt of the particle with label $n \in \Z$, and $x(\tau) \leq \tilde{x}(\tau)$ for $\tau \in [0,t_0)$. If $x_n(t_0 ^-) < \tilde{x}_n(t_0 ^-)$, then also $x_n(t_0) \leq \tilde{x}_n(t_0)$. On the other hand, if $x_n(t_0 ^-) = \tilde{x}_n(t_0 ^-)$, then $x_{n-1}(t_0^ -) \leq \tilde{x}_{n-1}(t_0 ^-)$ implies that a jump of $x_n$ gives a jump of $\tilde{x}_n$ as well. Thus, we still have $x_n(t_0) \leq \tilde{x}_n(t_0)$.
\end{proof}
Most importantly, if in one initial configuration the spaces between particles are larger than in the other initial configuration, then this property is preserved over time and further yields an order of increments.
\begin{lem} \label{lemma_clock_coupling_particle_distances}
	Assume $x_{n-1}(0) - x_n(0) \geq \tilde{x}_{n-1}(0) - \tilde{x}_n(0)$ for each $n \in \Z$. Then, for each $n \in \Z$ and any time $t \geq 0$,
	\begin{equation}
	x_{n-1}(t) - x_n(t) \geq \tilde{x}_{n-1}(t) - \tilde{x}_n(t). \label{eq1_lemma_clock_coupling_particle_distances}
	\end{equation}
	Furthermore, for times $0 < t_1 < t_2$, the increments of the processes satisfy
	\begin{equation}
	x_n(t_2) - x_n(t_1) \geq \tilde{x}_n(t_2) - \tilde{x}_n(t_1). \label{eq2_lemma_clock_coupling_particle_distances}
	\end{equation}
\end{lem}
\begin{proof}
	Suppose that at time $t$, the clock of the particles with label $n$ rings. Only gaps between the particles with labels $n+1$ and $n$ respectively $n$ and $n-1$ can change. There are three possible cases:

(a) If $x_{n-1}(t^-) - x_{n}(t^-) = \tilde{x}_{n-1}(t^-) - \tilde{x}_{n}(t^-) = 1$, then both $x_n$ and $\tilde{x}_n$ do not jump.

(b) If $x_{n-1}(t^-) - x_{n}(t^-) > \tilde{x}_{n-1}(t^-) - \tilde{x}_{n}(t^-) = 1$, then $x_n$ jumps and $\tilde{x}_n$ does not jump.

(c) If $x_{n-1}(t^-) - x_{n}(t^-) \geq \tilde{x}_{n-1}(t^-) - \tilde{x}_{n}(t^-) > 1$, then both $x_n$ and $\tilde{x}_n$ jump.

	In the cases (a) and (c), the gaps $x_{n-1} - x_n$ and $\tilde{x}_{n-1} - \tilde{x}_n$ as well as $x_{n} - x_{n+1}$ and $\tilde{x}_{n} - \tilde{x}_{n+1}$ change by the same amount, which preserves their order. In case (b), $x_{n-1} - x_n$ decreases by $1$ while $\tilde{x}_{n-1} - \tilde{x}_n$ is unchanged. Since their difference was positive before time $t$, their order is still maintained. On the other hand, $x_n - x_{n+1}$ increases by $1$ while $\tilde{x}_n - \tilde{x}_{n+1}$ remains unchanged. This verifies \eqref{eq1_lemma_clock_coupling_particle_distances}.

	The inequality \eqref{eq2_lemma_clock_coupling_particle_distances} is due to the fact that by \eqref{eq1_lemma_clock_coupling_particle_distances}, whenever $\tilde{x}_n(t)$ jumps, the same applies to $x_n(t)$.
\end{proof}
A direct consequence of Lemma~\ref{lemma_clock_coupling_particle_distances} is that for two TASEPs with step initial condition and rightmost particle at the same position, the displacement of the particle  with less particles on its right is greater than the displacement of the particle with more particles on its right.
\begin{cor} \label{corollary_clock_coupling_order_TASEP_step}
	Consider two TASEPs $x(t), \tilde{x}(t)$ with step initial condition and rightmost particle at site $a \in \Z$. Let the labels of $x(t)$ and $\tilde{x}(t)$ be $\geq m$ respectively $\geq \tilde{m}$, where $\tilde{m} \leq m$. That is, $x_n(0) = a - n + m$, $n \geq m$ and $\tilde{x}_n(0) = a - n + \tilde{m}$, $n \geq \tilde{m}$. Then, for each $N \geq m$, it holds:
	\begin{equation}
	x_N(t) \geq \tilde{x}_N(t) + m - \tilde{m} \textrm{ for all } t > 0, \label{eq1_corollary_clock_coupling_TASEP_step}
\end{equation}
and
\begin{equation}
	 x_N(t_2) - x_N(t_1) \geq \tilde{x}_N(t_2) - \tilde{x}_N(t_1) \textrm{ for all } t_2 > t_1 > 0. \label{eq2_corollary_clock_coupling_TASEP_step}
\end{equation}
\end{cor}
\begin{proof}
	Lemma~\ref{lemma_clock_coupling_particle_distances} yields \eqref{eq2_corollary_clock_coupling_TASEP_step}. Having this, \eqref{eq1_corollary_clock_coupling_TASEP_step} is obtained by setting $t_2 = t$ and $t_1 = 0$.
\end{proof}
By Corollary~\ref{Prop_3.4_[Fer18]_clock_coupling} and Corollary~\ref{corollary_clock_coupling_order_TASEP_step}, we can derive the comparison of tagged particle increments over time intervals. For this, we incorporate a shift in the clock coupling.

\begin{lem} \label{lemma_comparison_particle_increments}
	We consider two TASEPs $x(t),\tilde{x}(t)$ and two labels $N,M \in \Z$. Let $x(t)$ and $\tilde{x}(t)$ be coupled by clock coupling with a shift, meaning that $x_{N+n}(t)$ and $\tilde{x}_{M+n}(t)$ share their jump attempts for each $n \in \Z$. For fixed $0 < t_1 < t_2$, we construct the index process $N(t_2 \downarrow \tau)$ with respect to $x(t)$ and $M(t_1 \downarrow \tau)$ with respect to $\tilde{x}(t)$. Let $\tilde{x}_M(t_1) \leq x_N(t_1)$ and assume there is some $\tau \in [0,t_1]$ such that $x_{N(t_2 \downarrow \tau)}(\tau) = \tilde{x}_{M(t_1 \downarrow \tau)}(\tau)$. Then, it holds
	\begin{equation}
	x_N(t_2) - x_N(t_1) \geq \tilde{x}_M(t_2) - \tilde{x}_M(t_1).
	\end{equation}
\end{lem}

\begin{figure}[t!]
	\centering
	\includegraphics[width=\textwidth]{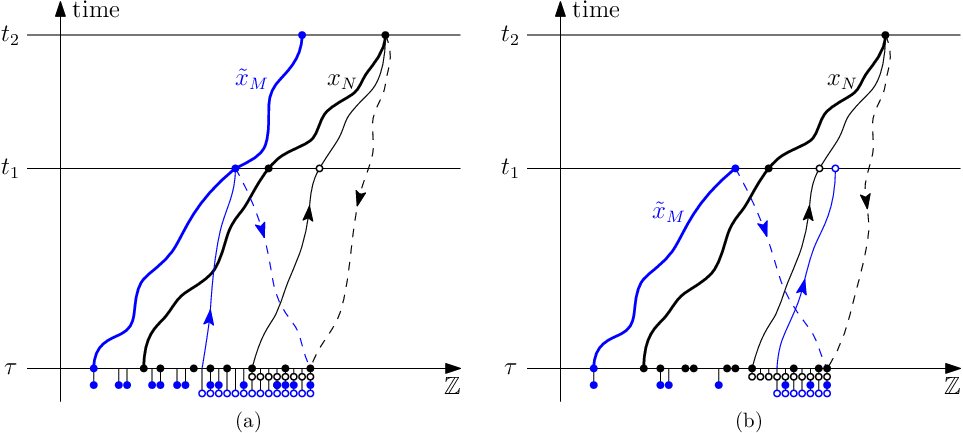}
	\caption{The thick solid lines are the evolution of $x_N$ and $\tilde{x}_M$ and the dashed lines depict their backwards paths. The thin solid lines are the evolution of $y_N^{\textrm{step}, N(t_2 \downarrow \tau)}(\tau;t)$ respectively $y_M^{\textrm{step}, M(t_1 \downarrow \tau)}(\tau;t)$. The solid dots are the particle configurations at time $\tau$ and the empty dots are the configurations after resetting to the step initial condition at time $\tau$. Picture (a) corresponds to $M(t_1 \downarrow \tau) \leq N(t_2 \downarrow \tau)$ (or $N - N(t_2 \downarrow \tau) \leq M - M(t_1 \downarrow \tau)$ if $M \neq N$). Picture (b) shows the case $M(t_1 \downarrow \tau) > N(t_2 \downarrow \tau)$. Here, the contradiction is that the blue empty dot at time $t_1$ should be at position $\tilde{x}_M(t_1) \leq x_N(t_1)$, but is also strictly to the right of the black empty dot at time $t_1$.}
	\label{figure_comparison_particle_increments}
\end{figure}

\begin{proof}
	We provide an illustration of our arguments in Figure~\ref{figure_comparison_particle_increments}. Set \mbox{$x^* = x_{N(t_2 \downarrow \tau)}(\tau) = \tilde{x}_{M(t_1 \downarrow \tau)}(\tau)$}. Without loss of generality, we assume $N=M$. Corollary~\ref{Prop_3.4_[Fer18]_clock_coupling} yields
	\begin{equation}
\begin{aligned}
	& x_N(t_2) = y_N^{\textrm{step}, N(t_2 \downarrow \tau)}(\tau;t_2), \ x_N(t_1) \leq y_N^{\textrm{step}, N(t_2 \downarrow \tau)}(\tau;t_1), \\
	& \tilde{x}_M(t_1) = y_M^{\textrm{step}, M(t_1 \downarrow \tau)}(\tau;t_1), \ \tilde{x}_M(t_2) \leq y_M^{\textrm{step}, M(t_1 \downarrow \tau)}(\tau;t_2) \label{eq_pf_corollary_order_increments_concatenation}
\end{aligned}
	\end{equation}
	with all processes coupled by clock coupling.
	Assume for a moment that
	\begin{equation}
	M(t_1 \downarrow \tau) \leq N(t_2 \downarrow \tau). \label{eq_pf_lemma_comparison_particle_increments}
	\end{equation}
	Then, \eqref{eq_pf_corollary_order_increments_concatenation} and Corollary~\ref{corollary_clock_coupling_order_TASEP_step} applied to $y_N^{\textrm{step}, N(t_2 \downarrow \tau)}(\tau;t)$, $ y_M^{\textrm{step}, M(t_1 \downarrow \tau)}(\tau;t)$ with $a = x^*$, $ m = N(t_2 \downarrow \tau)$, $ \tilde{m} = M(t_1 \downarrow \tau)$ give
	\begin{equation}\label{eq2.26}
\begin{aligned}
	x_N(t_2) - x_N(t_1) &\geq  y_N^{\textrm{step}, N(t_2 \downarrow \tau)}(\tau;t_2) - y_N^{\textrm{step}, N(t_2 \downarrow \tau)}(\tau;t_1) \\
	&\geq  y_M^{\textrm{step}, M(t_1 \downarrow \tau)}(\tau;t_2) - y_M^{\textrm{step}, M(t_1 \downarrow \tau)}(\tau;t_1) \\
	&\geq  \tilde{x}_M(t_2) - \tilde{x}_M(t_1).
\end{aligned}
	\end{equation}
	It remains to prove \eqref{eq_pf_lemma_comparison_particle_increments}. Suppose it did not hold true. Then, by \eqref{eq1_corollary_clock_coupling_TASEP_step} of Corollary~\ref{corollary_clock_coupling_order_TASEP_step}, we would have $x_N(t_1) \leq y_N^{\textrm{step}, N(t_2 \downarrow \tau)}(\tau;t_1) < y_M^{\textrm{step}, M(t_1 \downarrow \tau)}(\tau;t_1) = \tilde{x}_M(t_1)$, which contradicts the assumption $\tilde{x}_M(t_1) \leq x_N(t_1)$.
\end{proof}
In~\cite{BBF21}, this was the statement of Proposition~2.2 (assuming basic coupling). The proof however contains a small mistake, since under basic coupling the second inequality in \eqref{eq2.26} holds true only in law and not pathwise. This is the reason why we consider the clock coupling instead. The comparison results of~\cite{BBF21} are still correct once we use this new coupling, with appropriate change of notations and adapting some of the proofs, which we do below.

An important observation is that, given a family of times $t_1 < t_2$ in an interval $[t,T]$, it is enough to require the conditions of Lemma~\ref{lemma_comparison_particle_increments} for $t, T$.

\begin{lem} \label{lemma_comparison_increments_conditions_only_on_t_T}
	Suppose the conditions of Lemma~\ref{lemma_comparison_particle_increments} are met for times $0 \leq t < T$. Then, for all times $t \leq t_1 < t_2 \leq T$, we likewise have $\tilde{x}_M(t_1) \leq x_N(t_1)$ and there is some $\tau \in [0,t_1]$ such that $x_{N(t_2 \downarrow \tau)}(\tau) = \tilde{x}_{M(t_1 \downarrow \tau)}(\tau)$.
\end{lem}
\begin{proof}
	Given Lemma~\ref{lemma_comparison_particle_increments}, this follows exactly as Lemma 2.6 of~\cite{BBF21}.
\end{proof}

 The comparison of particle distances at a fixed time is captured in Lemma~\ref{lemma_order_at_a_time_intersection_backwards_indices}. Its proof is similar to the one of Lemma 4.6 of~\cite{BF22}. However, in~\cite{BF22}, the result is formulated in the context of height functions and basic coupling instead of the framework of particle positions and clock coupling.
\begin{lem} \label{lemma_order_at_a_time_intersection_backwards_indices}
	We consider two TASEPs $x(t)$ and $\tilde{x}(t)$ coupled by clock coupling as well as two labels $M < \tilde{M}$. For a fixed time $t > 0$, we construct the index process $M(t \downarrow \tau)$ via the evolution of $x(t)$ and the index process $\tilde{M}( t \downarrow \tau)$ via the evolution of $\tilde{x}(t)$. Suppose there is some $\tau \in [0,t]$ such that $M(t \downarrow \tau) = \tilde{M}( t \downarrow \tau)$. Then, it holds
	\begin{equation} x_M(t) - x_{\tilde{M}}(t) \geq \tilde{x}_M(t) - \tilde{x}_{\tilde{M}}(t). \label{ineq_comparison_lemma} \end{equation}
\end{lem}
\smallskip
\begin{proof}
	It holds $\tilde{M}( t \downarrow \tau) = M(t \downarrow \tau) \leq M < \tilde{M}$. By Corollary~\ref{Prop_3.4_[Fer18]_clock_coupling}, we obtain:
	\begin{equation}
\begin{aligned}
	&x_M(t) = y_M^{\textrm{step},M(t \downarrow \tau)}(\tau;t), \ x_{\tilde{M}}(t) \leq y_{\tilde{M}}^{\textrm{step},M(t \downarrow \tau)}(\tau;t), \\
	&\tilde{x}_M(t) \leq y_M^{\textrm{step}, \tilde{M}(t \downarrow \tau)}(\tau;t), \ \tilde{x}_{\tilde{M}}(t) = y_{\tilde{M}}^{\textrm{step}, \tilde{M}(t \downarrow \tau)}(\tau;t). \label{pf_eq_lemma_particle_interdistances_concatenation}
\end{aligned}
	\end{equation}
	Each pair of these processes is coupled by clock coupling. In particular,
	\begin{equation}y^{\textrm{step},M(t \downarrow \tau)}_{M(t \downarrow \tau)+n}(\tau;\cdot) \textrm{ and } y^{\textrm{step}, \tilde{M}(t \downarrow \tau)}_{\tilde{M}(t \downarrow \tau) +n}(\tau;\cdot) \end{equation} share their clocks for each $n \geq 0$. Hence, the only difference between the processes $y^{\textrm{step},M(t \downarrow \tau)}(\tau;\cdot)$ and $y^{\textrm{step},\tilde{M}(t \downarrow \tau)}(\tau;\cdot)$ is the initial position of their rightmost particle. The latter does not influence the interdistances of particles in the respective processes. Thus, we obtain
	\begin{equation}
	y_M^{\textrm{step},M(t \downarrow \tau)}(\tau;t) - y_{\tilde{M}}^{\textrm{step},M(t \downarrow \tau)}(\tau;t) = y_M^{\textrm{step}, \tilde{M}(t \downarrow \tau)}(\tau;t) - y_{\tilde{M}}^{\textrm{step}, \tilde{M}(t \downarrow \tau)}(\tau;t).
	\end{equation}
	Together with \eqref{pf_eq_lemma_particle_interdistances_concatenation}, this implies \eqref{ineq_comparison_lemma}.
\end{proof}

\subsection{Asymptotic results} \label{section_2.2}
Given the finite time results and the estimates collected in Appendix~\ref{appendix_estimates}, asymptotic comparison results can be proven for TASEP with step initial condition and stationary TASEP. We always label the particles in a stationary TASEP such that particles with labels in $\N$ start weakly to the left of the origin, while particles with non-positive labels start strictly to the right of it.

\begin{prop} \label{prop_comparison_increments_of_particles}
	Let $x(t)$ denote a TASEP with step initial condition, set \mbox{$N = \gamma T$} for $\gamma \in (0,1)$ and $t = T - \varkappa T^{2/3}$ with $\varkappa$ in a bounded subset of $\R$. Further, define $\rho_0 = \sqrt{ \tfrac{\gamma T}{t}}$ and $\rho_{\pm} = \rho_0 \pm \kappa t^{-1/3}$ for some $\kappa > 0$ with $\kappa = o(T^{1/3})$. Set $M = \rho_+^2 t - \tfrac{3}{2} \kappa \rho_+ t^{2/3}$ and $P = \rho_-^2 t + \tfrac{3}{2} \kappa \rho_- t^{2/3}$. Consider the stationary TASEPs $x^{\rho_+}(t)$ with density $\rho_+$ and $x^{\rho_-}(t)$ with density $\rho_-$, coupled with $x(t)$ by clock coupling \emph{with a shift}: for each $n \in \mathbb{Z}$, $x_{N+n}, x^{\rho_+}_{M+n}$ and $x^{\rho_-}_{P+n}$ share the same jump attempts.

	Then, with a probability of at least $1 - C e^{-c \kappa}$, it holds for all times $T$ large enough:
	\begin{equation}
	\forall t \leq t_1 < t_2 \leq T: \ x_M^{\rho_+}(t_2) - x_M^{\rho_+}(t_1) \leq x_N(t_2) - x_N(t_1) \leq x_P^{\rho_-}(t_2) - x_P^{\rho_-}(t_1) .
	\end{equation}
	The constants $C, c > 0$ can be chosen uniformly for large times $T$, for the different densities that appear and for $\gamma$ in a closed subset of $(0,1)$.
\end{prop}
With help of Lemma~\ref{lemma_comparison_particle_increments}, Lemma~\ref{lemma_comparison_increments_conditions_only_on_t_T} and the estimates in Appendix~\ref{appendix_estimates}, Proposition~\ref{prop_comparison_increments_of_particles} can be proven as described in~\cite{BBF21} for Theorem 2.8 of~\cite{BBF21}. As most details are similar to the proof of Proposition~\ref{prop_on_particle_distances_stationary} and only the latter will be applied in this work, we do not repeat the arguments here.

By Lemma~\ref{lemma_order_at_a_time_intersection_backwards_indices}, we obtain the comparison of particle interdistances in TASEP with step initial condition and stationary TASEP.

\begin{prop} \label{prop_on_particle_distances_stationary}
	Let $x(t)$ denote a TASEP with step initial condition, set \mbox{$N= \gamma T$} for $ \gamma \in (0,1)$ and $t = T - \varkappa T^{2/3}$ with $\varkappa$ in a bounded subset of $\R$. Define $\rho_0 = \sqrt{ \tfrac{\gamma T}{t}}$ and $\rho_{\pm} = \rho_0 \pm \kappa t^{-1/3}$ for some $\kappa > 0$ with $\kappa = o(T^{1/3})$. Consider the stationary TASEPs $x^{\rho_+}(t)$ with density $\rho_+$ and $x^{\rho_-}(t)$ with density $\rho_-$, coupled with $x(t)$ by clock coupling. Set $M = \rho_+^2 t - \tfrac{3}{2} \kappa \rho_+ t^{2/3}$ and $P = \rho_-^2 t + \tfrac{3}{2} \kappa \rho_- t^{2/3}$.
 Then, with a probability of at least $1 - C e^{-c \kappa}$, it holds for all times $T$ large enough:
	\begin{equation} x_N(t) - x_M(t) \geq x_N^{\rho_+}(t) - x_M^{\rho_+}(t) \textrm{ and } x_P(t) - x_N(t) \leq x_P^{\rho_-}(t) - x_N^{\rho_-}(t).\end{equation}
	The constants $C,c > 0$ can be chosen uniformly for large times $T$, for the different densities that appear and for $\gamma$ in a closed subset of $(0,1)$.
\end{prop}

\begin{proof}
	First, we bound the particle distances from above by the ones in the stationary TASEP with slightly lower density.

	Lemma~\ref{Lemma 2.8}, Lemma~\ref{Lemma A.3} with $ w = - \tfrac{3}{4} \kappa \chi^{-1/3}$ and Lemma~\ref{Lemma 2.9} yield for all times $T$ large enough:
	\begin{equation}
\Pb(N(t\downarrow 0)> \kappa t^{1/3})= \Pb( |x_{N(t \downarrow 0)}(0)| \geq \kappa t^{1/3}) \leq C e^{-c \kappa} \label{eq_pf_prop_comparison_distances_estimates}
\end{equation}
and \begin{equation}
\begin{aligned}
&	 \Pb(x_P^{\rho_-}(t) \leq (1-2\rho_-) t - \kappa t^{2/3}) \geq 1 - C e^{-c \kappa}, \\
&	 \Pb( x_P^{\rho_-}(t) - x^{\rho_-}_{P(t \downarrow 0)}(0) \leq (1-2 \rho_-) t - \tfrac{1}{2} \kappa t^{2/3}) \leq C e^{-c \kappa}.
\end{aligned}
	\end{equation}
	These bounds imply
	\begin{equation}
	\Pb( x_{P(t \downarrow 0)}^{\rho_-}(0) \geq - \tfrac{1}{2} \kappa t^{2/3})
	\leq C e^{-c \kappa}.
	\end{equation}
	From $x^{\rho_-}_{P(t \downarrow 0)}(0) \leq - \tfrac{1}{2} \kappa t^{2/3}$ with high probability, we now derive that it holds \mbox{$P(t \downarrow 0) \geq \kappa t^{1/3}$} with high probability as well. In doing so, we denote
\begin{equation} Z_ {\kappa,t} := \# \{ \textrm{particles in } x^{\rho_-}(0) \textrm{ at sites in } \{-\tfrac{1}{2} \kappa t^{2/3} , \dots,  - 1\} \} \sim \textrm{Bin}(\tfrac{1}{2} \kappa t^{2/3}, \rho_-). \end{equation} Utilizing the exponential Chebyshev inequality, we find
\begin{equation} \Pb(Z_{\kappa,t} \geq \kappa t^{1/3}) \geq 1 - e^{-c \kappa t^{2/3}}, \label{eq_pf_comp_distances_chebyshev} \end{equation}
leading to
\begin{equation}
\Pb( P(t \downarrow 0) \geq \kappa t^{1/3}) \geq  \Pb \left( Z_{\kappa,t} \geq \kappa t^{1/3}, x_{P(t \downarrow 0)}^{\rho_-}(0) \leq - \tfrac{1}{2} \kappa t^{2/3}\right)
\geq  1 - C e^{-c \kappa}. \label{eq_pf_comp_interdistances_binomial_rv}
\end{equation}
	Together with \eqref{eq_pf_prop_comparison_distances_estimates}, this implies
	\begin{equation} \label{eq3.39}
\begin{aligned}
	\Pb( P(t \downarrow 0) > N(t \downarrow 0)) \geq & \Pb( P(t \downarrow 0) \geq \kappa t^{1/3} > N(t \downarrow 0))
	\geq  1 - C e^{-c \kappa}.
\end{aligned}
	\end{equation}
	As $P < N$, if we have $P(t \downarrow 0) > N(t \downarrow 0)$, then there almost surely exists some $\tau \in [0,t]$ such that $P(t \downarrow \tau) = N(t \downarrow \tau)$. By Lemma~\ref{lemma_order_at_a_time_intersection_backwards_indices}, we get
	\begin{equation} x_P^{\rho_-}(t) - x_N^{\rho_-}(t) \geq x_P(t) - x_N(t)\end{equation} with the probability greater than $1 - C e^{-c \kappa}$.

	The lower bound is proven by similar means. In fact, it is easier to obtain: here, \begin{equation} \Pb(M(t \downarrow 0) \leq 0) = \Pb(x^{\rho_+}_{M(t \downarrow 0)}(0) > 0) \geq 1 - C e^{-c\kappa}\end{equation} already implies $ \Pb(M(t \downarrow 0) < N(t \downarrow 0)) \geq 1 - C e^{-c\kappa}$ since it holds $N(t \downarrow 0) \geq 1$.

	To conclude, we want to point out that the constants in the bounds obtained from Lemma~\ref{Lemma 2.8}, Lemma~\ref{Lemma A.3} and Lemma~\ref{Lemma 2.9} can be chosen uniformly for all $T$ large enough and for $\gamma$ in a closed subset of $(0,1)$. In fact, as we have $\kappa = o(T^{1/3})$ and $\varkappa$ is contained in a bounded subset of $\R$, for large times the appearing densities $\rho_{\pm} = \sqrt{\gamma} + ( \tfrac{1}{2} \varkappa \sqrt{\gamma} \pm \kappa ) t^{-1/3} + \Or(\varkappa^2 t^{-2/3})$ are contained in a closed subset of $(0,1)$ as well. Therefore, the constants can be chosen uniformly for all those densities.
 Also the application of the exponential Chebyshev inequality for \eqref{eq_pf_comp_distances_chebyshev} allows the choice of uniform constants.
\end{proof}

To see that the choice of $\kappa = o(T^{1/3})$ is admissible, refer to Remark~\ref{rem_possible_extensions_auxiliary_lemmata}.

\begin{remark}
In the case of basic coupling, the comparisons in Lemma~\ref{lemma_comparison_particle_increments}, Proposition~\ref{prop_comparison_increments_of_particles} (with $t_1 < t_2$ fixed) and Lemma~\ref{lemma_order_at_a_time_intersection_backwards_indices}, Proposition~\ref{prop_on_particle_distances_stationary} still hold true in law. However, our localization result on backwards paths in TASEP with step initial condition, specifically the proof of Theorem~\ref{thm_midtimeestimate}, requires the simultaneous comparison of particle interdistances at a fixed time for a family of particles.

More importantly, the proof of Lemma~\ref{corollary_4.1_of_BBF22_tightness_weak_conv_Xtilde} (Proposition 2.9 of~\cite{BBF21}) requires the comparison of tagged particle increments for a family of times $t_1 < t_2$. Consequently, results in law would not be sufficient for our purpose and our arguments indeed require the use of clock coupling. With our Proposition~\ref{prop_comparison_increments_of_particles} (instead of Theorem~2.8 of~\cite{BBF21}), the proof of Proposition 2.9 of~\cite{BBF21}, the weak convergence, is amended and does not require any changes.
\end{remark}

\section{Localization results}\label{sectLocal}
Let $x(t)$ denote a TASEP with step initial condition, let $\gamma \in (0,1)$, $N= \gamma T$, and consider the regions
\begin{equation}\label{eq4.1}
\begin{aligned}
{\cal C}^r=&\{(x,t) \in \Z \times [0,T] \ | \ x\leq (1-2\sqrt{\gamma})t + K T^{2/3}\}, \\
{\cal C}^l=&\{(x,t) \in \Z \times [0,T] \ | \ x\geq (1-2\sqrt{\gamma})t - K T^{2/3}\}.
\end{aligned}
\end{equation}
Construct the process $\tilde{x}^r(t)$ with $\tilde{x}^r(0) = x(0)$ by the same Poisson events as $x(t)$ in ${\cal C}^r$ and with $({\cal C}^r)^c$ completely filled by holes. Likewise, denote by $\tilde{x}^l(t)$ the process with $\tilde{x}^l(0) = x(0)$, constructed by the same Poisson events as $x(t)$ in ${\cal C}^l$ and with $({\cal C}^l)^c$ completely filled by particles. \\
The main objective of Section \ref{sectLocal} is to localize the randomness that $x_{N}(T)$ depends upon with help of $\tilde{x}^r_N(T)$ and $\tilde{x}_N^l(T)$:
\begin{prop} \label{prop_localization_backwards_path} There exists an event $E$ such that
\begin{equation}
\Pb(x_N(T)=\tilde{x}^r_N(T)=\tilde{x}^l_N(T)|E)=1
\end{equation}
and, for $T$ sufficiently large,
	\begin{equation}
\Pb(E^c) \leq C e^{-cK},
\end{equation}
where $1 \leq K = \Or(T^\eps)$ with $\eps\in (0, \tfrac{1}{12})$ fixed. The constants $C,c > 0$ can be chosen uniformly for $T$ large enough and for $\gamma$ in a closed subset of $(0,1)$.
\end{prop}

Proposition~\ref{prop_localization_backwards_path} is proven by controlling the right fluctuations of the backwards path starting at $x_N(T)$ and the left fluctuations of the backwards path of a hole nearby. We start with the first part, as the second is essentially a mirror image of the same problem.

As we will see in the proof of Proposition~\ref{prop_localization_backwards_paths_asympt_indep_part}, Proposition~\ref{prop_localization_backwards_path} can be extended to particle positions on (small) time intervals instead of at fixed times. This is due to an order of backwards paths: for two labels $N,M \in \Z$ with $N \leq M$ and times $0 \leq t < \tilde{t}$, we have
\begin{equation}
x_{M(t \downarrow \tau)}(\tau) \leq x_{N( \tilde{t} \downarrow \tau)}(\tau)\textrm{ for all } \tau \in [0,t].
 \end{equation}

\subsection{Control of fluctuations to the right}
In order to control the right fluctuations of a backwards path in TASEP with step initial condition, we adopt an approach from Section 11 of~\cite{BSS14} that was previously employed for Theorem 4.4 of~\cite{BF20} and Proposition 4.9 of~\cite{BF22}. To our knowledge, it has not been applied in the framework of particle positions before. Its advantage in comparison to other approaches to localization results in this setting, see~\cite{Fer18,FN19}, is that it allows us to obtain the bound in Proposition~\ref{prop_localization_backwards_path} even for constant values of $K$, instead of requiring $K \to \infty$ as $T \to \infty$.

Our notion of backwards paths shows not only similarities but also differences to the backwards geodesics in the cases of LPP models~\cite{BSS14,BF20} and height function representations~\cite{BF22}. In particular, our implementation of the approach of~\cite{BSS14} requires the use of Lemma~\ref{lemma_prob_1_ordering_backwards_paths}. This, however, only gives us control of the right fluctuations of the backwards path. Also, due to the application of Proposition~\ref{prop_on_particle_distances_stationary}, we only get $1$ as exponent in the exponential decay, instead of $2$ respectively $3$~\cite{BF20, BF22}. The next proposition captures the control of right fluctuations of the backwards path.

\begin{prop} \label{prop_right_fluctuations_backwards_path}
	For $ \gamma \in (0,1), N = \gamma T$ and all $T$ sufficiently large, it holds
	\begin{equation} \Pb( x_{N(T \downarrow t)}(t) - (1- 2 \sqrt{\gamma})t > K T^{2/3} \textrm{ for some } t \in [0,T]) \leq C e^{-cK},\end{equation} where $1 \leq K = \Or(T^\eps)$ with $\eps\in (0, \tfrac{1}{12})$ fixed. The constants $C,c > 0$ can be chosen uniformly for $T$ large enough and for $\gamma$ in a closed subset of $(0,1)$.
\end{prop}

\subsection{Mid-time estimate}

The proof of Proposition~\ref{prop_right_fluctuations_backwards_path} requires a mid-time estimate for backwards paths in TASEP with step initial condition. We first state it in terms of backwards indices and adapt the argumentation from the proof of Proposition 4.8 of~\cite{BF22} to the context of particle positions.

\begin{thm} \label{thm_midtimeestimate}
	Let $N = \alpha T$ with $\alpha \in (0,1)$ and fix some $\eps\in (0, \tfrac{1}{3})$. Then, for all sufficiently large times $T$, it holds
	\begin{equation} \Pb( | N( T \downarrow \tfrac{T}{2}) - \tfrac{\alpha}{2} T | > 2 K T^{2/3}) \leq C e^{-c K},\end{equation} where $1\leq K = \Or(T^\eps)$.
	The constants $C, c > 0$ can be chosen uniformly for all $T$ large enough and for $\alpha$ in a closed subset of $(0,1)$.
\end{thm}
\begin{proof}
	Let $\tilde{N} := \tfrac{\alpha}{2} T + 2 K T^{2/3}$. In the following, we show \begin{equation} \Pb(N(T \downarrow \tfrac{T}{2}) \leq \tilde{N}) \geq 1- C e^{-c K} \label{eq.4.20} \end{equation} for all times $T$ large enough and for constants $C,c > 0$ which are uniform in the sense stated above.
	By the relations in Proposition~\ref{Prop_3.4_[Fer18]_basic_coupling} and (\ref{estimate_Sep98_backwards_paths}), we observe that if we have
	\begin{equation} x_N(T) < \min_{ \tilde{N} < n \leq N} \Bigl \{ x_n(\tfrac{T}{2}) + y^{x_n( {T \over 2})}_{N-n+1}(\tfrac{T}{2}, T) \Bigr \} ,\end{equation}
	then this implies $N(T \downarrow \tfrac{T}{2}) \leq \tilde{N}$. As we do not need the coupling of the processes from now on, we replace $y^{x_n( {T \over 2})}_{N-n+1}({T \over 2}, T)$ by $x^{\textrm{step}}_{N-n+1}(\tfrac{T}{2})$ in our notation. For any $S \in \R$, we deduce:
	\begin{equation}
\begin{aligned}
	\Pb( N(T \downarrow \tfrac{T}{2}) \leq \tilde{N})
	\geq & \Pb\Bigl( x_N(T) < S < \min_{\tilde{N} < n \leq N} \{x_n(\tfrac{T}{2}) + x^{\textrm{step}}_{N-n+1}(\tfrac{T}{2})\} \Bigr) \\
	\geq & 1 - \Pb( x_N(T) \geq S) - \Pb\Bigl( \min_{\tilde{N} < n \leq N} \{ x_n(\tfrac{T}{2}) + x^{\textrm{step}}_{N-n+1}(\tfrac{T}{2}) \} \leq S \Bigr).
\end{aligned}
	\end{equation}
	Choosing $S := (1-2 \sqrt{\alpha} )T + \phi K^2 T^{1/3}$ with $\phi := \tfrac{1}{4 \alpha^{3/2}}$, Lemma~\ref{Lemma_A.2} yields
	\begin{equation}\label{eq3.23}
\Pb( x_N(T) \geq S ) \leq C e^{-c K^3}.
\end{equation}
	Further, we define $h(n) := \tfrac{1}{\sqrt{\alpha}} ( n - \tfrac{\alpha}{2} T)$ and observe
	\begin{equation}
\begin{aligned}
	&\Pb\Bigl( \min_{\tilde{N} < n \leq N} \{ x_n(\tfrac{T}{2}) + x^{\textrm{step}}_{N-n+1}(\tfrac{T}{2}) \} \leq S \Bigr )  \\
	& \leq  \Pb\Bigl( \min_{\tilde{N} < n \leq N}\{ x_n( \tfrac{T}{2}) + h(n)\} \leq \tfrac{S}{2} \Bigr) + \Pb\Bigl(\min_{\tilde{N} < n \leq N} \{x^{\textrm{step}}_{N-n+1}(\tfrac{T}{2}) - h(n)\} \leq \tfrac{S}{2}\Bigr) . \label{pf_mid_time_estimate_two_summands_to_be_bounded}
\end{aligned}
	\end{equation}
	
	We only provide the details of the bound on the first summand in (\ref{pf_mid_time_estimate_two_summands_to_be_bounded}) as the bound on the second one can be approached in a similar manner.
	Henceforth, we use the notation
	\begin{equation}
 \Pb \Bigl( \min_{{\tilde{N} < n \leq N}} \{ x_n( \tfrac{T}{2}) + h(n)\} \leq \tfrac{S}{2}\Bigr)  =  \Pb \Bigl( \min_{ 2K T^{2/3} < k \leq \tfrac{\alpha}{2}T}\{ x_{\tfrac{\alpha}{2} T + k}( \tfrac{T}{2}) + h(\tfrac{\alpha}{2} T +k)\}\leq \tfrac{S}{2} \Bigr). \label{first_summand_pf_mid_time_estimate_in_pf}
	\end{equation}
	In order to bound \eqref{first_summand_pf_mid_time_estimate_in_pf}, we partition the domain of $k$, $\{ 2 K T^{2/3} +1, \dots, \tfrac{\alpha}{2}T \}$, into several subregions.
	Having small $k$, that is $k = \Or(K T^{2/3})$, in the domain, it is not possible to derive a bound on \eqref{first_summand_pf_mid_time_estimate_in_pf} solely with help of Lemma~\ref{Lemma_A.2}. Instead, we utilize a comparison to stationary TASEP. For large $k$, the one-point probabilities become very small or even equal zero. Lastly, for $k$ neither too large nor too small, we obtain uniform bounds from Lemma~\ref{Lemma_A.2}.
	Thus, for some $\delta \in (0,\tfrac{1}{3} - \eps)$, we use
\begin{align}
	\eqref{first_summand_pf_mid_time_estimate_in_pf}  & \leq \Pb \Bigl( \min_{ 2 K T^{2/3} < k \leq K T^{2/3 + \delta}}\{ x_{\tfrac{\alpha}{2} T + k}( \tfrac{T}{2}) + h(\tfrac{\alpha}{2} T +k)\} \leq \tfrac{S}{2}\Bigr) \label{smallandlargek_1} \\
	& \quad + \Pb\Bigl( \min_{K T^{2/3 + \delta} < k \leq \tfrac{\alpha}{2}T} \{ x_{\tfrac{\alpha}{2} T + k}( \tfrac{T}{2}) + h(\tfrac{\alpha}{2} T +k)\} \leq \tfrac{S}{2}\Bigr). \label{smallandlargek}
	\end{align}
\textbf{Claim A:} For all $T$ large enough, it holds
$ \eqref{smallandlargek} \leq C e^{-c K^2 T^{2\delta}}$. The constants $C,c > 0$ can be taken uniformly for $\alpha$ in a closed subset of $(0,1)$.

 We suppose $\alpha \in [a,b] \subseteq (0,1)$ and denote $k = \tfrac{\alpha}{2} T \xi$ with $\tfrac{2}{\alpha} K T^{-1/3+\delta} < \xi \leq 1$ as well as \begin{equation} \xi_0 = \min\Bigl(1, 2 \tfrac{1-\sqrt{\alpha}}{\sqrt{\alpha}}\Bigr) = \begin{cases}
1, & \alpha < \tfrac{4}{9}, \\ 2 \tfrac{1-\sqrt{\alpha}}{\sqrt{\alpha}}, & \alpha \geq \tfrac{4}{9}.
\end{cases} \end{equation}
First consider the case $\alpha \geq \tfrac{4}{9}$. Since $\sqrt{\alpha} > \alpha$ and $x_n(t) \geq -n$ for all $t \geq 0, n \in \N$, we obtain for $\xi \geq \xi_0$:
\begin{equation}
\begin{aligned}
x_{ \tfrac{\alpha T}{2} (1+\xi) }(\tfrac{T}{2}) + \tfrac{\sqrt{\alpha}}{ 2} \xi T \geq & - \tfrac{\alpha}{2} T(1+\xi) + \tfrac{\sqrt{\alpha}}{ 2} \xi T
\\ \geq & - \tfrac{\alpha}{2} T(1+\xi_0) + \tfrac{\sqrt{\alpha}}{ 2} \xi_0 T
= (1-2\sqrt{\alpha}) \tfrac{T}{2} + (1-\sqrt{\alpha})^2 \tfrac{T}{2}.
\end{aligned}
\end{equation}
As $K= o(T^{1/3})$, we deduce
\begin{equation}
\Pb(x_{ {{\alpha T} \over 2} (1+\xi) }(\tfrac{T}{2}) + \tfrac{\sqrt{\alpha}}{ 2} \xi T \leq \tfrac{S}{2}) =0
\end{equation} for $T$ large enough.

Thus for all $T$ large enough, $\eqref{smallandlargek}$ is the same as taking the minimum over $K T^{2/3 + \delta} < k \leq \tfrac{\alpha}{2}\xi_0 T$. For such values of $k$ we have\footnote{We introduced $\xi_0$ such that \eqref{eq4.33} holds true for all $\alpha \in [a,b] \subseteq (0,1)$.}
\begin{equation} \label{eq4.33}
( 1 - 2 \sqrt{ \alpha + 2 k T^{-1}} ) \tfrac{T}{2} \geq ( 1 - 2 \sqrt{\alpha}) \tfrac{T}{2} - \tfrac{k}{\sqrt{\alpha}} + \tfrac{k^2}{4 \alpha^{3/2}}T^{-1}
\end{equation} and $\alpha + 2 k T^{-1}$ is contained in a closed subset of $(0,1)$. Thus, by Lemma~\ref{Lemma_A.2}, there exist uniform constants $C,c > 0$ such that for all $T$ large enough,
\begin{equation}
\begin{aligned}
\eqref{smallandlargek} \leq & \sum_{K T^{2/3 + \delta} < k \leq \tfrac{\alpha}{2}\xi_0T} \Pb(x_{{\alpha \over 2} T + k}( \tfrac{T}{2}) \leq ( 1 - 2 \sqrt{ \alpha + 2 k T^{-1}} ) \tfrac{T}{2} - \tfrac{k^2}{4 \alpha^{3/2}}T^{-1} + \tfrac{\phi}{2} K^2 T^{1/3} ) \\ \leq & \sum_{K T^{2/3 + \delta} < k \leq \tfrac{\alpha}{2}\xi_0T} C e^{-c (2k^2 T^{-4/3} - K^2)} \leq  C e^{- c K^2 T^{2\delta}}.
\end{aligned}
\end{equation}

\textbf{Claim B:} For all $T$ large enough, it holds $\eqref{smallandlargek_1} \leq C e^{-cK}$.\\
	Now, we deal with the region where the integer $k$ takes small values. We have
	\begin{equation}
	\eqref{smallandlargek_1}
	\leq  \sum_{\ell=2}^{T^{\delta} -1} \Pb\Bigl( \min_{ k \in I_\ell} \{ x_{\tfrac{\alpha}{2} T + k}( \tfrac{T}{2}) + h(\tfrac{\alpha}{2} T +k)\} \leq \tfrac{S}{2}\Bigr) \label{sum_smallk}
	\end{equation}
	for $I_\ell := (\ell K T^{2/3}, (\ell+1) K T^{2/3}] \cap \Z$. The $\ell$-th term in \eqref{sum_smallk} can be bounded by the sum of
	 \begin{align}
     &\Pb\Bigl( \min_{ k \in I_\ell} \{ x_{\tfrac{\alpha}{2} T + k}( \tfrac{T}{2}) - x_{ \tfrac{\alpha}{2} T + \ell K T^{2/3}} ( \tfrac{T}{2}) + \tfrac{1}{\sqrt{\alpha}} (k - \ell K T^ {2/3})\}
	\leq -\tilde\phi \ell^2 K^2 T^{1/3}\Bigr) \label{sum_smallk_1}
	\end{align}
	and
	\begin{align}
   \Pb\Bigl( x_{\tfrac{\alpha}{2} T + \ell K T^{2/3}}( \tfrac{T}{2}) + h(\tfrac{\alpha}{2} T + \ell K T^{2/3}) \leq \tfrac{S}{2} + \tilde\phi \ell^2 K^2 T^{1/3}\Bigr), \label{sum_smallk_2}
	\end{align}
where we choose $\tilde\phi := \tfrac{1}{3 \alpha^{3/2}}$.

	\textbf{Claim B.1:} For all $T$ large enough, it holds $\sum_{\ell=2}^{T^{\delta} -1} \eqref{sum_smallk_2} \leq C e^{-c K^2}$. \\
	By Lemma~\ref{Lemma_A.2}, series expansion, and the choice of $\tilde\phi$ and $\phi$, we obtain
	\begin{equation}
\begin{aligned}
	& \Pb( x_{\tfrac{\alpha}{2} T + \ell K T^{2/3}}( \tfrac{T}{2}) + h(\tfrac{\alpha}{2} T + \ell K T^{2/3}) \leq \tfrac{S}{2} + \tilde\phi \ell^2 K^2 T^{1/3})\\
	& \leq  \Pb \Bigl( x_{\tfrac{\alpha}{2} T + \ell K T^{2/3}}( \tfrac{T}{2}) \leq ( 1 - 2 \sqrt{ \alpha + 2 \ell K T^{-1/3} } ) \tfrac{T}{2} - \tfrac{5}{96 \alpha^{3/2}} \ell^2 K^2 T^{1/3}\Bigr ) \\
	& \leq  C e^{- c \ell^2 K^2 }
\end{aligned}
	\end{equation}
with uniform constants for all $T$ large enough, $\alpha$ in a closed subset of $(0,1)$ and $\ell \in \{2, \dots, T^{\delta}-1\}$. Summing over $\ell$, Claim B.1 follows.
	
	\textbf{Claim B.2:} For all $T$ large enough, it holds $\sum_{\ell=2}^{T^{\delta} -1}\eqref{sum_smallk_1} \leq C e^{-c K}$. \\
	With help of Lemma~\ref{lemma_order_at_a_time_intersection_backwards_indices} and Proposition~\ref{prop_on_particle_distances_stationary}, we can bound the distances of particles in TASEP with step initial condition by the distances of particles in a stationary TASEP with suitable density.

For a fixed $\ell \in \{2, \dots, T^\delta -1\}$ and $k \in I_\ell$, let
\begin{equation}
t = \tfrac{T}{2}, \ P = \tfrac{\alpha}{2} T + \ell K T^{2/3}, \ N_+ = \tfrac{\alpha}{2} T + (\ell +1 ) K T^{2/3}, \ N_k = \tfrac{\alpha}{2} T + k.
\end{equation}
Also set
\begin{equation}
\kappa = \kappa(\ell)=\tfrac{(\ell+2)K}{2^{1/3} \sqrt{\alpha}},\,
 \rho_0 = \sqrt{ \tfrac{N_+}{ t}} = \sqrt{\alpha} + \tfrac{(\ell +1) K T^{-1/3}}{\sqrt{\alpha}} - \tfrac{(\ell +1)^2 K^2 T^{-2/3}}{2 \alpha^{3/2}} + \Or(\ell^3 K^3 T^{-1}),
\end{equation}
as well as
\begin{equation}
\rho_- = \rho_0 - \kappa ( \tfrac{T}{2})^{-1/3} \textrm{ and } P_- = \rho_-^2 \tfrac{T}{2} + \tfrac{3}{2} \kappa \rho_- ( \tfrac{T}{2})^{2/3}.
\end{equation}
	It holds $N_k \leq N_+$. Constructing both backwards index processes via the evolution of $x(t)$, we hence have $N_k(t \downarrow \tau) \leq N_+(t \downarrow \tau)$ for all $\tau \in [0,t]$.
	Further, the choice of $\kappa$ gives $P_- < P$ for $T$ large enough, uniformly in $\ell$. Constructing both backwards index processes starting at $P_-$ and $P$ with respect to the stationary TASEP $x^{\rho_-}(t)$ with density $\rho_-$, it holds $P_-(t \downarrow \tau) \leq P(t \downarrow \tau)$ for all $\tau \in [0,t]$.

	We couple $x(t)$ and $x^{\rho_-}(t)$ by clock coupling and apply Proposition~\ref{prop_on_particle_distances_stationary} for $t, N_+$ and $P_-$. Instead of $T$, we consider the time $ \tfrac{T}{2}$. With this, $t = \tfrac{T}{2}$ leads to $\varkappa = 0$, and we get $N_+ = \gamma \tfrac{T}{2}$ for $\gamma = \alpha + 2 (\ell+1) K T^{-1/3}$. Since $(\ell+1)K \leq T^{\delta}K = o(T^{1/3})$ (as $\eps+\delta<\tfrac{1}{3}$), $\gamma$ is contained in a closed subset of $(0,1)$ for $T$ large enough and we have $\kappa= o(T^{1/3})$. As seen in the proof of Proposition~\ref{prop_on_particle_distances_stationary} (see \eqref{eq3.39} and the paragraph below), it holds
\begin{equation}
\Pb( \exists \tau \in [0,t]: P_-(t \downarrow \tau) = N_+(t \downarrow \tau)) \geq 1 - Ce^{-c\kappa}.
\end{equation}
But then, since $P_- < P < N_k \leq N_+$ and the probability of several jump attempts at the same time equals zero, we also get
\begin{equation}
\Pb( \forall k \in I_\ell, \exists \tau \in [0,t]: P(t \downarrow \tau) = N_k(t \downarrow \tau)) \geq 1 - C e^{-c\kappa}.
\end{equation}
The constants $C, c > 0$ can be chosen uniformly for $\alpha$ in a closed subset of $(0,1)$, for times $T$ large enough and for $\ell \in \{2, \dots, T^\delta-1\}$.

	Lemma~\ref{lemma_order_at_a_time_intersection_backwards_indices} implies that with a probability of at least $1 - C e^{-c \kappa}$, it holds
	\begin{equation}
  x_{ \tfrac{\alpha}{2} T + k }(t) -x_{ \tfrac{\alpha}{2} T + \ell K T^{2/3}}(t) \geq x^{\rho_-}_{ \tfrac{\alpha}{2} T + k }(t)- x^{\rho_-}_{ \tfrac{\alpha}{2} T + \ell K T^{2/3}}(t),
  \end{equation}
  uniformly for all $k \in I_\ell$. Since $\sum_{\ell =2}^{T^\delta -1} C e^{-c \kappa(\ell)} \leq C e^{-c K}$,
this means that replacing the process $x$ with $x^{\rho_-}$ in $\sum_{\ell=2}^{T^{\delta} -1}\eqref{sum_smallk_1}$, the error term is at most $C e^{-c K}$.

It thus remains to bound \eqref{sum_smallk_1} for the stationary process $x^{\rho_-}$. For notational simplicity we write $\rho=\rho_-$ in the rest of the proof. Let $Z_j$ be independent geometrically distributed random variables with $\Pb(Z_j = i ) = \rho (1-\rho)^{i}, i \geq 0$. Then
	\begin{equation}
x^\rho_{ \tfrac{\alpha}{2} T + \ell K T^{2/3}} ( \tfrac{T}{2}) - x^\rho_{\tfrac{\alpha}{2} T + k}( \tfrac{T}{2}) - \tfrac{1}{\sqrt{\alpha}}(k - \ell K T^{2/3}) \overset{(d)}{=} \sum_{j=1}^{ k - \ell K T^{2/3}} (1 + Z_j - \tfrac{1}{\sqrt{\alpha}}). \label{eq_stationary_distance_as_sum_geom_rv}
\end{equation}
Setting
\begin{equation}
q := \rho-\sqrt{\alpha} =  - \tfrac{K T^{-1/3}}{\sqrt{\alpha}} - \tfrac{(\ell+1)^2 K^2 T^{-2/3}}{2 \alpha^{3/2}} + \Or(\ell^3 K^3 T^{-1})
\end{equation}
we get
\begin{equation}
	\E[Z_j] = \tfrac{1-\rho}{\rho}  = \tfrac{1-\sqrt{\alpha}}{\sqrt{\alpha}} - \tfrac{q}{\alpha} + \Or(q^2).
\end{equation}
Having $q^2 = \Or(\ell^2 K^2 T^{-2/3})$, this implies
\begin{equation}
	\eqref{eq_stationary_distance_as_sum_geom_rv}
	= \sum_{j=1}^{ k - \ell K T^{2/3}} (Z_j - \E[Z_j])  - (k-\ell K T^{2/3}) \tfrac{q}{\alpha} + \Or(\ell^2 K^3).
	\end{equation}
Since $q < 0$, we obtain
	\begin{equation}
\begin{aligned}
	& \Pb\Bigl( \min_{k \in I_\ell} \{ x^\rho_{\tfrac{\alpha}{2} T + k}( \tfrac{T}{2}) - x^\rho_{ \tfrac{\alpha}{2} T + \ell K T^{2/3}} ( \tfrac{T}{2}) + \tfrac{1}{\sqrt{\alpha}}(k - \ell K T^{2/3})\} \leq -\tilde{\phi} \ell^2 K^2 T^{1/3}\Bigr) \\
	& \leq \Pb\Biggl( \max_{k \in I_\ell} \sum_{j=1}^{ k - \ell K T^{2/3}} (Z_j - \E[Z_j]) \geq \tilde{\phi} \ell^2 K^2 T^{1/3} + K T^{2/3} \tfrac{q}{\alpha} - \Or(\ell^2 K^3)\Biggr).
\end{aligned}
	\end{equation}
	For $\ell \in \{2, \dots, T^{\delta}-1\}$ and $T$ large enough, we have
	\begin{equation} \tilde{\phi} \ell^2 K^2 T^{1/3} + K T^{2/3} \tfrac{q}{\alpha} - \Or(\ell^2 K^3) \geq \tfrac{1}{8} \tilde{\phi} \ell^2 K^2 T^{1/3}.\end{equation}
	Thus, to derive Claim B.2 it remains to bound
	\begin{equation}
	 \sum_{\ell=2}^{T^\delta-1} \Pb\Biggl( \max_{ k \in I_\ell} \sum_{j=1}^{ k - \ell K T^{2/3}} (Z_j - \E[Z_j]) \geq \tfrac{1}{8} \tilde{\phi} \ell^2 K^2 T^{1/3}\Biggr). \label{eq_sum_geom_rv_bound}
	\end{equation}
	 As $ \sum\nolimits_{j=1}^{ k - \ell K T^{2/3}} (Z_j - \E[Z_j]) $ is a martingale with respect to the index $k$, we can apply $x \mapsto e^{\lambda x}, \lambda > 0$, on both sides and employ Doob's submartingale inequality. Choosing $\lambda = T^{-1/3}$, a standard computation leads to
	\begin{equation}
	\eqref{eq_sum_geom_rv_bound} \leq \sum_{\ell=2}^{T^\delta-1} C e^{-c \ell^2 K^2} \leq C e^{-c K^2}
	\end{equation}
	with uniform constants for all $\alpha$ in a closed subset of $(0,1)$, all $\ell \in \{2, \dots, T^\delta-1\}$ and for all $T$ sufficiently large.
	
	With the bounds in Claim B.1 and Claim B.2, the proof of Claim B is completed.
	
\paragraph*{Conclusion} Claim A and Claim B together imply
	\begin{equation}
\Pb\Bigl( \min_{\tilde{N} < n \leq N} \{x_n(\tfrac{T}{2}) + h(n)\}\leq \tfrac{S}{2}\Bigr)\leq C e^{-c K}
\end{equation}
with uniform constants $C,c > 0$ for all times $T$ large enough and for $\alpha $ in a closed subset of $(0,1)$. In essentially the same way one shows
	\begin{equation}
\Pb\Bigl( \min_{\tilde{N} < n \leq N} \{ x^{\textrm{step}}_{N-n+1}(\tfrac{T}{2}) - h(n)\} \leq \tfrac{S}{2}\Bigr) \leq C e^{-c K}
\end{equation}
as well. Combining these two bounds and \eqref{eq3.23}, the proof of \eqref{eq.4.20} is completed. The converse estimate on $N(T \downarrow \tfrac{T}{2})$ is directly obtained by switching the roles of $x(t)$ and $x^{\text{step}}(t)$.
\end{proof}

In terms of backwards paths, the mid-time estimate formulates as follows:
\begin{thm} \label{thm-mid-time-estimate-paths}
	Let $N = \alpha T $ with $\alpha \in (0,1)$ and fix some $\eps\in (0, \tfrac{1}{3})$. Then, it holds for all $T$ large enough:
	\begin{equation}
	\Pb( | x_{N( T \downarrow {T \over 2})} ( \tfrac{T}{2}) - (1- 2 \sqrt{\alpha}) \tfrac{T}{2} | > \tfrac{3}{\sqrt{\alpha}} K T^{2/3}) \leq C e^{-c K},
	\end{equation}
	where $ 0 < K = \Or(T^\eps)$.
	The constants $C,c > 0$ exist uniformly for all large times $T$ and for $\alpha$ in a closed subset of $(0,1)$.
\end{thm}
\begin{proof}
This result is a direct consequence of Theorem~\ref{thm_midtimeestimate} and Lemma~\ref{Lemma_A.2}.
\end{proof}

With the mid-time estimate at hand, we now turn to the proof of Proposition~\ref{prop_right_fluctuations_backwards_path}.

\subsection{Proof of Proposition~\ref{prop_right_fluctuations_backwards_path}}

	We define \begin{equation} l(t) := (1- 2 \sqrt{\gamma})t, \ m := \min\{ n \in \N | 2^{-n} T \leq T^{1/2} \}\end{equation} and choose $u_1 < u_2 < \dots$ by
	$ u_1 := \frac{K}{10} \textrm{ and } u_n - u_{n-1} = u_1 2^{-(n-1)/2}$. We divide the interval $[0,T]$ up into $2^n$ subintervals for each $n \in \{1, \dots , m \}$ and denote
	\begin{equation}
	\begin{aligned}
	A_n & := \{ x_{N(T \downarrow k 2^{-n} T)}(k 2^{-n} T) \leq l(k 2^{-n} T) + u_n T^{2/3}, 0 \leq k \leq 2^n \}, \\
	B_{n,k} & := \{ x_{N(T \downarrow k 2^{-n} T)}(k 2^{-n} T) > l(k 2^{-n} T) + u_n T^{2/3}\} , \ 0 \leq k \leq 2^n,
	\end{aligned}
	\end{equation}
	as well as
	\begin{equation}
	\begin{aligned}
	G & := \Big\{ \sup_{t \in [0,T]} \{ x_{N(T \downarrow t)}(t) - l(t) \} \leq K T^{2/3} \Big\}.
	\end{aligned}
	\end{equation}
	As $k$, we only consider integers. Our goal is to show $ \Pb(G^c) \leq C e^{-c K} $ for suitable constants $C,c > 0$. We define
	\begin{equation}
	F := \{ x_{N(T \downarrow k 2^{-m} T)} ( k 2^{-m}T) > l(k 2^{-m}T) + \tfrac{K}{2} T^{2/3} \textrm{ for some } k \in \{ 0, \dots, 2^m -1 \} \}.
	\end{equation}
	The first feature we need is the inclusion
	\begin{equation}
	F \subseteq \bigcup_{n=2}^m \bigcup_{k=1}^{2^n-1} (B_{n,k} \cap A_{n-1}) \cup \bigcup_{k=0}^2 B_{1,k}. \label{inclusionforE}
	\end{equation}
	Heuristically, \eqref{inclusionforE} implies that if the backwards path fluctuates to the right at one edge point of the subintervals of length $2^{-m}T$, then this either happens at time $0$ or at time $T$, or we find a decomposition of $[0,T]$ into larger dyadic subintervals such that at one of their inner edge points, the backwards path fluctuates to the right, while at the two adjacent edge points, it stays to the left of a suitable line.
	
	To obtain \eqref{inclusionforE}, it is essential that $u_n < \tfrac{K}{2}$ for all $n \in \N$. We make the following case distinction:
	If we have $ x_{N(T \downarrow 0)}(0) > u_1 T^{2/3}$, then $B_{1,0}$ occurs.
	If it holds $ x_{N}(T) > (1-2 \sqrt{\gamma}) T + u_1 T^{2/3}$, then $B_{1,2}$ takes place.
	
	Assuming $ x_{N(T \downarrow 0)}(0) \leq u_1 T^{2/3}$ and $x_{N}(T) \leq (1 - 2 \sqrt{\gamma}) T + u_1 T^{2/3}$, we see that there exists a maximal $n_0 \in \{0, \dots, m-1\}$ such that \begin{equation} x_{N(T \downarrow k 2^{-m} T)} ( k 2^{-m}T) > l(k 2^{-m}T) + u_{m-n_0} T^{2/3}\end{equation} is fulfilled by some $k = 2^{n_0} z$ with $z \in \{1, \dots, 2^{m-n_0}-1 \}$ odd. Then, $B_{m-n_0,z}$ occurs. If $n_0 = m-1$, then we have $B_{m-n_0,z} = B_{1,1}$. Else, we show that $A_{m-n_0-1}^c$ taking place would contradict the maximality of $n_0$ and conclude that $B_{m-n_0,z} \cap A_{m-n_0-1}$ occurs. This shows \eqref{inclusionforE}.
	
	Denoting
	\begin{equation}
	\begin{aligned}
	L & := \Big\{ \sup_{s \in [0,1]} | x_{N(T \downarrow (k+s) 2^{-m} T )}((k+s) 2^{-m} T ) - x_{N(T \downarrow k 2^{-m} T)}(k 2^{-m}T) - l(s 2^{-m} T) | \\ & \qquad \leq \tfrac{K}{2} T^{2/3}, 0 \leq k \leq 2^m-1 \Big\},
	\end{aligned}
	\end{equation}
	we have $F^c \cap L \subseteq G$ and the inclusion \eqref{inclusionforE} implies
	\begin{equation}
	\Pb(G^c) \leq \Pb(L^c) + \sum_{n=2}^m \sum_{k=1}^{2^n-1} \Pb(B_{n,k}\cap A_{n-1}) + \sum_{k=0}^2 \Pb( B_{1,k}).
	\end{equation}
	
	\paragraph*{Bound on $\Pb(L^c)$}	Both the number of jumps to the left and the number of jumps to the right of the backwards path $x_{N(t \downarrow \tau)}(\tau)$ are stochastically dominated by a Poisson process $\mathcal{P}$ with intensity $1$. Consequently, the same applies to the change of position of the path: stochastically,
	\begin{equation}
	\sup_{s \in [0,1]} | x_{N(T \downarrow (k+s) 2^{-m} T )}((k+s) 2^{-m} T ) - x_{N(T \downarrow k 2^{-m} T)}(k 2^{-m}T)|
	\end{equation}
	is bounded by
	\begin{equation} \mathcal{P}((k+1) 2^{-m} T) - \mathcal{P}(k 2^{-m} T).\end{equation}
	Exploiting stationarity of the increments of $\mathcal{P}$ and applying the exponential Chebyshev inequality, we derive \begin{equation} \Pb(L^c) \leq C e^{-c K T^{2/3}}.\end{equation}
	Here, it is essential that we consider increments over intervals of length $2^{-m} T \leq T^{1/2} \ll K T^{2/3}$.
	
	\begin{figure}[t!] \centering
		\includegraphics[scale=1]{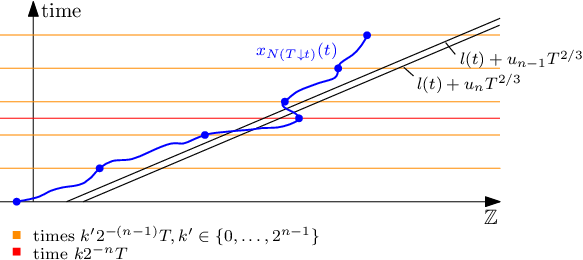}
		\caption{The event $B_{n,k} \cap A_{n-1}$ for $k \in \{1, \dots, 2^n-1\}$ odd.}
		\label{figure_intersection_A_B}
	\end{figure}
	
	\paragraph*{Bound on $\sum_{n=2}^m \sum_{k=1}^{2^n-1} \Pb(B_{n,k}\cap A_{n-1}) + \sum_{k=0}^2 \Pb( B_{1,k})$}
	For $ n \in \{2, \dots , m \} $ and $k \in \{1, \dots, 2^n-1\}$ even, we have $B_{n,k} \cap A_{n-1} = \varnothing$ and therefore $ \Pb( B_{n,k} \cap A_{n-1}) = 0$. Thus, assume that $k \in \{1 , \dots, 2^n-1\}$ is odd and suppose $B_{n,k} \cap A_{n-1}$ occurs.
	We denote \begin{equation} t_1 := (k-1)2^{-n}T, t_2 := (k+1) 2^{-n} T, t := k 2^{-n} T \label{eq_def_t1_t2_t} \end{equation} and obtain by $A_{n-1}$:
	\begin{align}
	x_{N(T \downarrow t_1)} (t_1) & \leq l(t_1) + u_{n-1} T^{2/3} =: x_1, \label{eq_def_x1} \\ x_{N(T \downarrow t_2)}(t_2) &\leq l(t_2) + u_{n-1} T^{2/3} =: x_2. \label{eq_def_x2}
	\end{align}
	The event $B_{n,k}$ gives $x_{N(T \downarrow t)}(t) > l(t) + u_n T^{2/3}$. Thus, we have
	\begin{equation}
	\Pb( B_{n,k} \cap A_{n-1}) \leq \Pb(x_{N(T \downarrow t)}(t) > l(t) + u_n T^{2/3}, x_{N(T \downarrow t_1)}(t_1) \leq x_1, x_{N(T \downarrow t_2)}(t_2) \leq x_2),
	\end{equation}
	see also Figure~\ref{figure_intersection_A_B}. Making use of this inequality, we derive a suitable bound on $\Pb( B_{n,k} \cap A_{n-1})$ in Proposition~\ref{prop_singleboundB_n,kA_n-1} and conclude in Corollary~\ref{cor_boundsumBnkAn1} that it holds
	\begin{equation}
	\sum_{n=2}^m \sum_{k=1}^{2^n-1} \Pb(B_{n,k} \cap A_{n-1}) \leq C e^{-c K}.
	\end{equation}
	In particular, the proof of Proposition~\ref{prop_singleboundB_n,kA_n-1} displays why we need to choose \mbox{$K = \Or(T^\eps) \subseteq o(T^{1/12})$}.
	
	Suitable bounds on $\Pb(B_{1,0}), \Pb(B_{1,1})$ and $\Pb(B_{1,2})$ are obtained by Lemma~\ref{Lemma 2.8}, see also Remark~\ref{rem_possible_extensions_auxiliary_lemmata}, by Theorem~\ref{thm-mid-time-estimate-paths} and lastly by Lemma~\ref{Lemma_A.2}.
	
	For all bounds, the constants $C,c > 0$ can be chosen uniformly for large $T$ and for $\gamma$ in a closed subset of $(0,1)$. Combining all bounds, we conclude \mbox{$\Pb(G^c) \leq C e^{-c K}$}.

\subsection{Application of a comparison inequality and the mid-time estimate}
It remains to bound
\begin{equation} \begin{aligned}
& \Pb( B_{n,k} \cap A_{n-1}) \\ & \leq \Pb(x_{N(T \downarrow t)}(t) > l(t) + u_n T^{2/3}, x_{N(T \downarrow t_1)}(t_1) \leq x_1, x_{N(T \downarrow t_2)}(t_2) \leq x_2), \end{aligned}
\end{equation}
with $t,t_1,t_2$ and $x_1,x_2$ defined in \eqref{eq_def_t1_t2_t}, \eqref{eq_def_x1}, \eqref{eq_def_x2} and $N = \gamma T$ with $\gamma \in (0,1)$.

By Lemma~\ref{lemma_alwaysstepparticle}, $x_{N(T \downarrow t_1)}(t_1) \leq x_1$ implies that there exists some $m_{t_2} \in \N$ such that $x_{N(T \downarrow t_2)}(t_2) = x_{m_{t_2}}^{\textrm{step}, x_1} (t_1, t_2).$
Having $ x_{N(T \downarrow t_2)}(t_2) \leq x_2$, this in particular gives the existence of a label $M \in \N$ such that
\begin{equation} x_{N(T \downarrow t_2)}(t_2) \leq x_M^{\textrm{step}, x_1}(t_1,t_2) \leq x_2\end{equation} is valid. We choose $M$ to be the minimal label for which the inequalities above are fulfilled. The next lemma gives us control on the value of $M$.

\begin{figure}[t!]
	\centering
	\includegraphics[scale=0.8]{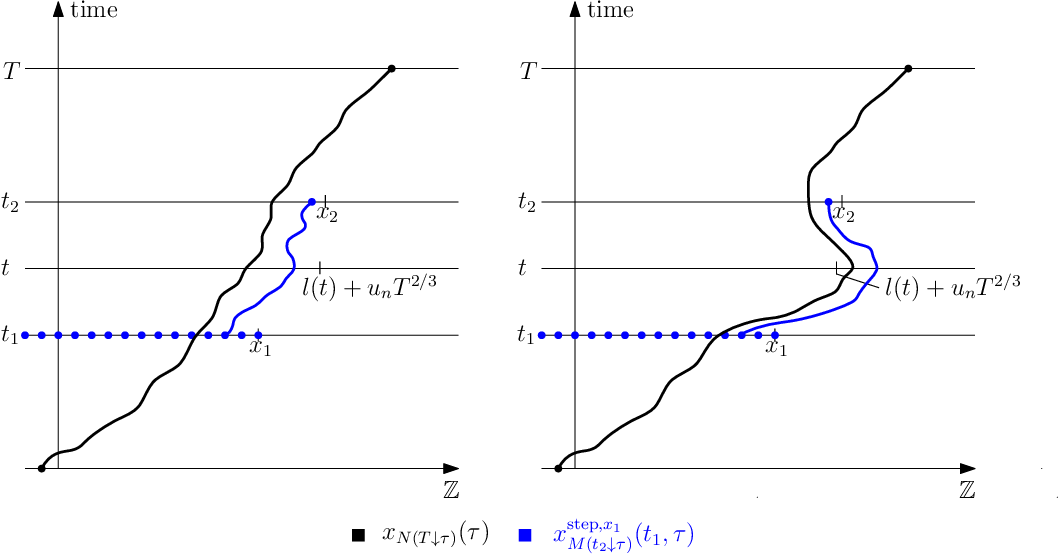}
	\caption{To bound $\Pb( x_{N(T \downarrow t)}(t) > l(t) + u_n T^{2/3}, x_{N(T \downarrow t_1)}(t_1) \leq x_1, x_{N(T \downarrow t_2)}(t_2) \leq x_2)$, we compare $x_{N(T \downarrow \tau)}(\tau)$ to $x^{\textrm{step},x_1}_{M(t_2 \downarrow \tau)}(t_1,\tau)$ and apply the mid-time estimate to the latter. Thereby, we can derive that the event in the figure to the right is very unlikely. } \label{figure_relation_TASEPs_step_IC_bound_unlikely_event}
\end{figure}

 \begin{lem} \label{lemma_minimal_M_region}
	Define $ \phi_n := \sqrt{\gamma} \tfrac{1}{20} (2^{n-1})^{1/6} \ \textrm{and} \ \tilde{M} := \gamma 2^{-n+1} T - \phi_n K (2^{-n+1}T)^{2/3}.$
	Then, it holds \begin{equation} \Pb( M > \tilde{M}) \geq 1 - C e^{-c (2^{n-1})^{1/2} K}\end{equation} for $T$ sufficiently large and $0 < K = \Or(T^\eps)$ with $\eps \in (0, \tfrac{1}{12})$ fixed. The constants $C, c > 0$ can be chosen uniformly for all times $T$ large enough, for all $n \in \{2, \dots, m \}$ and for $\gamma$ in a closed subset of $(0,1)$.
\end{lem}
\begin{proof}
	We observe $ \{ M > \tilde{M} \} = \{ x^{\textrm{step}, x_1}_{\tilde{M}} ( t_1, t_2) > x_2 \} $ and
	\begin{equation} \tfrac{\tilde{M}}{2^{-n+1}T} = \gamma - \Or(T^{\eps-1/12}) \end{equation} by our assumptions on $n$ and $K$ and the choice of $\phi_n$. By series expansion and since $x_{\tilde{M}}^{{\rm step},x_1}(t_1,\tau)$ has the same distribution as $x_{\tilde{M}}^{\rm step}(\tau-t_1) + x_1$, we obtain
	\begin{equation}
\begin{aligned}
	\Pb( M > \tilde{M})
	& =  \Pb(x^{\textrm{step}, x_1}_{\tilde{M}} ( t_1, t_2) > x_2)  =  \Pb(x^{\textrm{step}}_{\tilde{M}}(t_2 - t_1) > x_2 - x_1 ) \\
	& =  \Pb(x^{\textrm{step}}_{\tilde{M}}( 2^{-n+1}T) > (1- 2 \sqrt{\gamma}) 2^{-n+1}T ) \\
	& \geq \Pb\Bigl(x^{\textrm{step}}_{\tilde{M}}( 2^{-n+1}T) > \Bigl (1- 2 \sqrt{\tfrac{\tilde{M}}{2^{-n+1}T}} \Bigr ) 2^{-n+1} T - \tfrac{1}{\sqrt{\gamma}} \phi_n K ( 2^{-n+1} T)^{2/3}\Bigr ) .
\end{aligned}
	\end{equation}
By Lemma~\ref{Lemma_A.2}, the latter probability is bounded from below by
	 \begin{equation} 1 - C e^{-c (2^{n-1})^{1/6} K (2^{-n+1} T)^{1/3}}
	\geq 1 - C e^{-c (2^{n-1})^{1/2} K}.
	\end{equation}
	Since we have $2^{-n+1} T > T^{1/2}$ for all $n$, the constants exist uniformly for all times $T$ large enough, $n \in \{2, \dots, m\}$, and $\gamma$ in a closed subset of $(0,1)$. 	
\end{proof}

 With this last ingredient, we can derive the desired bound. Our key argument is depicted in Figure~\ref{figure_relation_TASEPs_step_IC_bound_unlikely_event}.

 \begin{prop} \label{prop_singleboundB_n,kA_n-1}
 	For $n \in \{2, \dots, m \}$ and $k \in \{1, \dots, 2^n-1 \}$ odd, it holds
 	\begin{equation} \Pb( B_{n,k} \cap A_{n-1}) \leq C e^{-c K (2^{n-1})^{1/6}}\end{equation} for all times $T$ large enough and $0 < K = \Or(T^{\eps})$ with $\eps\in (0, \tfrac{1}{12})$ fixed. The constants $C,c > 0$ can be chosen to be uniform in $T$, $n$, $k$ and for $\gamma$ in a closed subset of $(0,1)$.
 \end{prop}

 \begin{proof}
 	By our choice of the label $M$, we have
 	\begin{equation}\label{eq3.61}
 \begin{aligned}
 	& \Pb( B_{n,k} \cap A_{n-1}) \leq  \Pb(x_{N(T \downarrow t)}(t) > l(t) + u_n T^{2/3}, x_{N(T \downarrow t_1)}(t_1) \leq x_1, x_{N(T \downarrow t_2)}(t_2) \leq x_2) \\
 	& \leq  \Pb( x^{\textrm{step},x_1}_{M(t_2 \downarrow t)}(t_1, t) > l(t) + u_n T^{2/3}) \\
 	& \leq  \Pb( x^{\textrm{step}, x_1}_{\tilde{M}(t_2 \downarrow t)}(t_1, t) > l(t) + u_n T^{2/3}) + \Pb( M \leq \tilde{M}) \\
 	& \leq  \Pb( x^{\textrm{step}}_{\tilde{M}(t_2 - t_1 \downarrow {{t_2 - t_1} \over 2})} ( \tfrac{t_2 - t_1}{2}) > l(t) - x_1 + u_n T^{2/3} ) + \Pb( M \leq \tilde{M}),
\end{aligned}
 	\end{equation}
 where the second inequality holds by Lemma~\ref{lemma_prob_1_ordering_backwards_paths} and for the fourth inequality, we use
 \begin{equation} x^{\textrm{step}, x_1}_{\tilde{M}(t_2 \downarrow t)}(t_1, t) \overset{(d)}{=} x^{\textrm{step}}_{\tilde{M}(t_2 - t_1 \downarrow {{t_2 - t_1} \over 2})} ( \tfrac{t_2 - t_1}{2}) + x_1 \end{equation} since $t = \tfrac{t_1+t_2}{2}$. Furthermore, by Lemma~\ref{lemma_minimal_M_region}, the right hand side of \eqref{eq3.61} is bounded from above by	\begin{equation}
 \begin{aligned}
&\Pb\left( x^{\textrm{step}}_{\tilde{M}( 2^{-n+1} T \downarrow 2^{-n } T) } ( 2^{-n} T) > (1-2 \sqrt{ \gamma }) 2^{-n} T + \tfrac{K}{10} (2^{n-1})^{1/6} (2^{-n+1} T)^{2/3}\right) \\
 	& \quad + C e^{- c (2^{n-1} )^{1/2} K} \\
 	& \leq  \Pb \Bigl ( x^{\textrm{step}}_{\tilde{M}( 2^{-n+1} T \downarrow 2^{-n } T) } ( 2^{-n} T) > \Bigl (1- 2 \sqrt{\tfrac{\tilde{M}}{2^{-n+1}T}} \Bigr ) 2^{-n} T + \tfrac{K}{20} (2^{n-1})^{1/6} (2^{-n+1} T)^{2/3} \Bigr ) \\
 	&\quad + C e^{- c (2^{n-1} )^{1/2} K}
\end{aligned}
 	\end{equation}
 for all $T$ large enough. The last inequality is due to series expansion and the choice of $\phi_n$ in Lemma~\ref{lemma_minimal_M_region}.

 	To bound the probability in the last term, we apply Theorem~\ref{thm-mid-time-estimate-paths} with
 	\begin{equation}
 	T \mapsto 2^{-n+1} T,\quad \alpha \mapsto \gamma - \phi_n K (2^{-n+1}T)^{-1/3},\quad \tfrac{3}{\sqrt{\alpha}} K \mapsto \tfrac{K}{20} (2^{n-1})^{1/6}.
 	\end{equation}
 	Here, it is essential that we have $ \tfrac{K}{20} (2^{n-1})^{1/6} = \Or(T^\delta)$ with $\delta \in (0, \tfrac{1}{6})$, since this implies \begin{equation}\tfrac{K}{20} (2^{n-1})^{1/6} = \Or ((2^{-n+1}T)^{2\delta}) \textrm{ with } 2 \delta < \tfrac{1}{3}\end{equation} for all $n \in \{2, \dots, m\}$. In particular, the constants $C,c > 0$ exist uniformly for all $T$ large enough, $n \in \{2, \dots, m\}$, and $\gamma$ in a closed subset of $(0,1)$. We conclude
 	\begin{equation}
 	\Pb( B_{n,k } \cap A_{n-1}) \leq C e^{-c K (2^{n-1})^{1/6}} + C e^{-c K (2^{n-1})^{1/2}} \leq C e^{-c K (2^{n-1})^{1/6}}.
 	\end{equation}
 \end{proof}

\begin{cor} \label{cor_boundsumBnkAn1}
Let $\eps\in (0, \tfrac{1}{12})$ be fixed. Then, there exist constants $C,c > 0$ such that for all times $T$ large enough, it holds
\begin{equation}
\sum_{n=2}^m \sum_{k=1}^{2^n-1} \Pb(B_{n,k} \cap A_{n-1}) \leq C e^{-c K}
\end{equation}
for $ 1 \leq K = \Or(T^\eps)$. The constants can be chosen uniformly for large $T$ and for $\gamma$ in a closed subset of $(0,1)$.
\end{cor}
\begin{proof}
As seen in the proof of Proposition~\ref{prop_right_fluctuations_backwards_path}, we have $\Pb( B_{n,k} \cap A_{n-1}) = 0$ for even $k$. Combining this with Proposition~\ref{prop_singleboundB_n,kA_n-1} yields
\begin{equation}
\sum_{n=2}^m \sum_{k=1}^{2^n-1} \Pb(B_{n,k} \cap A_{n-1})
\leq \sum_{n=2}^m 2^n C e^{-c K (2^{n-1})^{1/6}}
\leq C e^{-c K}
\end{equation}
for times $T$ large enough, with uniform constants $C,c > 0$ for large $T$ and for $\gamma$ in a closed subset of $(0,1)$. The second inequality is obtained by adjusting the constants suitably, since an exponentially decreasing function dominates a polynomial prefactor. Further, we apply the bound on the exponential integral proven in~\cite{Gau59} and use $K \geq 1$.
\end{proof}

\subsection{Control of fluctuations to the left}\label{sect4.3}
Given Proposition~\ref{prop_right_fluctuations_backwards_path}, we can complete the proof of Proposition~\ref{prop_localization_backwards_path} with help of the particle-hole duality in TASEP. This technique has for example been utilized in~\cite{BBF21, BF22, Nej20}.
\begin{proof}[Proof of Proposition~\ref{prop_localization_backwards_path}]
We recall
\begin{equation}
\begin{aligned}
{\cal C}^r&=\{(x,t) \in \Z \times [0,T] \ | \ x\leq (1-2 \sqrt{\gamma}) t +K T^{2/3}\},\\
{\cal C}^l&=\{(x,t) \in \Z \times [0,T] \ | \ x\geq (1-2 \sqrt{\gamma}) t -K T^{2/3}\},
\end{aligned}
\end{equation}
and define
\begin{equation}
E_{{\cal C}^r}=\{x_{N(T\downarrow t)}(t)\leq (1-2 \sqrt{\gamma}) t +K T^{2/3}\textrm{ for all }0 \leq t \leq T\}.
\end{equation}
Recall that the process $\tilde{x}^r(t)$ with $\tilde{x}^r(0) = x(0)$ is constructed by the same Poisson events as $x(t)$ in ${\cal C}^r$ and with $({\cal C}^r)^c$ completely filled by holes. Likewise, $\tilde{x}^l(t)$ is the process with $\tilde{x}^l(0) = x(0)$, constructed by the same Poisson events as $x(t)$ in ${\cal C}^l$ and with $({\cal C}^l)^c$ completely filled by particles.
Then, Lemma~\ref{lemma_tagged_particle_position_only_depends_on_backwards_path} implies
\begin{equation}
\Pb(x_N(T)=\tilde{x}^r_N(T)|E_{{\cal C}^r})=1 \label{eq.4.72}
\end{equation}
and Proposition~\ref{prop_right_fluctuations_backwards_path} yields
	\begin{equation}
\Pb(E_{{\cal C}^r})\geq 1- C e^{-c K}.
\end{equation}

It remains to accomplish a similar statement for $\tilde{x}_N^l(T)$. For this, we compare the particle positions in the TASEP with step initial condition, described by $x(t)$, to the positions of the holes. We denote the process formed by them by $x^h(t)$, and we determine their labelling by setting $x^h_M(0) = M$ for $M \in \N$. The process $x^h(t)$ represents a TASEP with step initial condition shifted by $1$ which is mirrored at the origin, such that its particles (the holes) jump to the left.

For all $N, S \in \N$, it holds $x_N(T) = S - N + 1 = x_N(0) + S$ if and only if at time $T$, there are exactly $S$ holes to the left of $x_N(T)$. That is, we have
\begin{equation}
\{ x_N(T) = S-N+1 \} = \{ x^h_S(T) < x_N(T) < x^h_{S+1}(T) \}.
\end{equation}
Suppose $x^h_M(T) < x_N(T)$ is valid for some label $M \in \N$. Then, $x_N(T)$ is determined by the positions $\{x^h_k(T), k \geq M\}$ of the holes to the right of $x^h_M(T)$.

If $x^h_M(T)$ is independent of a family of Poisson events in the underlying graphical construction located in $({\cal C}^l)^c$,
then the same holds for $\{x^h_k(T), k \geq M\}$ and consequently for $x_N(T)$. We denote by $\tilde{x}^{l,h}(t)$ the process formed by the holes in the process $\tilde{x}^l(t)$ and define
\begin{equation}
E_{M,{\cal C}^l}=\{x^h_M(T)< x_N(T)\}\cap \{\forall\, 0 \leq t \leq T: x^h_{M(T\downarrow t)}(t)\geq (1-2 \sqrt{\gamma}) t -K T^{2/3}\}.
\end{equation}
In $\tilde{x}^{l,h}(t)$, the holes in ${\cal C}^l$ evolve by the same Poisson events as those in $x^h(t)$, but $({\cal C}^l)^c$ is completely filled by particles. Thus, combining the previous arguments with Lemma~\ref{lemma_tagged_particle_position_only_depends_on_backwards_path}, see also Remark~\ref{remark_on_lemma_3.1}, we obtain: if $E_{M,{\cal C}^l}$ occurs, then $\tilde{x}^{l,h}_k(T)=x^h_k(T)$ for all $k\geq M$ and thus also $\tilde{x}^l_N(T)=x_N(T)$. This means that
\begin{equation}
\Pb(x_N(T)=\tilde{x}^l_N(T)|E_{M,{\cal C}^l})=1. \label{eq.4.76}
\end{equation}

Setting $E=E_{{\cal C}^r}\cap E_{M,{\cal C}^l}$ and putting \eqref{eq.4.72} and \eqref{eq.4.76} together, we find
\begin{equation}
\Pb(x_N(T)=\tilde{x}^r_N(T)=\tilde{x}^l_N(T)|E)=1.
\end{equation}

It remains to bound $\Pb(E_{M,{\cal C}^l})$ from below for some suitable $M$.
We choose the label $M \in \N$ such that the event $\{ x^h_M(T) < x_N(T) \}$ has a high probability, but the macroscopic position of the backwards path $x^h_{M(T \downarrow t)}(t)$ is still within an $o(T^{2/3})-$neighbourhood of the macroscopic position of the backwards path $x_{N(T \downarrow t)}(t)$. Since $N = \gamma T$ with $\gamma \in (0,1)$, we choose
	$ M = (1- \sqrt{\gamma})^2 T - T^{1/2}$, see Figure~\ref{figure_left_fluctuations_backwards_path}.
\begin{figure}[t!]
		\centering
		\includegraphics[scale=1]{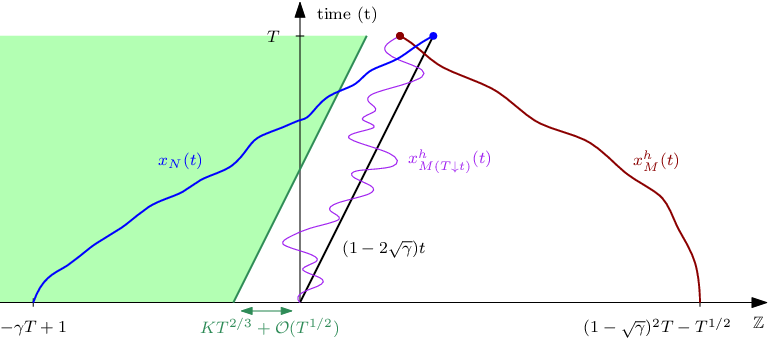}
		\caption{Controlling the fluctuations of the backwards path of $x^h_M(T)$ to the left, we observe that with high probability $x^h_M(T)$ is independent of the Poisson events in the green region. Then, the same applies to $x_N(T)$.  }
		\label{figure_left_fluctuations_backwards_path}
\end{figure}

Then, Lemma~\ref{Lemma_A.2} (with $\alpha = \gamma$ and $c_1(\alpha)s = T^{1/6}$) yields
	\begin{equation} \label{eq4.79}
	\Pb( x^h_M(T) < x_N(T)) = \Pb( x_N(T) \geq M - N +1 )
	\geq  1 - C e^{-c T^{1/6}}
	\end{equation}
	for all $T$ large enough.

Next, we want to apply Proposition~\ref{prop_right_fluctuations_backwards_path} to the backwards path of the hole $x^h_M(T)$. Since it jumps to the left, Proposition~\ref{prop_right_fluctuations_backwards_path} provides control of the fluctuations of its backwards path to the left as well. As already mentioned, $x^h(t)$ evolves like a TASEP with step initial condition mirrored at the origin. Specifically, for $M_\beta = \beta T$ with $\beta \in (0,1)$, Proposition~\ref{prop_right_fluctuations_backwards_path} implies
	\begin{equation}
	 \Pb\Bigl( \inf_{ t \in [0,T]} \{x^h_{M_\beta (T \downarrow t)}(t) + (1-2 \sqrt{\beta }) t \} < - K T^{2/3} \Bigr)
	 \leq  C e^{-c K} \label{ineq_appl_right_fluctuations_backwards_paths_for_holes}
	\end{equation}
	for constants $C,c > 0$ that are uniform for all $T$ large enough and for $\beta$ in a closed subset of $(0,1)$. In our case, we have $ \beta = (1-\sqrt{\gamma})^2 - T^{-1/2}$, which fulfils this condition for all $T$ large enough. By series expansion, we derive
	\begin{equation}
 (1-2 \sqrt{ \beta} )
	= - (1-2\sqrt{\gamma}) + \tfrac{1}{1-\sqrt{\gamma}} T^{-1/2} + \Or(T^{-1})
	\end{equation}
	and therefore
	\begin{equation}
	\Pb\Bigl( \inf_{ t \in [0,T]}\{x^h_{M(T \downarrow t)}(t) - ( 1 - 2 \sqrt{\gamma} )t \} < - K T^{2/3} \Bigr) \leq C e^{-c K}
	\end{equation}
	with uniform constants $C, c > 0$ for $T$ large enough and $\gamma$ in a closed subset of $(0,1)$. This concludes the proof of Proposition~\ref{prop_localization_backwards_path}.
\end{proof}

\section{Functional slow decorrelation}\label{SectSlowDec}
The fluctuations of particle positions at a macroscopic time $t$ are of order $t^{1/3}$. However, fluctuations along characteristic lines (macroscopic lines around which the backwards path fluctuates), at times differing on a mesoscopic scale, are the same up to $o(t^{1/3})$. This is called slow decorrelation phenomenon and will be needed in the proof of our main result. Since the starting formula \eqref{eq1.2} contains the full process, a finite-dimensional slow decorrelation result as in~\cite{Fer08,CFP10b} is not enough. We need to upgrade it to a functional slow decorrelation statement as first made in~\cite{CLW16}, although unlike them, we use tightness directly without a comparison to auxiliary processes.

In our case, we need to argue that what happens in a mesoscopic time at the beginning of the evolution is not creating relevant fluctuations as they are in $o(T^{1/3})$. For this reason, we first show a slow decorrelation result for the following rescaled tagged particle processes:
\begin{equation}\tilde{X}_T^i (\tau) = \frac{x_{\alpha T}(t) - \tilde{\mu}^i(\tau,T)}{- \tilde{c}_1^i T^{1/3}}
\end{equation}
and
\begin{equation}\tilde{X}_T^{\nu, i}(\tau) = \frac{ y^{x_i}_{\alpha T - {\alpha \over \alpha_i} T^\nu}(T^\nu, t) + \Bigl(1-2 \sqrt{\tfrac{\alpha}{\alpha_i}}\Bigr)T^\nu - \tilde{\mu}^{i}(\tau,T)}{- \tilde{c}_1^{i} T^{1/3}},
\end{equation}
where $x_i = x_{ {\alpha \over \alpha_i} T^\nu} (T^\nu)  $
and $t = \alpha_i T - \tilde{c}_2^{i} \tau T^{2/3}$ with $\tau \in [-\varkappa,\varkappa]$, for some fixed $\varkappa > 0$ and $\nu \in (0,1)$. The process $y^{x_i}(T^\nu,t)$ emerges from \eqref{estimate_Sep98_backwards_paths}.

We obtain the result by combining pointwise slow decorrelation and tightness.

\begin{prop} \label{prop_functional_slow_decorrelation}
	For $i \in \{0,1\}$ and all $\eps > 0$, it holds
\begin{equation}
	\lim_{T \to \infty} \Pb( | \tilde{X}_T^{\nu,i} (\tau) - \tilde{X}_T^{i}(\tau)| \geq \eps \ \text{ for some } \tau \in [- \varkappa,  \varkappa]) = 0.
\end{equation}
\end{prop}

\begin{proof}
	
	By \eqref{estimate_Sep98_backwards_paths}, we know:
\begin{equation}
	x_{\alpha T}(t) \leq x_{ {\alpha \over \alpha_i} T^\nu} (T^\nu) + y^{x_i}_{\alpha T - {\alpha \over \alpha_i} T^\nu}(T^\nu, t). \label{eq_pf_slow_decorr_concatenation}
\end{equation}
	Further, the laws of large numbers of the respective particle positions match up to $o(T^{1/3})$, since by series expansion we have
	\begin{equation}
	\begin{aligned}
	& \biggl ( 1 - 2 \sqrt{ \tfrac{ \alpha T - {\alpha \over \alpha_i} T^\nu}{t - T^\nu}} \biggr ) (t - T^\nu) \\
	= \ & \Bigl ( 1 - 2 \sqrt{ \tfrac{\alpha T}{t}} \Bigr ) t - \Bigl ( 1 - \sqrt{ \tfrac{\alpha T}{t}} - \tfrac{\alpha}{\alpha_i} \sqrt{ \tfrac{t}{\alpha T}} \Bigr )T^\nu + \Or(T^{2\nu - 5/3})  \\
	= \ & \Bigl ( 1 - 2 \sqrt{ \tfrac{\alpha T}{t}} \Bigr ) t - \Bigl ( 1 - 2 \sqrt{ \tfrac{\alpha}{\alpha_i}} \Bigr ) T^\nu + \Or(T^{\nu - 2/3}) \label{eq_pf_slow_decorr_lln}
	\end{aligned} \end{equation}
	with $T^{\nu - 2/3} = o(T^{1/3})$.  Thus, pointwise slow decorrelation (see \eqref{eq_pf_slow_decorr_pointwise_result} below) can be obtained by the usual method proposed in~\cite{CFP10b}, see also~\cite{BF22,Fer18,FN19}. By \eqref{eq_Lemma_A.2_converging_alpha}, for each fixed $\tau$ (recall that then $t=\Or(T)$) we have
		\begin{align}
	\tfrac{x_{\alpha T}(t) - \Bigl( 1 - 2 \sqrt{ \tfrac{\alpha T}{t}}\Bigr) t}{- \tilde{c}_1^{i} T^{1/3}} \Rightarrow \ &F_{2}, \\
	\tfrac{y^{x_i}_{\alpha T - {\alpha \over \alpha_i} T^\nu}(T^\nu,t) - \biggl ( 1 - 2 \sqrt{ \tfrac{ \alpha T - {\alpha \over \alpha_i} T^\nu}{t - T^\nu}} \biggr ) (t - T^\nu)}{- \tilde{c}_1^{i} T^{1/3} } \Rightarrow \ &F_{2}, \\
	\tfrac{x_{{\alpha \over \alpha_i}T^\nu}(T^\nu) - \bigl ( 1 - 2 \sqrt{ \tfrac{\alpha}{\alpha_i}} \bigr) T^\nu}{- \tilde{c}_1^{i} T^{\nu / 3}} \Rightarrow \ & F_{2}
	\end{align}	
as $T\to\infty$, where $F_2$ denotes the GUE Tracy-Widom distribution function. By \eqref{eq_pf_slow_decorr_concatenation} and \eqref{eq_pf_slow_decorr_lln}, we obtain
	\begin{equation}
	\begin{aligned}  \tfrac{x_{\alpha T}(t) - \Bigl( 1 - 2 \sqrt{ \tfrac{\alpha T}{t}}\Bigr) t}{- \tilde{c}_1^{i} T^{1/3}}  \geq \ & \tfrac{y^{x_i}_{\alpha T - {\alpha \over \alpha_i} T^\nu}(T^\nu,t) - \biggl ( 1 - 2 \sqrt{ \tfrac{ \alpha T - {\alpha \over \alpha_i} T^\nu}{t - T^\nu}} \biggr ) (t - T^\nu)}{- \tilde{c}_1^{i} T^{1/3} } \\ & + \tfrac{x_{{\alpha \over \alpha_i}T^\nu}(T^\nu) - \bigl ( 1 - 2 \sqrt{ \tfrac{\alpha}{\alpha_i}} \bigr) T^\nu}{- \tilde{c}_1^{i} T^{ 1/ 3}} + o(1).
	\end{aligned} \end{equation} 	
	The second summand on the right hand side converges to $0$ in probability since we scale by $T^{1/3}$ instead of $T^{\nu/3}$. Hence, Slutzky's theorem, see Theorem 13.18 of~\cite{Kle20}, implies that both sides of the inequality converge weakly to the same limit. By Lemma 4.1 of~\cite{BC09}, their difference converges to $0$ in probability. Explicitly, this yields the pointwise slow decorrelation result: let $t=\alpha_i T - \tilde{c}_2^{i} \tau T^{2/3}$ with $\tau\in [-\varkappa,\varkappa]$ fixed. Then, for each $\e>0$
	\begin{equation}
	\lim_{T \to \infty} \Pb\Bigl(\Bigl| x_{\alpha T}(t) - y^{x_i}_{\alpha T - {\alpha \over \alpha_i} T^\nu}(T^\nu,t) - \Bigl ( 1 - 2 \sqrt{ \tfrac{\alpha}{\alpha_i}} \Bigr) T^\nu \Bigr| \geq \eps T^{1/3} \Bigr) = 0. \label{eq_pf_slow_decorr_pointwise_result}
	\end{equation}	
From \eqref{eq_pf_slow_decorr_pointwise_result}, we obtain convergence of $(\tilde{X}^{i}_T(\tau) - \tilde{X}^{\nu,i}_T(\tau))$ to $0$ in the sense of finite-dimensional distributions. By Lemma~\ref{corollary_4.1_of_BBF22_tightness_weak_conv_Xtilde}, $(\tilde{X}_T^{i}(\tau))$ is tight. Since $x_i$ as well as $T^\nu$ do not depend on $\tau$, we have 	
	\begin{equation} (y^{x_i}_{\alpha T - {\alpha \over \alpha_i} T^\nu}(T^\nu,t)) \overset{(d)}{=} (x^{\text{step}}_{\alpha T - {\alpha \over \alpha_i} T^\nu}(t - T^\nu)).\end{equation}	
	Therefore, defining $\tilde{T} = T - \tfrac{1}{\alpha_i}T^\nu$, we find	
	\begin{equation}
	(\tilde{X}^{\nu,i}_T(\tau)) \overset{(d)}{=} \biggl (\tfrac{x^{\text{step}}_{\alpha \tilde{T}}(\alpha_i \tilde{T} - \tilde{c}_2^{i} c_{\tilde{T}} \tau  \tilde{T}^{2/3}) - \tilde{\mu}^i(c_{\tilde{T}} \tau, \tilde{T})}{- \tilde{c}_1^{i} \tilde{T}^{1/3}} \ d_{\tilde{T}} \biggr ) \overset{(d)}{=} (\tilde{X}_{\tilde{T}}^{i}(c_{\tilde{T}}\tau)d_{\tilde{T}})
	\end{equation}	
	with $c_{\tilde{T}} = T^{2/3} \tilde{T}^{-2/3} \to 1$ and $d_{\tilde{T}} = \tilde{T}^{1/3} T^{-1/3} \to 1$ as $T \to \infty$. In~\cite{BBF21}, it is shown that $(\tilde{X}_T^{i}(\tau))$ satisfies the conditions of Theorem 15.5 of~\cite{Bil68}. Then, tightness of the process is preserved under the modifications of multiplying $\tau$ as well as the process itself by converging, deterministic sequences $c_{\tilde{T}}, d_{\tilde{T}}$. This yields tightness of $(\tilde{X}^{\nu,i}_T(\tau))$. If two processes fulfil the conditions of Theorem 15.5 of~\cite{Bil68}, then the same holds true for their difference. Thus, $(\tilde{X}^{i}_T(\tau) - \tilde{X}^{\nu,i}_T(\tau))$ is tight and converges to $0$ in the sense of finite-dimensional distributions. By Theorem 15.1 of~\cite{Bil68}, this gives weak convergence of $(\tilde{X}^{i}_T(\tau) - \tilde{X}^{\nu,i}_T(\tau))$ to $0$ in the space of c\`adl\`ag functions on $[-\varkappa,\varkappa]$. As the limit is continuous, we also have weak convergence with respect to the uniform topology. This concludes the proof of Proposition~\ref{prop_functional_slow_decorrelation}.	
\end{proof}

Proposition~\ref{prop_functional_slow_decorrelation} compares $(\tilde{X}_T^{i}(\tau))$ to the rescaled tagged particle process $(\tilde{X}^{\nu,i}_T(\tau))$ which, in law, evolves in a TASEP with the usual step initial condition (that is, with rightmost particle initially at zero) starting at time $T^\nu$. Still, comparing the processes for $i \in \{0,1\}$, it is useful to take the shift by $x_i$ in the initial condition into account again. This is due to the fact that $y^{x_i}(T^\nu,t)$, $i \in \{0,1\}$ are not coupled by basic coupling: when ${\cal P}_x$ describes the jump attempts of $x(t)$ at site $x \in \Z$, then it depicts the jump attempts of $y^{x_i}(T^\nu,t)$ at site $x-x_i$. For this reason, we define
\begin{equation}
\hat{X}_T^{\nu,i}(\tau) = \frac{x^{\textup{step},x_i}_{\alpha T - {\alpha \over \alpha_i} T^\nu}(T^\nu,t) - \tilde{\mu}^i(\tau,T)}{-\tilde{c}_1^{i} T^{1/3}}
\end{equation}
and rephrase Proposition~\ref{prop_functional_slow_decorrelation} as follows:

\begin{cor} \label{cor_functional_slow_decorrelation}
For all $\eps > 0$, it holds
\begin{equation}
	\lim_{T \to \infty} \Pb( | \hat{X}_T^{\nu,i} (\tau) - \tilde{X}_T^{i}(\tau)| \geq \eps \ \text{ for some } \tau \in [- \varkappa,  \varkappa], i \in \{0,1\}) = 0.
\end{equation}
\end{cor}

\begin{proof}
This is a direct consequence of Proposition~\ref{prop_functional_slow_decorrelation} and the fact that the fluctuations of $x_i$ are of order $T^{\nu / 3} = o(T^{1/3})$.
\end{proof}

\section{Proof of the main results} \label{section_proof_of_the_main_results}
In determining the limit distribution of tagged particle fluctuations in TASEP with step initial condition and wall constraint, the crucial difference between the
cases of one or several wall influences is that in the latter, asymptotic independence of the suprema of the processes $(\tilde{X}_T^{i}(\tau))_{\tau \in [-\varkappa,\varkappa]}, i \in \{0,1\}$, is needed.

\begin{prop} \label{prop_on_asymptotic_independence_main_result_chapter4}
	For any $\varkappa > 0$ and all $s_0,s_1 \in \R$, it holds
	\begin{equation}
\begin{aligned}
	& \lim_{ T \to \infty} \Pb\Bigl( \sup_{\tau \in [-\varkappa,\varkappa]} \{\tilde{X}_T^{i}(\tau) + \tau^2 - g_T^{i}(\tau)\} \leq s_i , i \in \{0,1\} \Bigr) \\
 	& =  \prod_{i\in\{0,1\}}\lim_{T\to\infty} \Pb\Bigl( \sup_{\tau \in [-\varkappa,\varkappa]} \{\tilde{X}_T^{i}(\tau) + \tau^2 - g_T^{i}(\tau)\}\leq s_i \Bigr ).
\end{aligned}
	\end{equation}
\end{prop}

 Slow decorrelation tells us that the fluctuations happening at times smaller than $T^\nu$ are irrelevant for the asymptotic behaviour of $x_{\alpha T}(t)$, see also~\cite{BF22}. In particular, we can identify $(\tilde{X}_T^{i}(\tau))$ with $(\hat{X}^{\nu,i}_T(\tau))$ and show asymptotic independence of the particle positions \begin{equation}
 \left(x^{\textup{step},x_i}_{\alpha T - {\alpha \over \alpha_i} T^\nu}(T^\nu,t)\right)_{t \in [t_l^i,t_r^i]}, \ \ i \in \{0,1\} \label{eq_y_x_i_asympt_indep}
 \end{equation}
 instead, where $t_l^i = \alpha_i T - \tilde{c}_2^{i} \varkappa T^{2/3}, \ t_r^i = \alpha_i T + \tilde{c}_2^{i} \varkappa T^{2/3}$. Here, it is essential to choose $\nu \in (\tfrac{2}{3},1)$.

 The first step is to recall that for fixed $i \in \{0,1\}$,
 \begin{equation}
 (x^{\textup{step},x_i}(T^\nu,t)) \overset{(d)}{=} (x^{\text{step}}(t-T^\nu)+x_i),
 \end{equation}
 and to extend Proposition~\ref{prop_localization_backwards_path} from fixed times to the time intervals $[t_l^i, t_r^i]$. For this, we define regions\footnote{For notational simplicity, we do not restrict the time intervals as in \eqref{eq4.1} because we consider particles at different times here. However, it is still enough to construct the processes up to a finite time (like $T$), so we could replace $\infty$ by it in the definition of the regions.}
 \begin{equation}
 \begin{aligned}
  {\cal C}^r_i = & \Bigl \{ (x,t) \in  \Z \times [0,\infty) \ \Bigl| \ x \leq \left( 1-2 \sqrt{\tfrac{\alpha}{\alpha_i}} \right) t + K T^{2/3} \Bigr \}, \\
  {\cal C}^l_i = & \Bigl \{ (x,t) \in  \Z \times [0,\infty) \ \Bigl| \ x \geq \left( 1-2 \sqrt{\tfrac{\alpha}{\alpha_i}} \right) t - K T^{2/3} \Bigr \}
 \end{aligned}
 \end{equation}
 and TASEPs $\tilde{x}^{r,i}(t), \tilde{x}^{l,i}(t)$ analogously as in Proposition~\ref{prop_localization_backwards_path}: the processes $x^{\text{step}}(t), \tilde{x}^{r,i}(t), \tilde{x}^{l,i}(t)$ have the same (step) initial condition and share their jump attempts, except that $\tilde{x}^{r,i}(t)$ has density $0$ in $({\cal C}_i^r)^c$ and $\tilde{x}^{l,i}(t)$ has density $1$ in $({\cal C}_i^l)^c$. We set $N_i = \alpha T - \tfrac{\alpha}{\alpha_i} T^\nu$.

 \begin{prop} \label{prop_localization_backwards_paths_asympt_indep_part}
Fix $\nu \in (\tfrac{2}{3},1)$ and $ \eps \in (0, \tfrac{1}{12})$. Then, there exists an event $E_i$ such that
\begin{equation}
\Pb( (x^{\textup{step}}_{N_i}(t-T^\nu))_{t \in [t_l^i, t_r^i]} = (\tilde{x}^{r,i}_{N_i}(t-T^\nu))_{t \in [t_l^i, t_r^i]} = (\tilde{x}^{l,i}_{N_i}(t-T^\nu))_{t \in [t_l^i, t_r^i]} | E_i ) = 1
\end{equation}
and, for $T$ large enough,
\begin{equation}
\Pb(E_i) \geq 1 - C e^{-c K},
\end{equation}
where $ K_0 \leq K = \Or(T^\eps)$ for some constant $K_0 = K_0(\varkappa) > 0$. The constants $K_0,C,c > 0$ can be chosen uniformly for all times $T$ large enough.
\end{prop}

\begin{proof}
We write $N = N_i, t_l = t_l^i, t_r = t_r^i$ and define the events
\begin{align}
E_r = \Bigl \{ \forall s \in [0,t-T^\nu], t \in [t_l,t_r]: \ x^{\text{step}}_{N(t - T^\nu \downarrow s)}(s) \leq \Bigl(1 - 2 \sqrt{ \tfrac{\alpha}{\alpha_i}} \Bigr) s + K T^{2/3} \Bigr \}
\end{align}
and
\begin{equation} \begin{aligned}
E_l = & \Bigl \{ \forall s \in [0,t-T^\nu], t \in [t_l,t_r]: \ x^{\text{step},h}_{M_l(t-T^\nu \downarrow s)}(s) \geq \Bigl ( 1 - 2 \sqrt{\tfrac{\alpha}{\alpha_i}} \Bigr ) s - K T^{2/3} \Bigr \} \\ & \cap  \{ \forall t \in [t_l,t_r]: \  x^{\text{step}}_N(t - T^\nu) > x^{\text{step},h}_{M_l}(t - T^\nu) \},
\end{aligned} \end{equation}
with $M_l \in \N$ chosen below and $x^{\text{step},h}(t)$ denoting the process formed by the holes in $x^{\text{step}}(t)$. Almost surely, by Lemma~\ref{lemma_tagged_particle_position_only_depends_on_backwards_path} and Remark~\ref{remark_on_lemma_3.1}, $E_i = E_r \cap E_l$ implies $x_N^{\text{step}}(t-T^\nu) = \tilde{x}^{r,i}_{N}(t-T^\nu) = \tilde{x}^{l,i}_{N}(t-T^\nu)$ for all $t \in [t_l,t_r]$. This follows by the same arguments as in the proof of Proposition~\ref{prop_localization_backwards_path}.

It remains to show $\Pb(E_i^c) \leq C e^{-c K}$. Series expansion and the fact that $x^{\text{step}}_{N(t_1 \downarrow s)}(s) \leq x^{\text{step}}_{N(t_2 \downarrow s)}(s)$ whenever $t_1 \leq t_2$ yield
\begin{equation}
\begin{aligned}
 \Pb(E_r^c)
\leq \ &  \Pb\Bigl(\exists s \in [0,t_r - T^\nu]: \ x^{\text{step}}_{N(t_r - T^\nu \downarrow s)}(s) \geq \Bigl(1 - 2 \sqrt{ \tfrac{\alpha}{\alpha_i}} \Bigr) s + K T^{2/3} \Bigr ) \\
\leq \ &  \Pb\Bigl(\exists s \in [0,t_r - T^\nu]: \ x^{\text{step}}_{N(t_r - T^\nu \downarrow s)}(s) \geq \Bigl(1 - 2 \sqrt{ \tfrac{N}{t_r - T^\nu}} \Bigr) s + (K - \tilde{c} \varkappa) T^{2/3} \Bigr ) \label{eq.6.7}
\end{aligned} \end{equation}
for some constant $\tilde{c} > 0$ and $T$ large enough. By Proposition~\ref{prop_right_fluctuations_backwards_path}, the last probability can be bounded by $Ce^{-cK}$ for $K$ large enough.

In $E_l$, we set
\begin{equation}
M_l = \Bigl ( 1 - \sqrt{\tfrac{N}{t_l - T^\nu}}\Bigr)^2(t_l - T^\nu) - T^{1/2}.
\end{equation}
Then, $x^{\text{step}}_N(t_l - T^\nu) > x^{\text{step},h}_{M_l}(t_l - T^\nu)$ is equivalent to $x^{\text{step}}_N(t - T^\nu) > x^{\text{step},h}_{M_l}(t - T^\nu)$ for all $t \in [t_l,t_r]$ since once a particle jumped over a hole, their paths will not cross again. As in the proof of Proposition~\ref{prop_localization_backwards_path} (see \eqref{eq4.79} and use that $t_l-T^\nu$ is a macroscopic time), Lemma~\ref{Lemma_A.2} yields
\begin{equation}
 \Pb ( x^{\text{step}}_N(t_l - T^\nu) < x^{\text{step},h}_{M_l}(t_l - T^\nu) ) \leq C e^{-c T^{1/6}}.
\end{equation}
By the same strategy as for \eqref{eq.6.7}, we deduce
\begin{equation}
\begin{aligned}
& \Pb \Bigl( \exists s \in [0,t-T^\nu], t \in [t_l,t_r]: \ x^{\text{step},h}_{M_l(t-T^\nu \downarrow s)}(s) \leq \Bigl ( 1 - 2 \sqrt{\tfrac{\alpha}{\alpha_i}} \Bigr ) s - K T^{2/3} \Bigr ) \\
\leq \ &  \Pb \Bigl( \exists s \in [0,t_r-T^\nu]: \ x^{\text{step},h}_{M_l(t_r-T^\nu \downarrow s)}(s) \leq \Bigl ( 1 - 2 \sqrt{\tfrac{\alpha}{\alpha_i}} \Bigr ) s - K T^{2/3} \Bigr ) \\
\leq \ & \Pb \Bigl( \exists s \in [0,t_r-T^\nu]: \ x^{\text{step},h}_{M_l(t_r-T^\nu \downarrow s)}(s) + \Bigl ( 1-2 \sqrt{ \tfrac{M_l}{t_r - T^\nu}} \Bigr ) s \leq - (K-\tilde{c} \varkappa) T^{2/3} \Bigr )
\end{aligned}
\end{equation}
for some constant $\tilde{c} > 0$ and $T$ large enough. As stated in the proof of Proposition~\ref{prop_localization_backwards_path} (see \eqref{ineq_appl_right_fluctuations_backwards_paths_for_holes}), the last probability can be bounded by $C e^{- c K}$ by Proposition~\ref{prop_right_fluctuations_backwards_path}, for $K$ large enough. Thus, $\Pb (E_l^c) \leq C e^{-c K}$.
\end{proof}

 Clearly, the regions ${\cal C}_0^r$ and ${\cal C}_1^l$ are not disjoint. Still, considering \eqref{eq_y_x_i_asympt_indep}, we need to take the shift of the initial condition by $x_i$ into account.

 For a better understanding of the overall picture, we illustrate in Figure~\ref{figure_slow_decorrelation_asymptotic_independence} how we would argue if we localized the backwards paths in TASEP with step initial condition completely.
 \begin{figure}[t!]
	\centering
	\includegraphics[scale=0.9]{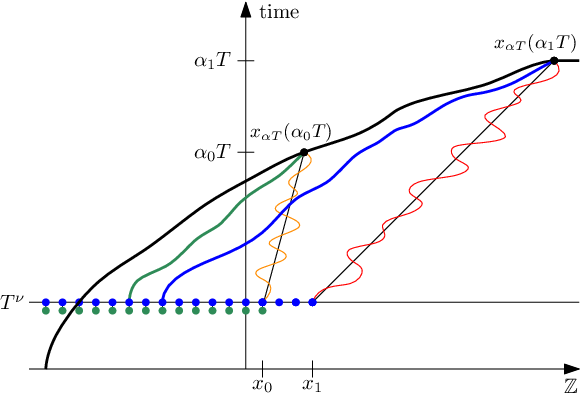}
	\caption{
		This image shows \eqref{eq_slow_decorr_x_equals_xi_plus_y} for $t = \alpha_0 T$ and $t = \alpha_1 T$, under the assumption that we localized the backwards paths starting at $x^{\text{step},x_i}_{N_i}(T^\nu,t)$ completely. The evolution of $x_{\alpha T}(t)$ corresponds to the black line and those of $x^{\text{step},x_i}_{N_i}(T^\nu,t)$ correspond to the green respectively blue lines. Their backwards paths starting at times $\alpha_i T$, depicted by the orange respectively red lines, are contained within $\Or(T^{2/3})-$neighbourhoods of $\Bigl(1 - 2 \sqrt{ \tfrac{\alpha}{\alpha_i}}\Bigr)t $ and do not meet with probability converging to $1$ because the distance of $x_0$ and $x_1$ is of order $T^\nu \gg T^{2/3}$.}
	\label{figure_slow_decorrelation_asymptotic_independence}
\end{figure}
Though in Section~\ref{sectLocal} and Proposition~\ref{prop_localization_backwards_paths_asympt_indep_part}, we rather delimited the regions that tagged particle positions depend upon, this is the implicit idea behind our argumentation. Corollary~\ref{cor_functional_slow_decorrelation} implies
 \begin{equation}
x_{\alpha T}(t) \simeq x^{\text{step},x_i}_{N_i}(T^\nu,t) + o(T^{1/3}) \ \text{ for all } t \in [t_l^i, t_r^i], i \in \{0,1\} \label{eq_slow_decorr_x_equals_xi_plus_y}
\end{equation}
and the three processes are coupled by basic coupling. In particular, asymptotic independence of $(x_{\alpha T}(t), t \in [t_l^{i},t_r^{i}])$ for $i \in \{0,1\}$, subject to $T^{1/3}-$scaling, could be obtained from Lemma~\ref{lemma_tagged_particle_position_only_depends_on_backwards_path} by localizing the backwards paths starting at the right hand side of \eqref{eq_slow_decorr_x_equals_xi_plus_y} in disjoint regions with probability converging to $1$. If Proposition~\ref{prop_localization_backwards_paths_asympt_indep_part} localized backwards paths directly, this could be done by shifting ${\cal C}_i^l \cap {\cal C}_i^r$ and estimating $x_i$. Choosing $K T^{2/3} \ll T^\nu$, the resulting regions would be disjoint indeed.

 We return to rigorous arguments and derive asymptotic independence of \eqref{eq_y_x_i_asympt_indep} from Proposition~\ref{prop_localization_backwards_paths_asympt_indep_part}. For this, we define
 \begin{equation}
 \begin{aligned}
 {\cal D}_i^r = & \Bigl \{ (x,t) \in \Z \times [T^\nu,\infty) \ \Bigl | \ x \leq \Bigl (1 - 2 \sqrt{ \tfrac{\alpha}{\alpha_i}} \Bigr) t + 2KT^{2/3} \Bigr \}, \\
 {\cal D}_i^l = & \Bigl \{ (x,t) \in \Z \times [T^\nu,\infty) \ \Bigl | \ x \geq \Bigl (1 - 2 \sqrt{ \tfrac{\alpha}{\alpha_i}} \Bigr) t - 2KT^{2/3} \Bigr \}.
  \end{aligned}
 \end{equation}

\begin{cor} \label{corollary_asympt_indep_prior}
Fix $\nu \in (\tfrac{2}{3},1)$ and $ \eps \in (0, \tfrac{1}{12}) \cap (0, \nu - \tfrac{2}{3})$. Then, there exists an event $E_i$ such that
\begin{equation}
\Pb ((x^{\textup{step},x_i}_{N_i}(T^\nu,t))_{t \in [t_l^i,t_r^i]} = (\tilde{x}^{r,i}_{N_i}(T^\nu,t))_{t \in [t_l^i,t_r^i]} = (\tilde{x}^{l,i}_{N_i}(T^\nu,t))_{t \in [t_l^i,t_r^i]} \ | E_i ) = 1
\end{equation}
and, for $T$ large enough,
\begin{equation}
\Pb (E_i) \geq 1 - Ce^{-cK},
\end{equation}
where $K_0 \leq K = \Or(T^\eps)$ for some constant $K_0 = K_0(\varkappa) > 0$. The constants $K_0,C,c > 0$ can be chosen uniformly for all times $T$ large enough. The process $\tilde{x}^{r,i}(T^\nu,t)$ has density $0$ in $({\cal D}_i^r)^c$ and the process $\tilde{x}^{l,i}(T^\nu,t)$ has density $1$ in $({\cal D}_i^l)^c$.
\end{cor}

\begin{proof}
Given the Poisson events from the construction of $x(t)$ (and $x^{\text{step},x_i}(T^\nu,t)$), we define the process $\tilde{x}^{r,i}(T^\nu,t)$ by starting the process $\tilde{x}^{r,i}(t)$ at time $T^\nu$ from the step initial condition $x^{\text{step},x_i}(T^\nu,T^\nu)$. Then, we have
\begin{equation}
(x^{\text{step},x_i}(T^\nu,t) , \tilde{x}^{r,i}(T^\nu,t)) \overset{(d)}{=} (x^{\text{step}}(t-T^\nu)+x_i, \tilde{x}^{r,i}(t-T^\nu) + x_i)
\end{equation}
jointly, and $\tilde{x}^{r,i}(T^\nu,t)$ has density $0$ in the region
\begin{equation}
({\cal F}_i^r)^c =  \Bigl \{ (x+x_i,t)  \in \Z \times [T^\nu,\infty) \ \Bigl| \  x > \Bigl ( 1 - 2 \sqrt{ \tfrac{\alpha}{\alpha_i}} \Bigr ) (t-T^\nu) + K T^{2/3} \Bigr \}.
\end{equation}
The existence of an event with the desired properties follows from Proposition~\ref{prop_localization_backwards_paths_asympt_indep_part}. We denote the event $E_i$ from Proposition~\ref{prop_localization_backwards_paths_asympt_indep_part} by $\tilde{E}_i$ here. For
\begin{equation}
B_i = \Bigl \{ \Bigl|x_i - \Bigl ( 1 - 2 \sqrt{ \tfrac{\alpha}{\alpha_i}} \Bigr ) T^\nu \Bigr | \leq K T^{2/3} \Bigr \}
\end{equation}
(with $x_i = x_{{\alpha \over \alpha_i} T^\nu}(T^\nu)$), Lemma~\ref{Lemma_A.2} yields
\begin{equation}
\Pb(B_i) \geq 1 - C e^{-c K T^{(2-\nu)/3}}.
\end{equation}
Thus, we define $E_i=\tilde{E}_i \cap B_i$ and, given this event, obtain
\begin{equation}
({\cal D}_i^r)^c \subseteq({\cal F}_i^r)^c.
\end{equation}
For $\tilde{x}^{l,i}(T^\nu,t)$, the arguments are analogous.
\end{proof}

Since the processes $\tilde{x}^{r,i}(T^\nu,t)$, $\tilde{x}^{l,i}(T^\nu,t)$ are coupled by basic coupling for \mbox{$i \in \{0,1\}$}, Corollary~\ref{corollary_asympt_indep_prior} yields the asymptotic independence we were looking for.

\begin{cor} \label{corollary_asympt_indep}
In the setting of Corollary~\ref{corollary_asympt_indep_prior}, it holds ${\cal D}_0^r \cap {\cal D}_1^l = \varnothing$ for all $T$ large enough. In particular, for $T$ large enough and \emph{given the event $E_0 \cap E_1$}, the processes $\tilde{x}^{r,0}(T^\nu,t)$ and $\tilde{x}^{l,1}(T^\nu,t)$ are independent, and thus, $(x^{\textup{step},x_i}_{N_i}(T^\nu,t))_{t \in [t_l^i,t_r^i]}$, $i \in \{0,1\}$ are independent as well.
\end{cor}

\begin{proof}
The definitions of ${\cal D}_0^r$ and ${\cal D}_1^l$ and the choice $K T^{2/3} \ll T^\nu$ immediately imply ${\cal D}_0^r \cap {\cal D}_1^l = \varnothing$ for $T$ large enough. Since given $E = E_0 \cap E_1$, the processes $\tilde{x}^{r,0}(T^\nu,t)$ resp. $\tilde{x}^{l,1}(T^\nu,t)$ only depend on the randomness in ${\cal D}_0^r$ resp. ${\cal D}_1^l$, Corollary~\ref{corollary_asympt_indep_prior} yields, for any measurable subsets $A_i$ of $\mathbb{D}([t_l^i,t_r^i])$, $i \in \{0,1\}$,
\begin{equation}
\begin{aligned}
& \Pb ( (x^{\textup{step},x_i}_{N_i}(T^\nu,t))_{t \in [t_l^i,t_r^i]} \in A_i, i \in \{0,1\} |  E) \\
= \ & \Pb ( (\tilde{x}^{r,0}_{N_0}(T^\nu,t))_{t \in [t_l^0,t_r^0]} \in A_0, (\tilde{x}^{l,1}_{N_1}(T^\nu,t))_{t \in [t_l^1,t_r^1]} \in A_1 |  E) \\
= \ & \Pb ( (x^{\textup{step},x_0}_{N_0}(T^\nu,t))_{t \in [t_l^0,t_r^0]} \in A_0  |  E) \ \Pb ( (x^{\textup{step},x_1}_{N_1}(T^\nu,t))_{t \in [t_l^1,t_r^1]} \in A_1  |  E).
\end{aligned}
\end{equation}
\end{proof}

\begin{remark}
Choosing $K = K_T \to \infty$ as $T \to \infty$ in Corollary~\ref{corollary_asympt_indep}, the particle positions \eqref{eq_y_x_i_asympt_indep} are independent with probability converging to $1$.

If we only imposed $\alpha_1 T \geq \alpha_0 T + T^{2/3 + \sigma}$ as mentioned in Remark~\ref{remark_weaken_assumption}, then we would need $1 - \sigma + \eps < \nu$ in order to obtain $\mathcal{D}^r_0 \cap \mathcal{D}^l_1 = \varnothing$.
\end{remark}

The proofs below require the following result.

\begin{lem} \label{lemma_continuity_sup_Airy_minus_fct} Let $\mathcal{A}_2$ denote an Airy$_2$ process and let $I \subseteq \R$ be a bounded interval. If $g$ is a c\`adl\`ag and piecewise continuous function satisfying $g(\tau) \geq - M + \tfrac{\tau^2}{2}, \tau \in \R,$ for some constant $M \in \R$, then the random variable
	\begin{equation}
	\sup_{\tau \in I} \{\mathcal{A}_2(\tau) - g(\tau)\}
	\end{equation} is continuous.
\end{lem}

\begin{proof}
	For $g$ being continuous, this result has been proven in~\cite{CLW16}, refer to the equations (103), (104) and Remark 2. There, one can pass from compact intervals $I$ to compactly contained ones since $g$ is continuous and the Airy$_2$ process has continuous sample paths. We want to point out that $\sup_{\tau \in \overline{I}} \{\mathcal{A}_2(\tau) - g(\tau)\}$ is finite almost surely as the same already holds for $\sup_{\tau \in \overline{I}} \{\mathcal{A}_2(\tau)\}$: for each $M > 0$ and $n \in \N$, we have
	\begin{equation}
	\Pb\Bigl ( \sup_{\tau \in [-n,n]}\{\mathcal{A}_2(\tau)\} \leq M \Bigr) \geq \Pb\Bigl ( \sup_{\tau \in \R} \{\mathcal{A}_2(\tau) - \tau^2\} \leq M - n^2 \Bigr) = F_{1}(2^{2/3}(M-n^2)),
	\end{equation} and the latter converges to $1$ as $M \to \infty$.
	
	For $g$ being piecewise continuous, we decompose $I$ into finitely many intervals $I_j, j \in \{1, \dots, n\}$, on which $g$ is continuous. Then, the result is obtained by induction over $n$. For the induction step, observe that for example,
	\begin{equation}
	\begin{aligned}
	& \Pb\Bigl( \sup_{\tau \in I} \{\mathcal{A}_2(\tau) - g(\tau)\} \leq s + \delta \Bigr ) - \Pb \Bigl ( \sup_{\tau \in I} \{\mathcal{A}_2(\tau) - g(\tau)\} \leq s \Bigr ) \\
	\leq & \Pb\Bigl( \sup_{j \in \{1, \dots, n -1\}} \sup_{\tau \in I_j}\{\mathcal{A}_2(\tau) - g(\tau)\}\leq s + \delta \Bigr ) - \Pb \Bigl( \sup_{j \in \{1, \dots, n-1\}} \sup_{\tau \in I_j} \{\mathcal{A}_2(\tau) - g(\tau)\} \leq s \Bigr) \\
	\ & + \Pb \Bigl( \sup_{\tau \in I_{n}} \{\mathcal{A}_2(\tau) - g(\tau)\}\leq s + \delta \Bigr ) - \Pb \Bigl( \sup_{\tau \in I_{n}} \{\mathcal{A}_2(\tau) - g(\tau)\} \leq s \Bigr).
	\end{aligned}
	\end{equation}
	The first difference converges to zero by the induction hypothesis, and the second one by the result for continuous functions.
	
\end{proof}

By Corollary~\ref{cor_functional_slow_decorrelation} and Corollary~\ref{corollary_asympt_indep}, we can now prove Proposition~\ref{prop_on_asymptotic_independence_main_result_chapter4}.

 \begin{proof}[Proof of Proposition~\ref{prop_on_asymptotic_independence_main_result_chapter4}]
	 We fix an arbitrary $\eps> 0$ and some $\nu \in (\tfrac{2}{3},1)$. Further, we choose a sequence $K_T \to \infty$ as $T \to \infty$ such that $K_T T^{2/3} = \Or(T^\sigma)$ with $\sigma \in (\tfrac{2}{3},\nu)$, $\sigma - \tfrac{2}{3} < \tfrac{1}{12}$ and $K_T \geq K_0$. Then, Corollary~\ref{corollary_asympt_indep_prior}  yields
	\begin{equation}
	\lim_{T \to \infty} \Pb(E_T) = 1,
	\end{equation}
	where we write $E_T$ for $E_0 \cap E_1$ depending on $K_T$.
	For \begin{equation} A^T_\eps:= \{ | \hat{X}_T^{\nu,i}(\tau) - \tilde{X}_T^{i}(\tau)| \leq \eps\ \forall \tau \in [- \varkappa, \varkappa], i \in \{0,1 \} \},\end{equation}
	Corollary~\ref{cor_functional_slow_decorrelation} gives us
 \begin{equation}
 \lim_{T \to \infty} \Pb( A^T_\eps) = 1.
 \end{equation}
	Defining $f_0(\tau) = g_T^0(\tau) - \tau^2 + s_0, f_1(\tau) = g_T^1(\tau) - \tau^2 + s_1$ for arbitrary $s_0,s_1 \in \R$, we obtain:
	\begin{equation}
\begin{aligned}
	& \lim_{T \to \infty} \Pb( \tilde{X}_T^0(\tau) \leq f_0(\tau), \tilde{X}_T^1(\tau) \leq f_1(\tau)\, \forall | \tau | \leq \varkappa) \\
	& =  \lim_{T \to \infty} \Pb( \{ \tilde{X}_T^{0}(\tau) \leq f_0(\tau), \tilde{X}_T^{1}(\tau) \leq f_1(\tau)\, \forall | \tau | \leq \varkappa \} \cap A^T_\eps) \\
	&\leq \lim_{T \to \infty} \Pb( \{ \hat{X}_T^{\nu,0}(\tau) \leq f_0(\tau) + \eps, \hat{X}_T^{\nu,1}(\tau) \leq f_1(\tau) + \eps\,\forall |\tau| \leq \varkappa \} \cap A^T_\eps) \\
	& =  \lim_{T \to \infty} \Pb( \hat{X}_T^{\nu,0}(\tau) \leq f_0(\tau) + \eps, \hat{X}_T^{\nu,1}(\tau) \leq f_1(\tau)+ \eps\,\forall |\tau| \leq \varkappa | E_T) \\
	& =  \lim_{T \to \infty} \prod_{i \in \{0,1\}} \Pb( \hat{X}_T^{\nu,i}(\tau) \leq f_i(\tau) + \eps\,\forall | \tau | \leq \varkappa |E_T ) \\
	&\leq  \lim_{T \to \infty} \Pb( \tilde{X}_T^0(\tau) \leq f_0(\tau) + 2 \eps\,\forall | \tau | \leq \varkappa) \Pb( \tilde{X}_T^1(\tau) \leq f_1(\tau) + 2 \eps\,\forall |\tau| \leq \varkappa ).
\end{aligned}
	\end{equation}
In the second inequality we exploited slow decorrelation in order to replace $(\tilde{X}_T^{i}(\tau))$ by $(\hat{X}_T^{\nu,i}(\tau))$, which is possible by adding $\eps$ to the right hand sides in the probability. Afterwards, we used that Corollary~\ref{corollary_asympt_indep} implies independence of the positions \eqref{eq_y_x_i_asympt_indep} given the event $E_T$. The last inequality follows by applying slow decorrelation again to recover the original process $(\tilde{X}_T^{i}(\tau))$.
	Analogously, we find
	\begin{equation}
\begin{aligned}
	& \lim_{T \to \infty} \Pb( \tilde{X}_T^0(\tau) \leq f_0(\tau), \tilde{X}_T^1(\tau) \leq f_1(\tau)\, \forall |\tau| \leq \varkappa ) \\
	& \geq \lim_{T \to \infty} \Pb( \tilde{X}_T^0(\tau) \leq f_0(\tau) - 2 \eps\,\forall | \tau | \leq \varkappa) \Pb( \tilde{X}_T^1(\tau) \leq f_1(\tau) - 2 \eps\, \forall |\tau| \leq \varkappa ).
\end{aligned}
	\end{equation}
	In the proof of Theorem~\ref{thm_main_result}, with help of Lemma~\ref{lemma_continuity_sup_Airy_minus_fct}, we see that the weak limit of \begin{equation}\sup_{\tau \in [-\varkappa,\varkappa]} \{ \tilde{X}_T^{i}(\tau) + \tau^2 - g_T^{i}(\tau) \} \end{equation} is a continuous random variable. Hence, we can take $\eps\to 0$ in the upper and lower bounds above and obtain the claimed asymptotic independence.
\end{proof}
With help of Proposition~\ref{prop_on_asymptotic_independence_main_result_chapter4}, we can finally prove our main convergence result.
\begin{proof}[Proof of Theorem~\ref{thm_main_result}]
	Fix some $S \in \R$. Since we have $\xi < (1-2 \sqrt{\alpha})$, Lemma~\ref{Lemma_A.2} gives \begin{equation} \lim_{T \to \infty} \Pb(x_{\alpha T}(T) > \xi T - S T^{1/3}) = 1.\end{equation} Thus, Proposition~\ref{prop_3.1_of_BBF22} implies
	\begin{equation}
\begin{aligned}
&	\lim_{T \to \infty}  \Pb( x^f_{\alpha T}(T) \geq \xi T - S T^{1/3})  \\ &= \lim_{T \to \infty} \Pb( x_{\alpha T}(t) \geq \xi T - S T^{1/3} - f(T-t)\, \forall t \in [0,T]). \label{prob_we_want_to_take_limit_of}
\end{aligned}
	\end{equation}
	We can replace \qq{$>$} by \qq{$\geq$} here since the asymptotic behaviour remains the same.
	Given Assumption~\ref{assumption_for_main_result}, Lemma~\ref{proof_claim_in_pf_of_main_thm} states that the only relevant time regions are neighbourhoods of $\alpha_i T, i \in \{0,1\}$. There, we consider the rescaled tagged particle positions $(\tilde{X}_T^i(\tau))$. For $t = \alpha_i T - \tilde{c}_2^i \tau T^{2/3}$, by Assumption~\ref{assumption_for_main_result} (b) we get
	\begin{equation}
\begin{aligned}
	&x_{\alpha T}(t) \geq \xi T - S T^{1/3} - f(T-t) \textrm{ for all } t \in [(\alpha_i - \e )T, (\alpha_i + \e )T] \\
	\Leftrightarrow \quad&\tilde{X}_T^i(\tau) + \tau^2 - g_T^i(\tau) \leq S (\tilde{c}_1^i)^{-1} \textrm{ for all } | \tau | \leq \e (\tilde{c}_2^i)^{-1} T^{1/3}.
\end{aligned}
	\end{equation}
	Let $[- \varkappa, \varkappa]$ be an interval with $\varkappa \geq 1$ fixed as $T \to \infty$. Then, for $T$ large enough it holds
\begin{equation}
	[- \varkappa, \varkappa] \subseteq [-\e (\tilde{c}_2^0)^{-1} T^{1/3}, \e (\tilde{c}_2^0)^{-1} T^{1/3}] \cap [-\e (\tilde{c}_2^1)^{-1} T^{1/3}, \e (\tilde{c}_2^1)^{-1} T^{1/3}].
	\end{equation}
	By Assumption~\ref{assumption_for_main_result} (b), the sequences $(g_T^0)$ and $(g_T^1)$ converge uniformly on $[- \varkappa, \varkappa]$ to $g_0$ and $g_1$ respectively.
	Moreover, we obtain weak convergence of $(\tilde{X}_T^{i}(\tau) + \tau^2)$ to $(\mathcal{A}_2^{i}(\tau))$ in the space of c\`adl\`ag functions on compact intervals from Lemma~\ref{corollary_4.1_of_BBF22_tightness_weak_conv_Xtilde}, where $\mathcal{A}_2^{i}$ denotes an Airy$_2$ process. Therefore, on the interval $[- \varkappa, \varkappa]$, Theorem 4.4 of~\cite{Bil68} yields $ ( \tilde{X}_T^{i}(\tau) + \tau^2, g_T^{i}(\tau)) \Rightarrow (\mathcal{A}_2^{i}(\tau), g_i(\tau))$. In general, the pointwise addition of functions in $\mathbb{D}([- \varkappa, \varkappa])$ is not a continuous mapping~\cite[p. 123, Problem 3]{Bil68}. Still, since the limits are independent and the Airy$_2$ process has continuous sample paths, addition indeed preserves convergence in this case~\cite[Section 4]{Whi80}.
	Therefore, by the continuous mapping theorem, see Theorem 5.1 of~\cite{Bil68}, it holds
\begin{equation}
\tilde{X}_T^{i}(\tau) + \tau^2 - g_T^{i}(\tau)\Rightarrow \mathcal{A}^{i}_2(\tau) - g_i(\tau)
\end{equation}
in $\mathbb{D}([- \varkappa,\varkappa])$. It further yields
	\begin{equation}
\sup_{\tau \in [- \varkappa, \varkappa]} \{\tilde{X}_T^{i}(\tau) + \tau^2 - g_T^{i}(\tau)\} \Rightarrow \sup_{\tau \in [- \varkappa, \varkappa]} \{\mathcal{A}^{i}_2(\tau) - g_i(\tau)\}.
\end{equation}
	Since $g_i$ satisfies the conditions in Lemma~\ref{lemma_continuity_sup_Airy_minus_fct}, we obtain
	\begin{equation}
\begin{aligned}
&\lim_{ T \to \infty} \Pb \Bigl ( \sup_{\tau \in [-\varkappa, \varkappa]} \{\tilde{X}_T^{i}(\tau) + \tau^2 - g_T^{i}(\tau)\} \leq S ( \tilde{c}_1^{i})^{-1} \Bigr )\\
 &= \Pb \Bigl( \sup_{\tau \in [-\varkappa, \varkappa]} \{\mathcal{A}^{i}_2(\tau) - g_i(\tau)\}\leq S ( \tilde{c}_1^{i})^{-1} \Bigr) \label{eq_5.29}
\end{aligned}
\end{equation}
 for all $S \in \R$.

 Next, we observe
	\begin{equation} \lim_{\varkappa\to\infty} \Pb \Bigl( \sup_{ \tau \in [- \varkappa, \varkappa]} \{ \mathcal{A}_2^i(\tau) - g_i(\tau)\} \leq S ( \tilde{c}_1^{i})^{-1} \Bigr) = \Pb \Bigl( \sup_{\tau \in \R} \{\mathcal{A}_2^i(\tau) - g_i(\tau)\}\leq S ( \tilde{c}_1^{i})^{-1} \Bigr). \label{conv_distr_sup_a2_over_compact_interval_to_sup_over_R}
	\end{equation}
	Indeed, the condition $g_i(\tau) \geq - M + \tfrac{\tau^2}{2}$ for some constant $M \in \R$ and Proposition~2.13 (b) of~\cite{CLW16}, see also Proposition 4.4 of~\cite{CH11}, imply
	\begin{equation}
	\Pb \Bigl( \sup_{ \tau \not \in [- \varkappa, \varkappa]} \{\mathcal{A}_2^i(\tau) - g_i(\tau)\} > S ( \tilde{c}_1^{i})^{-1} \Bigr)
	\leq  e^{-c ( {1 \over 2} \varkappa^2 + S (\tilde{c}_1^{i})^{-1} - M )^{3/2}}
	\end{equation}
	for $\varkappa$ large enough, and the right hand side converges to $0$ as $\varkappa \to \infty$. This yields (\ref{conv_distr_sup_a2_over_compact_interval_to_sup_over_R}). For representations of the probabilities as Fredholm determinants, refer to Theorem 1.19 and Equation (3.4) of~\cite{QR16} and Theorem 2 of~\cite{CQR11}.

	Now, let \begin{equation} J := [0,T] \setminus ( [\alpha_0 T - \tilde{c}_2^0 \varkappa T^{2/3}, \alpha_0 T + \tilde{c}_2^0 \varkappa T^{2/3}] \cup [\alpha_1 T - \tilde{c}_2^1 \varkappa T^{2/3}, \alpha_1 T + \tilde{c}_2^1 \varkappa T^{2/3}] ) \end{equation}
	and define
	\begin{equation}
\begin{aligned}
	&E_{T,\varkappa}^i := \{ \tilde{X}_T^i(\tau) + \tau^2 - g_T^i(\tau) \leq S (\tilde{c}_1^{i})^{-1} \forall |\tau| \leq \varkappa \}, \ i \in \{0,1\}, \\
	&A_{T,\varkappa} := \{ x_{\alpha T}(t) \geq \xi T - S T^{1/3} - f(T-t) \forall t \in J \}.
\end{aligned}
	\end{equation}
	Then, it holds
	\begin{equation}
	(\ref{prob_we_want_to_take_limit_of})
	= \lim_{T \to \infty} \Pb( E_{T, \varkappa}^0 \cap E_{T, \varkappa}^1 \cap A_{T,\varkappa} ).
	\end{equation}
	Proposition~\ref{prop_on_asymptotic_independence_main_result_chapter4}, \eqref{eq_5.29} and \eqref{conv_distr_sup_a2_over_compact_interval_to_sup_over_R} yield
	\begin{equation}
\begin{aligned}
	& \lim_{\varkappa\to\infty} \lim_{T \to \infty} \Pb(E_{T,\varkappa}^0 \cap E_{T,\varkappa}^1)  \\
	& =  \lim_{\varkappa\to\infty} \lim_{T \to \infty} \Pb(E_{T,\varkappa}^0 ) \Pb( E_{T,\varkappa}^1)  \\
	& =  \Pb\Bigl( \sup_{\tau \in \R}\{\mathcal{A}_2^0(\tau) - g_0(\tau)\} \leq S ( \tilde{c}_1^{0})^{-1}\Bigr) \Pb\Bigl( \sup_{ \tau \in \R}\{\mathcal{A}_2^1(\tau) - g_1(\tau)\} \leq S ( \tilde{c}_1^{1})^{-1}\Bigr). \label{eq_pf_main_result_product}
\end{aligned}
	\end{equation}
	As the limit distribution above is a product measure, we can choose the Airy$_2$ processes $\mathcal{A}_2^{0}, \mathcal{A}_2^{1}$ to be independent, such that the joint distribution of \mbox{$\sup_{\tau \in \R} \{\mathcal{A}_2^{0}(\tau)-g_0(\tau)\}$} and $\sup_{\tau \in \R} \{\mathcal{A}_2^{1}(\tau)-g_1(\tau)\}$ takes this form. This is of no direct significance for our result, but it is meant to emphasize the asymptotic independence required to obtain the limit distribution.

	In Lemma~\ref{proof_claim_in_pf_of_main_thm} below the proof, we establish:
	\begin{equation}
	\lim_{ \varkappa \to \infty} \lim_{T\to\infty} \Pb(A_{T,\varkappa}) = 1. \label{Claim_in_pf_of_main_thm}
	\end{equation}
	Since (\ref{prob_we_want_to_take_limit_of}) is independent of $\varkappa$, \eqref{eq_pf_main_result_product} and \eqref{Claim_in_pf_of_main_thm} give
	\begin{equation}
\begin{aligned}
	&\lim_{T\to\infty} \Pb(x_{\alpha T}^f(T) \geq \xi T - S T^{1/3}) \\
	& =  \lim_{\varkappa\to\infty} \lim_{T\to\infty} \Pb( E_{T, \varkappa}^0 \cap E_{T, \varkappa}^1 \cap A_{T,\varkappa} ) \\
	& = \Pb\Bigl( \sup_{\tau \in \R} \{\mathcal{A}_2^0(\tau) - g_0(\tau)\}\leq S( \tilde{c}_1^0)^{-1}\Bigr) \Pb\Bigl( \sup_{\tau \in \R} \{\mathcal{A}_2^1(\tau) - g_1(\tau)\} \leq S( \tilde{c}_1^1)^{-1} \Bigr).
\end{aligned}
	\end{equation}
\end{proof}

Finally, we formulate the auxiliary lemma needed in the previous proof.
\begin{lem} \label{proof_claim_in_pf_of_main_thm}
	In the setting of the proof of Theorem~\ref{thm_main_result}, it holds \begin{equation} \lim_{\varkappa \to \infty} \lim_{T\to\infty} \Pb(A_{T,\varkappa}) = 1.\end{equation}
\end{lem}

\begin{proof}
	As this convergence is obtained analogously in the case of one wall influence in~\cite{BBF21}, we do not repeat the details here.
	
	The idea is to split the interval $J$ up into further subintervals. For the intervals $[0,\alpha T]$ and $[\alpha T, (\alpha + \delta)T]$ with some small $\delta > 0$, one utilizes Assumption~\ref{assumption_for_main_result} (a) and the fact that it holds $x_{\alpha T}(t) \geq - \alpha T$ for all times $t \geq 0$. For the remaining times outside of $\Or(T^{2/3+\sigma})-$neighbourhoods of $\alpha_0 T$ and $\alpha_1 T$, where $\sigma \in (0, \tfrac{1}{6})$, one can argue by Assumption~\ref{assumption_for_main_result} (a) and (b) as well as Lemma~\ref{Lemma_A.2}. Finally, inside of the $\Or(T^{2/3+\sigma})-$neighbourhoods, one utilizes Assumption~\ref{assumption_for_main_result} (b) and the comparison to stationary TASEP by Proposition~\ref{prop_comparison_increments_of_particles}. We refer to Lemma 4.11, Lemma 4.12, Lemma 4.13 and Lemma 4.14 of~\cite{BBF21}. \\
\end{proof}

\appendix

\section{Estimates for TASEP with step initial condition and stationary TASEP} \label{appendix_estimates}

Below, we shortly list some results pertaining to the distributions of tagged particle positions in TASEP with step initial condition and stationary TASEP.

\begin{lem}[Lemma A.2 of~\cite{BBF21}] \label{Lemma_A.2}
	Let $x(t)$ be a TASEP with step initial condition and $\alpha \in (0,1)$. Then, it holds
	\begin{equation}
	\lim_{T \to \infty} \Pb( x_{\alpha T}(T) \geq (1-2 \sqrt{\alpha})T - s c_1(\alpha) T^{1/3}) = F_{2}(s), \label{one_point_limit_GUE}
	\end{equation} where we define $c_1(\alpha) = \tfrac{(1-\sqrt{\alpha})^{2/3}}{\alpha^{1/6}}$.
	
	In addition, the following estimates on the lower and the upper tail can be derived:
	uniformly for all large times $T$ and for $\alpha$ in a closed subset of $(0,1)$, there exist constants $C,c > 0$ such that for all $T$ large enough, it holds
	\begin{equation}
	\Pb( x_{\alpha T}(T) \leq (1-2\sqrt{\alpha})T - s c_1(\alpha) T^{1/3}) \leq C e^{-c s} \ \ \textrm{for} \ s > 0 \label{firsttail}
	\end{equation}
	and
	\begin{equation}
	\Pb(x_{\alpha T}(T) \geq (1-2 \sqrt{\alpha})T + s c_1(\alpha) T^{1/3}) \leq C e^{-c s^{3/2}} \ \ \textrm{for} \ 0 < s = o(T^{2/3}). \label{secondtail}
	\end{equation}
\end{lem}
Considering the proof of \eqref{one_point_limit_GUE}, that is Theorem 1.6 of~\cite{Jo00b}, we notice that \eqref{one_point_limit_GUE} still holds true when $\alpha$ is replaced by a converging sequence $\alpha_T \to \alpha$: then, \begin{equation}\lim_{T \to \infty} \Pb( x_{\alpha_T T}(T) \geq (1-2 \sqrt{\alpha_T})T - s c_1(\alpha) T^{1/3}) = F_{2}(s).\label{eq_Lemma_A.2_converging_alpha} \end{equation}
The next result provides information on the location of backwards paths at time $0$.
\begin{lem}[Lemma 2.4 of~\cite{BBF21}] \label{Lemma 2.8} Let $x(t)$ be a TASEP with step initial condition, $N = \alpha T$ with $\alpha \in (0,1)$ and $t = T - \varkappa T^{2/3}$ with $\varkappa$ in a bounded subset of $\R$. Then, there exist constants $C, c > 0$, uniformly for $\alpha$ in a closed subset of $(0,1)$ and for all sufficiently large times $T$, such that \begin{equation}
\Pb( | x_{N(t \downarrow 0)}(0) | \geq K_1 T^{1/3} ) \leq C e^{-c K_1} \textrm{ for all } K_1 > 0 \ \textrm{and for } T \ \textrm{large enough}.
\end{equation}
\end{lem}
 Apart from TASEP with step initial condition, we also require some knowledge about the stationary TASEP.
\begin{lem}[Lemma A.3 of~\cite{BBF21}] \label{Lemma A.3} Let $x^\rho(t)$ denote a stationary TASEP with density $\rho \in (0,1)$ and set $N = \rho^2 T - 2 w \rho \chi^{1/3} T^{2/3}$ with $\chi = \rho(1-\rho)$. Then, it holds
	\begin{equation}
	\lim_{T\to\infty} \Pb( x_N^\rho(T) \geq (1-2\rho)T + 2 w \chi^{1/3} T^{2/3} - (1-\rho) \chi^{-1} s T^{1/3}) = F_{{\rm BR},w}(s), \label{one-point-limit_BR}
	\end{equation}
	where $F_{\textrm{BR},w}$ denotes the Baik-Rains distribution function with parameter $w \in \R$.
	In addition, uniformly for all sufficiently large times $T$ and for $\rho$ in a closed subset of $(0,1)$, there exist constants $C, c > 0$ such that for $T$ large enough, it holds
	\begin{equation}
	\label{thirdtail} \Pb(x^\rho_N(T) \leq (1 - 2 \rho)T + 2 w \chi^{1/3} T^{2/3} - (1-\rho) \chi^{-1} s T^{1/3}) \leq C e^{-cs}
	\end{equation}
	for $s > 0$ and
	\begin{equation}
	\label{forthtail} \Pb( x^\rho_N(T) \geq (1-2\rho)T + 2 w \chi^{1/3} T^{2/3} + (1-\rho) \chi^{-1} s T^{1/3}) \leq C e^{-c s^{3/2}}
	\end{equation}
	for $ 0 < s = o(T^{2/3})$.
\end{lem}
Lastly, we need to understand the asymptotic behaviour of the increment $x_N^{\rho}(T) - x^{\rho}_{N(T \downarrow 0)}(0)$:
\begin{lem}[Lemma 2.5 of~\cite{BBF21}] \label{Lemma 2.9} Let $x^\rho(t)$ denote a stationary TASEP with density $\rho \in (0,1)$. Then, uniformly in $T$ and for $\rho$ in a closed subset of $(0,1)$, there exist constants $C,c > 0$ such that for all $K > 0$ and any $N \in \mathbb{Z}$, it holds
	\begin{equation}
	\Pb( | x^\rho_N(T) - x^\rho_{N(T \downarrow 0)}(0) - (1-2\rho)T| \geq K T^{2/3}) \leq C e^{-c K}
	\end{equation}
	for all $T$ large enough.
\end{lem}
\begin{remark} \label{rem_possible_extensions_auxiliary_lemmata}
	In order to obtain Proposition~\ref{prop_comparison_increments_of_particles} and Proposition~\ref{prop_on_particle_distances_stationary} as stated, we observe that the following extensions of the previous estimates are admissible:
	\begin{itemize}
		\item The bound in Lemma~\ref{Lemma 2.8} also holds for $K_1 = o(T^{1/3})$, and with uniform constants for all $K_1 = \Or(T^\sigma)$ with $\sigma \in (0, \tfrac{1}{3})$. Indeed, one can also allow $\varkappa = \Or(T^\sigma)$.
		\item The estimates \eqref{thirdtail} and \eqref{forthtail} in Lemma~\ref{Lemma A.3} also hold with uniform constants for $w = o(T^{1/3})$ if we impose the additional constraint $s \geq \tilde{s}$ with $\tilde{s} \gg w^2$ as $T \to \infty$. We refer to Remark 3.11 of~\cite{Ger23}.
		\item In Lemma~\ref{Lemma 2.9}, we can also allow $K = o(T^{1/3})$. To obtain uniform constants for different values of $K$, we demand $K = \Or(T^\sigma)$ with $\sigma \in (0, \tfrac{1}{3})$.
	\end{itemize}
\end{remark}


\begin{thebibliography}{10}

\bibitem{AAV11}
G.~Amir, O.~Angel, and B.~Valk\'{o}, \emph{{The TASEP speed process}}, Ann.
  Probab. \textbf{39} (2011), 1205--1242.

\bibitem{AHR09}
O.~Angel, A.~Holroyd, and D.~Romik, \emph{The oriented swap process}, Ann.
  Probab. \textbf{37} (2009), 1970--1998.

\bibitem{BL13}
J.~Baik and Z.~Liu, \emph{{On the average of the Airy process and its time
  reversal}}, Electron. Commun. Probab. \textbf{18} (2013), 1--10.

\bibitem{BSS14}
R.~Basu, V.~Sidoravicius, and A.~Sly, \emph{{Last passage percolation with a
  defect line and the solution of the Slow Bond Problem}}, preprint,
  arXiv:1408.346 (2014).

\bibitem{BC09}
G.~{Ben Arous} and I.~Corwin, \emph{{Current fluctuations for TASEP: a proof of
  the Pr\"ahofer-Spohn conjecture}}, Ann. Probab. \textbf{39} (2011), 104--138.

\bibitem{Bil68}
P.~Billingsley, \emph{{Convergence of Probability Measures}}, Wiley ed., New
  York, 1968.

\bibitem{BB19}
A.~Borodin and A.~Bufetov, \emph{Color-position symmetry in interacting
  particle systems}, Ann. Probab. \textbf{49} (2021), 1607--1632.

\bibitem{BBF21}
A.~Borodin, A.~Bufetov, and P.L. Ferrari, \emph{{TASEP with a moving wall}},
  Ann. Inst. H. Poincar\'e Probab. Statist. \textbf{60} (2024), 692--720.

\bibitem{BFP06}
A.~Borodin, P.L. Ferrari, and M.~Pr{\"a}hofer, \emph{{Fluctuations in the
  discrete TASEP with periodic initial configurations and the Airy$_1$
  process}}, Int. Math. Res. Papers \textbf{2007} (2007), rpm002.

\bibitem{BFS07}
A.~Borodin, P.L. Ferrari, and T.~Sasamoto, \emph{{Transition between Airy$_1$
  and Airy$_2$ processes and TASEP fluctuations}}, Comm. Pure Appl. Math.
  \textbf{61} (2008), 1603--1629.

\bibitem{Buf20}
A.~Bufetov, \emph{{Interacting particle systems and random walks on Hecke
  algebras}}, preprint, arXiv:2003.02730 (2020).

\bibitem{BF22}
A.~Bufetov and P.L. Ferrari, \emph{{Shock fluctuations in TASEP under a variety
  of time scalings}}, Ann. Appl. Probab. \textbf{32} (2022), 3614--3644.

\bibitem{BF20}
O.~Busani and P.L. Ferrari, \emph{{Universality of the geodesic tree in last
  passage percolation}}, Ann. Probab. \textbf{50} (2022), 90--130.

\bibitem{CP15b}
E.~Cator and L.~Pimentel, \emph{On the local fluctuations of last-passage
  percolation models}, Stoch. Proc. Appl. \textbf{125} (2015), 879--903.

\bibitem{CFS16}
S.~Chhita, P.L. Ferrari, and H.~Spohn, \emph{{Limit distributions for KPZ
  growth models with spatially homogeneous random initial conditions}}, Ann.
  Appl. Probab. \textbf{28} (2018), 1573--1603.

\bibitem{CFP10b}
I.~Corwin, P.L. Ferrari, and S.~P{\'e}ch{\'e}, \emph{{Universality of slow
  decorrelation in KPZ models}}, Ann. Inst. H. Poincar\'e Probab. Statist.
  \textbf{48} (2012), 134--150.

\bibitem{CH11}
I.~Corwin and A.~Hammond, \emph{{Brownian Gibbs property for Airy line
  ensembles}}, Inventiones mathematicae \textbf{195} (2013), 441--508.

\bibitem{CLW16}
I.~Corwin, Z.~Liu, and D.~Wang, \emph{{Fluctuations of TASEP and LPP with
  general initial data}}, Ann. Appl. Probab. \textbf{26} (2016), 2030--2082.

\bibitem{CQR11}
I.~Corwin, J.~Quastel, and D.~Remenik, \emph{{Continuum statistics of the
  Airy$_2$ process}}, Comm. Math. Phys. (2012), online first.

\bibitem{Fer08}
P.L. Ferrari, \emph{{Slow decorrelations in KPZ growth}}, J. Stat. Mech.
  (2008), P07022.

\bibitem{Fer18}
P.L. Ferrari, \emph{{Finite GUE distribution with cut-off at a shock}}, J.
  Stat. Phys. \textbf{172} (2018), 505--521.

\bibitem{FN13}
P.L. Ferrari and P.~Nejjar, \emph{{Anomalous shock fluctuations in TASEP and
  last passage percolation models}}, Probab. Theory Relat. Fields \textbf{161}
  (2015), 61--109.

\bibitem{FN16}
P.L. Ferrari and P.~Nejjar, \emph{{Fluctuations of the competition interface in
  presence of shocks}}, ALEA, Lat. Am. J. Probab. Math. Stat. \textbf{14}
  (2017), 299--325.

\bibitem{FN19}
P.L. Ferrari and P.~Nejjar, \emph{{Statistics of TASEP with three merging
  characteristics}}, J. Stat. Phys. \textbf{180} (2020), 398--413.

\bibitem{FN24}
P.L. Ferrari and P.~Nejjar, \emph{The second class particle process at shocks},
  Stoch. Proc. Appl. \textbf{170} (2024), 104298.

\bibitem{FO17}
P.L. Ferrari and A.~Occelli, \emph{{Universality of the GOE Tracy-Widom
  distribution for TASEP with arbitrary particle density}}, Eletron. J. Probab.
  \textbf{23} (2018), 1--24.

\bibitem{FO18}
P.L. Ferrari and A.~Occelli, \emph{Time-time covariance for last passage
  percolation with generic initial profile}, Math. Phys. Anal. Geom.
  \textbf{22} (2019), 1.

\bibitem{Gal20}
P.~Galashin, \emph{Symmetries of stochastic colored vertex models}, Ann.
  Probab. \textbf{49} (2021), 2175--2219.

\bibitem{Gau59}
W.~Gautschi, \emph{{Some elementary inequalities relating to the gamma and
  incomplete gamma function}}, J. Math. Phys \textbf{38} (1959), 77--81.

\bibitem{Ger23}
S.~Gernholt, \emph{{Fluctuations of a tagged particle in TASEP under influence
  of a deterministic wall}}, Master's thesis, {Rheinische
  Friedrich-Wilhelms-Universität Bonn}, 2023.

\bibitem{GKM21}
L.~Gr{\"u}ne, T.~Kriecherbauer, and M.~Margaliot, \emph{{Random Attraction in
  the TASEP Model}}, SIAM J. Appl. Dyn. Syst. \textbf{20} (2021), 65--93.

\bibitem{Har78}
T.E. Harris, \emph{Additive set-valued Markov processes and pharical methods},
  Ann. Probab. \textbf{6} (1878), 355--378.

\bibitem{Har72}
T.E. Harris, \emph{{Nearest-neighbor Markov interaction processes on
  multidimensional lattices}}, Adv. Math. \textbf{9} (1972), 66--89.

\bibitem{Jo00b}
K.~Johansson, \emph{Shape fluctuations and random matrices}, Comm. Math. Phys.
  \textbf{209} (2000), 437--476.

\bibitem{Jo03b}
K.~Johansson, \emph{Discrete polynuclear growth and determinantal processes},
  Comm. Math. Phys. \textbf{242} (2003), 277--329.

\bibitem{Jo03}
K.~Johansson, \emph{The arctic circle boundary and the {Airy} process}, Ann.
  Probab. \textbf{33} (2005), 1--30.

\bibitem{KPZ86}
M.~Kardar, G.~Parisi, and Y.Z. Zhang, \emph{Dynamic scaling of growing
  interfaces}, Phys. Rev. Lett. \textbf{56} (1986), 889--892.

\bibitem{Kle20}
A.~Klenke, \emph{Probability theory: A comprehensive course}, 2014 2nd ed.,
  Universitext, Springer Nature, Netherlands, 2020.

\bibitem{MQR17}
K.~Matetski, J.~Quastel, and D.~Remenik, \emph{{The KPZ fixed point}}, Acta
  Math. \textbf{227} (2021), 115--203.

\bibitem{N17}
P.~Nejjar, \emph{{Transition to shocks in TASEP and decoupling of last passage
  times}}, ALEA, Lat. Am. J. Probab. Math. Stat. \textbf{15} (2018),
  1311--1334.

\bibitem{Nej20}
P.~Nejjar, \emph{{KPZ statistics of second class particles in ASEP via
  mixing}}, Commun. Math. Phys. \textbf{378} (2020), 601--623.

\bibitem{Nej21}
P.~Nejjar, \emph{{$\mathrm{GUE}\times \mathrm{GUE}$ limit law at hard shocks in
  ASEP}}, Ann. Appl. Probab. \textbf{31} (2021), 321 -- 350.

\bibitem{QR18}
J.~Quastel and M.~Rahman, \emph{{TASEP fluctuations with soft-shock initial
  data}}, Annales Henri Lebesgue \textbf{3} (2020), 999--1021.

\bibitem{QR13b}
J.~Quastel and D.~Remenik, \emph{Supremum of the {Airy}$_2$ process minus a
  parabola on a half line}, J. Stat. Phys. \textbf{150} (2013), 442--456.

\bibitem{QR13}
J.~Quastel and D.~Remenik, \emph{Airy processes and variational problems},
  Topics in Percolative and Disordered Systems (New York, NY) (A.F.
  Ram{\'i}rez, G.~Ben Arous, P.A. Ferrari, C.M. Newman, V.~Sidoravicius, and
  M.E. Vares, eds.), Springer New York, 2014, pp.~121--171.

\bibitem{QR16}
J.~Quastel and D.~Remenik, \emph{{How flat is flat in a random interface
  growth?}}, Trans. Amer. Math. Soc. \textbf{371} (2019), 6047--6085.

\bibitem{Sas05}
T.~Sasamoto, \emph{Spatial correlations of the {1D KPZ} surface on a flat
  substrate}, J. Phys. A \textbf{38} (2005), L549--L556.

\bibitem{Sep98c}
T.~Sepp{\"a}l{\"a}inen, \emph{Coupling the totally asymmetric simple exclusion
  process with a moving interface}, Markov Proc. Rel. Fields \textbf{4 no.4}
  (1998), 593--628.

\bibitem{Spi70}
F.~Spitzer, \emph{{Interaction of Markov processes}}, Adv. Math. \textbf{5}
  (1970), 246--290.

\bibitem{Whi80}
W.~Whitt, \emph{{Some useful functions for functional limit theorems}}, Math.
  Oper. Res. \textbf{5} (1980), 67--85.

\end{thebibliography}

\end{document}